\documentclass[7pt]{article}

\usepackage{amsfonts,amsmath,amsthm,amssymb}
\usepackage{bm}
\usepackage[colorlinks, citecolor=blue, urlcolor=blue, backref=page]{hyperref}

\usepackage{geometry}
\usepackage[utf8]{inputenc}
\usepackage[T1]{fontenc}

\usepackage{enumitem}
\usepackage{hyperref}
\usepackage{url}

\renewcommand{\theenumi}{{\rm(\roman{enumi})}}

\makeatletter
\renewcommand*{\@fnsymbol}[1]{\ensuremath{\ifcase#1\or *\or \dagger\or \ddagger\or
   \mathsection\or \mathparagraph\or \|\or **\or \dagger\dagger
   \or \ddagger\ddagger \else\@ctrerr\fi}}
\makeatother

\newtheorem{theorem}{Theorem}[section]

\newtheorem{lemma}[theorem]{Lemma}

\newtheorem{assumption}[theorem]{Assumption}

\newtheorem*{notation}{Notation}

\theoremstyle{definition}

\newtheorem{example}[theorem]{Example}
\newtheorem{remark}[theorem]{\textbf{Remark}}
\numberwithin{equation}{section}

\DeclareMathOperator{\argmin}{argmin}

\numberwithin{equation}{section}
\numberwithin{theorem}{section}
\def\@fnsymbol#1{\ensuremath{\ifcase#1\or *\or *\or *\ or \mathparagraph\or \|\or **\or \dagger\dagger
   \or \ddagger\ddagger \else\@ctrerr\fi}}

\allowdisplaybreaks

\begin{document}
\title{On the Mean-Field limit of diffusive games through the master equation: $L^{\infty}$ estimates and extreme value behavior.}

\author{Erhan Bayraktar\thanks{Department of Mathematics, University of Michigan, \url{erhan@umich.edu}.} \thanks{Funded in part by the NSF though DMS-2106556 and by the Susan M. Smith Chair.}
\and
Nikolaos Kolliopoulos\thanks{Department of Mathematics, University of Michigan, \url{nkolliop@umich.edu}.} 
\thanks{Department of Mathematics and Statistics, University of Cyprus, \url{Kolliopoulos.Nikolaos@ucy.ac.cy}.} 
\thanks{Funded by the NSF through DMS-2406232.}
}
\maketitle

\begin{abstract}
We consider an $N$-player game where the states of the players evolve with time as Stochastic Differential Equations (SDEs) with interaction only in the drift terms. Each player controls the drift of the SDE satisfied by her state process, aiming to minimize the expected value of a cost that depends on the paths of the player's state and the empirical measure of the states of all the players until a terminal time. When $N \to \infty$, previous works have established Central Limit Theorems and Large Deviation Principles for the state processes when the game is in Nash Equilibrium (the Nash states), by using the Master Equation to construct approximations of those processes that evolve with time as SDEs with classical Mean-Field interaction. Staying in this framework, we improve an existing $L^{1}$ estimate for the total error of approximating all the Nash states to an $L^{\infty}$ one, and we also establish the $N \to \infty$ asymptotic behavior of the upper order statistics of the Nash states. The latter initiates the development of an Extreme Value Theory for Stochastic Differential Games. 
\end{abstract}

\section{Introduction}
In this paper we establish an $L^{\infty}$ estimate for the error of an efficient approximation of a large Stochastic Differential Game under Nash equilibrium, and also a fundamental result of Extreme Value Theory for the states of the players when their number grows towards infinity. We work with a game of $N$ players on some filtered probability space $(\Omega, \mathcal{F}, (\mathcal{F}_t)_{t \in [0, T]}, \mathbb{P})$, where for each $i~\in~\{1, 2, \ldots, N\}$, the evolution in time of the $i$-th player's state is modeled as a diffusive process $X^{i, N, v^N} = (X_t^{i, N, v^N})_{t \in [0, T]}$ that satisfies
\begin{align}\label{system}
X_t^{i, N, v^N} = X_0^i + \int_0^t b\left(X_s^{i, N, v^N}, \mu_s^{N, v^N}, v_s^{i, N} \right)ds + \sigma W_t^i, \qquad t \geq 0,
\end{align}
where $\mu_t^{N, v^N}$ denotes the empirical measure of the $N$ players:
\begin{equation}\label{cost}
\mu_t^{N, v^N} = \frac{1}{N}\sum_{\ell = 1}^N\delta_{X_t^{\ell, N, v^N}},
\end{equation}
and $v^{i, N} = (v_t^{i, N})_{t \in [0, T]}$ is an $\mathcal{F}_t$-adapted process which represents the $i$-th player's strategy and takes values in some set $V \subset \mathbb{R}$ - we write $\mathcal{V}_N$ for the set from which each $v^{i,N}$ is picked. In the above, $(X_0^i)_{i \in \mathbb{N}}$ and $(W^i)_{i \in \mathbb{N}}$ are independent sequences of independent $\mathcal{F}_0$-measurable random variables with law $\mu_0 \in \mathcal{P}(\mathbb{R})$ and independent $\mathcal{F}_t$-adapted standard Brownian motions respectively, while $b: \mathbb{R} \times \mathcal{P}(\mathbb{R}) \times V \mapsto \mathbb{R}$ is a sufficiently regular function and $\sigma$ is a positive constant. Here, $\mathcal{P}(E)$ denotes the set of probability measures on a Banach space equipped with the 1-Wasserstein metric:
\begin{equation*}
\mathcal{W}_1(m_1, m_2) = \inf_{\substack{m_{1,2}(\cdot, E) = m_1 \\ m_{1,2}(E, \cdot) = m_2}}\int_{E \times E}\Vert x - y\Vert_{E}m_{1,2}(dx, dy)
\end{equation*}
In both $X_t^{i, N, v^N}$ and $\mu_t^{N, v^N}$, the superscript $v^N$ denotes the dependence of the states and the empirical measure on the vector of strategy processes $v^N = (v^{1, N}, v^{2, N}, \ldots, v^{N, N})  \in \mathcal{V}_N^N$. In this game, the $i$-th player wants to pick $v^{i, N}$ in a way that minimizes her expected cost:
\begin{align}\label{objectivequantity}
J^{i, N}\left(v^{N}\right) = \mathbb{E}\left[\int_0^Tf\left(X_s^{i, N, v^N}, \mu_s^{N, v^N}, v_s^{i, N}\right)ds + g\left(X_T^{i, N, v^N}, \mu_T^{N, v^N}\right)\right]
\end{align}
where $f: \mathbb{R} \times \mathcal{P}(\mathbb{R}) \times V \mapsto \mathbb{R}$ and $g: \mathbb{R} \times \mathcal{P}(\mathbb{R}) \mapsto \mathbb{R}$ are the running and terminal cost functions respectively. Then, two problems that naturally occur are: (a) the search of Nash equilibria, i.e. choices $v^{N, *} = (v^{1, N, *}, v^{2, N, *}, \ldots, v^{N, N, *})$ for $v^{N}$ that satisfy:
\begin{align}
&J^{i, N}\left(v^{N,*}\right) \leq J^{i, N}\left(v^{1, N,*}, v^{2, N,*}, \ldots, v^{i-1, N,*}, v^{i, N}, v^{i+1, N,*}, \ldots, v^{N, N,*}\right)
\end{align}
for any $i \in \{1, 2, \ldots, N\}$ and any process $v^{i, N} \in \mathcal{V}_N$; (b) the analysis of the system \eqref{system} under a Nash equilibrium $v^{N, *}$ where the $i$-th player picks $v^{i, N} = v^{i, N, *}$ for each $i \in \{1, 2, \ldots, N\}$, including the study of this regime as $N \to \infty$ to understand player dynamics in large population games. The study of a systemic risk model that falls into the above framework can be found in \cite{CFS15}, while variations of our setup arise in other financial contexts as well, see e.g \cite{LASO20, Carmona2020ApplicationsOM, HUZA, FU23, LAZA, ZA24}. In this work we restrict to the case where $v^{N,*}$ is a closed-loop Nash equilibrium, in which case each $v^{i,N,*}$ is picked from a set $\mathcal{V}_N$ of processes that are deterministic functions of the states of the players, i.e:
$$\mathcal{V}_N = \{v^{i,N}: v_t^{i,N} = h_t^{i,N}(X_t^{1, N, v^N}, X_t^{2, N, v^N}, \ldots, X_t^{N, N, v^N}) \quad \text{for} \quad h_t^{i,N}: \mathbb{R}^N \mapsto V, \quad \forall t\in [0, T]\}.$$

Sufficient conditions for the existence of a closed-loop Nash equilibrium $v^{N,*}$ for the Stochastic Differential Game \eqref{system} - \eqref{cost} can be found in \cite{CDLL19}, in which case each $v^{i,N,*}$ admits the form
\begin{align}\label{Nplayeropt}
v_t^{i, N, *} = \hat{v}\left(X_t^{i, N, v^{N, *}}, \mu_t^{N, v^{N, *}}, U_{x_i}^{i, N}\left(t, X_t^{1, N, v^{N, *}}, X_t^{2, N, v^{N, *}}, \ldots, X_t^{N, N, v^{N, *}}\right)\right),
\end{align}
where for any $x \in \mathbb{R}$, $y \in \mathbb{R}$ and $\mu \in \mathcal{P}(\mathbb{R})$ we write
\begin{align}
\hat{v}(x, \mu, y) = \displaystyle{\argmin_{v \in V}}\left\{b(x, \mu, v)y + f(x, \mu, v)\right\},
\end{align}
and $U^{1, N}, U^{2, N}, \ldots, U^{N,N}$ are the value functions of the $N$ players which satisfy the system of nonlinear PDEs
\begin{align}\label{sysPDE}
0 = U_t^{i,N}(t, \mathbf{x}) &+ \displaystyle{\min_{v \in V}}\left\{b\left(x_i, \mu_{\mathbf{x}}^N, v\right)U_{x_i}^{i,N}(t, \mathbf{x}) + f\left(x_i, \mu_{\mathbf{x}}^N, v\right)\right\} \nonumber \\
&+ \sum_{\substack{j = 1 \\ j \neq i}}^NU_{x_{j}}^{i,N}(t, \mathbf{x}) \cdot b\left(x_{j}, \mu_{\mathbf{x}}^N, \hat{v}\left(x_{j}, \mu_{\mathbf{x}}^N, U_{x_{j}}^{j, N}(t, \mathbf{x})\right)\right) \nonumber \\
&+ \frac{\sigma^2}{2}\sum_{j = 1}^NU_{x_{j}x_{j}}^{i, N}(t, \mathbf{x}),
\end{align}
along with terminal conditions 
\begin{equation}\label{terminalconditions1}
U^{i,N}(T, \mathbf{x}) = g\left(x_i, \mu_{\mathbf{x}}^N\right),
\end{equation}
in which $\mu_{\mathbf{x}}^N := \frac{1}{N}\sum_{\ell = 1}^N\delta_{x_{\ell}}$ for $\mathbf{x} = (x_1, x_2, \ldots, x_N)$. Thus, the evolution of the states of the $N$ players under Nash equilibrium is given by
\begin{align}\label{systemNash}
X_t^{i, N, v^{N,*}} = X_0^i + \int_0^t\hat{b}\left(X_s^{i, N, v^{N,*}}, \mu_s^{N, v^{N,*}}, U_{x_i}^{i, N}\left(s, X_s^{1, N, v^{N, *}}, X_s^{2, N, v^{N, *}}, \ldots, X_s^{N, N, v^{N, *}}\right)\right)ds + \sigma W_t^i
\end{align}
for $t \geq 0$, where we write: 
\begin{align}\label{minimized}
\hat{b}\left(x, m, y\right) =  b\left(x, m, \hat{v}\left(x, m, y\right)\right).
\end{align}

It is generally expected that the $N \to \infty$ asymptotics of $N$-player Stochastic Differential Games are governed by a property known as Propagation of Chaos, as happens with uncontrolled diffusions with standard Mean-Field interaction \cite{MR221595,MR1108185,MR968996}. In the game framework, this property is summarized as follows:
\begin{enumerate}
    \item The empirical measure $\displaystyle{\mu^{N,v^{N,*}} := \frac{1}{N}}\sum_{\ell = 1}^N\delta_{X^{\ell, N, v^{N,*}}}$ on $C[0, T]$ satisfies
    \begin{equation}\label{propchaos2}
    \mu^{N,v^{N,*}} \to \mathcal{L}(X^{v^{*}})
    \end{equation} 
    weakly as $N \to \infty$, where $X^{v} = (X_t^{v})_{t \in [0, T]}$ denotes the state process of a representative player in the limiting regime as $N \to \infty$ - which is frequently called "Mean-Field Game" - under an $\mathcal{F}_t$-adapted strategy process:
$$v = (v_t)_{t \in [0, T]} \in \mathcal{V}_{\infty} := \{v: v_t = h_t\left(X_t^v\right) \quad \text{for} \quad h_t: \mathbb{R} \mapsto V, \quad \forall t \in [0,T]\},$$
$\mathcal{L}(X^{v})$ being the law of the $C[0,T]$-valued random path $X^v = (X_t^{v})_{t \in [0,T]}$ and $v^{*} = (v_t^{*})_{t \in [0, T]}$ being the optimal strategy process for the representative player.
    \item For any fixed $k \in \mathbb{N}$ it holds that: 
    \begin{equation}\label{propchaos1}
    \left(X_t^{1,N,v^{N,*}}, X_t^{2,N,v^{N,*}}, \ldots, X_t^{k,N,v^{N,*}}\right) \to \left(X_t^{1,v^{*}}, X_t^{2,v^{*}}, \ldots, X_t^{k,v^{*}}\right)
    \end{equation} 
    in distribution as $N \to \infty$, with $X^{1,v^*}, X^{2,v^*}, \ldots, X^{k,v^*}$ being independent copies of $X^{v^*}$.
\end{enumerate}
For a standard Brownian motion $(W_t)_{t \in [0, T]}$ and a random variable $X_0$ with law $\mu_0$, the evolution of $(X^{v^{*}}, v^{*})$ is given by
\begin{align}\label{repagent0}
X_t^{v^{*}} &= X_0 + \int_0^tb\left(X_s^{v^{*}}, \mathcal{L}(X_s^{v^{*}}) , v_s^{*}\right)ds + \sigma W_t \nonumber \\
v^{*} &= \argmin_{v \in \mathcal{V}_{\infty}}\mathbb{E}\left[\int_0^Tf\left(X_s^{v}, \mathcal{L}(X_s^{v^{*}}), v_s\right)ds + g\left(X_T^{v}, \mathcal{L}(X_T^{v^{*}})\right)\right].
\end{align}
The derivation of Propagation of Chaos for a class of Stochastic Differential Games can be found in \cite{CDLL19}. As shown in \cite{LA23} through a simple example of uncontrolled diffusions, \eqref{propchaos1} is generally not true for $k = N$ in any sense and, for that reason, the asymptotic relation
\begin{equation}\label{HpropChaos}
H^N\left(X^{1,N,v^{N,*}}, X^{2,N,v^{N,*}}, \ldots, X^{k,N,v^{N,*}}\right) \sim H^N\left(X^{1,v^{*}}, X^{2,v^{*}}, \ldots, X^{k,v^{*}}\right) \,\,\,\, \text{as} \,\,\,\, N \to \infty
\end{equation}
holds in some sense only for a very limited class of functions $H^N$ with $N$ arguments; \eqref{propchaos2} is essentially a Law of Large Numbers for the interacting variables $X^{i,N,v^{N,*}}$, with the corresponding Central Limit Theorems and Large Deviation Principles being also available in the literature \cite{Delarue2018FromTM, DLR20}, and all these results can be seen as \eqref{HpropChaos} for appropriate choices of $H^N$. For example, the Law of Large Numbers \eqref{propchaos2} is precisely \eqref{HpropChaos} for $H^N(x_1,x_2,\ldots,x_N) = \frac{1}{N}\sum_{i=1}^N\delta_{x_i}$, since the left-hand side of the latter is precisely $\mu^{N,v^{N,*}}$ and its right-hand side converges to $\mathcal{L}(X^{v^*})$ as $N \to \infty$ by the Law of Large Numbers for independent and identically distributed (i.i.d.) random variables.

It is also shown in \cite{CDLL19} that the $N \to \infty$ limit of our setup is governed by the Master equation, which is the following nonlinear $PDE$ on $[0, T] \times \mathbb{R} \times \mathcal{P}(\mathbb{R})$:

\begin{align}\label{master}
0 = U_t(t, x, m) &+  \displaystyle{\min_{v \in V}}\left\{b\left(x, m, v\right)U_{x}(t, x, m) + f\left(x, m, v\right)\right\} \nonumber \\
&+ \frac{\sigma^2}{2}U_{xx}(t, x, m) +\int_{\mathbb{R}}U_{m}(t, x, m, z_1) \times b\left(z_1, m, \hat{v}\left(z_1, m, U_{x}(t, z_1, m)\right)\right)m(dz_1) \nonumber \\
&+ \frac{\sigma^2}{2}\int_{\mathbb{R}}U_{mz}(t, x, m, z_1)m(dz_1)
\end{align}
under the terminal condition $U(T, x, m) = g(x,m)$, with $U_m$ being our notation for the derivative with respect to $m \in \mathcal{P}(\mathbb{R})$ which is defined in e.g \cite{CD1,CDLL19,DLR20} (all notations are introduced in Section~2). In some sense, \eqref{master} plays the role of the $N \to \infty$ limit of \eqref{sysPDE}, and the analogue of \eqref{Nplayeropt} in the limit that gives the optimal strategy process $v^*$ is:
\begin{align}\label{optcontrollimit}
v_t^{*} = \hat{v}\left(X_t^{v^{*}}, \mathcal{L}(X_t^{v^{*}}), U_{x}\left(t, X_t^{v^{*}}, \mathcal{L}(X_t^{v^{*}})\right)\right),
\end{align}
so the controlled McKean-Vlasov SDE \eqref{repagent0} that describes the evolution of the representative player in the $N \to \infty$ limit of the game becomes:
\begin{align}\label{repagent}
X_t^{v^{*}} &= X_0 + \int_0^t\hat{b}\left(X_s^{v^{*}}, \mathcal{L}(X_s^{v^{*}}), U_{x}\left(s, X_s^{v^{*}}, \mathcal{L}(X_s^{v^{*}})\right)\right)ds + \sigma W_t.
\end{align}
To obtain a Central Limit Theorem and a Large Deviations Principle for the Nash states $X_t^{i, N, v^{N,*}}$ in \cite{Delarue2018FromTM} and \cite{DLR20} respectively, the authors utilize that when $N$ is large, the system \eqref{systemNash} satisfied by the Nash states $X^{i,N,v^{N,*}}$ is sufficiently close to the following simpler system:
\begin{equation}
\begin{aligned}\label{approxsys1}
\bar{X}_t^{i,N} &= X_0^i + \int_0^t\hat{b}\left(\bar{X}_s^{i,N}, \bar{\mu}_s^{N}, U_{x}\left(s, \bar{X}_s^{i,N}, \bar{\mu}_s^{N}\right) \right)ds + \sigma W_t^i, \\
\bar{\mu}_t^{N} &= \frac{1}{N}\sum_{\ell = 1}^N\delta_{\bar{X}_t^{\ell,N}} 
\end{aligned}
\end{equation}
for $t \in [0,T]$. Specifically, \eqref{approxsys1} describes a standard Mean-Field System that is easier to handle compared to \eqref{systemNash}, and its components $\bar{X}_t^{i,N}$ constitute a significantly better approximation of the Nash states $X^{i,N,v^{N,*}}$ compared to the i.i.d. processes $X^{i,v^{*}}$. The latter allowed the authors of \cite{Delarue2018FromTM, DLR20} to obtain the strong $L^1(\Omega)$-estimate
\begin{align}\label{expestimate}
&\Bigg\Vert\sum_{i=1}^N\Bigg\{\hat{b}\left(X_t^{i, N, v^{N,*}}, \mu_t^{N, v^{N,*}}, U_{x_i}^{i, N}\left(t, X_t^{1, N, v^{N, *}}, X_t^{2, N, v^{N, *}}, \ldots, X_t^{N, N, v^{N, *}}\right)\right) \nonumber \\
& \qquad\qquad\qquad\qquad\qquad - \hat{b}\left(X_t^{i, N, v^{N,*}}, \mu_t^{N, v^{N,*}}, U_{x}\left(t, X_t^{i, N, v^{N,*}}, \mu_t^{N, v^{N,*}}\right) \right)\Bigg\}^2\Bigg\Vert_{L^1(\Omega)} \leq \frac{C}{N},
\end{align}
which is a $\mathcal{O}(N^{-1})$ bound for the relative entropy that arises when we apply Girsanov's theorem to transform the law of $(X^{1,N,v^{N,*}}, X^{2,N,v^{N,*}}, \ldots, X^{N,N,v^{N,*}})$ into that of $(\bar{X}_t^{1,N}, \bar{X}_t^{2,N}, \ldots, \bar{X}_t^{N,N})$, and as $N \to \infty$, it is the fast decay of this entropy that allows for existing asymptotic results for the standard Mean-Field System of the processes $\bar{X}^{i,N}$ to be extended to the Nash states $X^{i,N,v^{N,*}}$, including a Central Limit Theorem and a Large Deviations Principle. More generally, using Pinsker's inequality and the entropy bound \eqref{expestimate}, the derivation of the asymptotic relation \eqref{HpropChaos} reduces to showing that
\begin{equation}\label{HpropChaos2}
H^N\left(\bar{X}_t^{1,N}, \bar{X}_t^{2,N}, \ldots, \bar{X}_t^{N,N}\right) \sim H^N\left(X_t^{1,v^{*}}, X_t^{2,v^{*}}, \ldots, X_t^{k,v^{*}}\right) \,\,\,\, \text{as} \,\,\,\, N \to \infty
\end{equation}
whenever $H^{N}$ is simply measurable and bounded. Then, since \eqref{repagent} is the Mckean-Vlasov SDE that captures the $N \to \infty$ behavior of both systems \eqref{systemNash} and \eqref{approxsys1}, the latter is actually the reduction of the establishment of an arbitrary propagation of chaos property for the Nash states to obtaining the same property for a standard system of uncontrolled diffusions with Mean-Field interaction in the drifts.

Our first contribution to the above framework is the improvement of the $L^1(\Omega)$-estimate \eqref{expestimate} where the $L^1(\Omega)$ norm is replaced by the $L^{\infty}(\Omega)$ norm. A potential application of this result could be the numerical simulation of large Stochastic Differential Games under Nash equilibrium, since it tells us that if we simulate the classical Mean-Field System \eqref{approxsys1} instead, the additional error will certainly not exceed a deterministic threshold of order $\mathcal{O}(N^{-1})$. It should be mentioned that it could be possible to make the total error negligible in that case, as there is a growing literature on the development of new methods for the numerical simulation of Mean-Field Systems like \eqref{approxsys1} (see e.g \cite{RBM0,RBM1,RBM2,RBM3} for the Random Batch Method and \cite{RBM4,RBM5} for two applications). To improve \eqref{expestimate}, we obtain BSDE estimates similar to those in \cite[Proof of Theorem 4.2]{Delarue2018FromTM}, using, however, the derivatives of the solutions to the PDEs \eqref{sysPDE} and \eqref{master} instead of the solutions themselves. The required assumptions include boundedness of the derivatives of the functions $U$ and $U^{i,N}$ that are hard to verify, but we provide an example which shows that boundedness of the derivatives of the terminal condition $g(x, m)$ is sufficient in some simple settings.

Our second result in this work is to extend the list of asymptotic results that have already been established for the Nash states $X^{i,N,v^{N,*}}$ as $N \to \infty$, which contains a Central Limit Theorem and a Large Deviations Principle, by also obtaining a fundamental result of Extreme Value Theory. The probabilistic results of classical Extreme Value Theory (EVT) concern independent random observations $Z^1, Z^2, \ldots, Z^N$ with some common cumulative distribution function $F: \mathbb{R} \mapsto [0, 1]$, and the most fundamental question is whether there exist deterministic sequences $(a^N)_{N \in \mathbb{N}}$ and $(b^N)_{N \in \mathbb{N}}$ such that the normalized maximum
\begin{equation*}
\max_{i \leq N}\frac{Z^i - b^N}{a^N} = \frac{Z^{i_1^N} - b^N}{a^N}
\end{equation*}
converges weakly as $N \to \infty$, with the limiting distribution being non-degenerate. When this is the case, the limiting cumulative distribution function admits the form $G_{\gamma}(ax + b)$ for $a, b, \gamma \in \mathbb{R}$, where we define:
\begin{equation*}
G_{\gamma}(x) := e^{-(1 + \gamma x)^{-\frac{1}{\gamma}}} \quad \text{for} \quad 1 + \gamma x > 0.
\end{equation*}
It is then said that the distribution of the variables $Z^i$ belongs to the domain of attraction of the extreme value distribution $G_{\gamma}(x)$, with the parameter $\gamma$ called "the extreme value index", and we typically have the joint convergence of multiple normalized upper order statistics. The latter means that if we sort our random variables as $Z^{i_1^N} \geq Z^{i_2^N} \geq \ldots \geq Z^{i_N^N}$, for any fixed $k \in \mathbb{N}$ we have
\begin{align}\label{evtresult0}
&\left(\frac{Z^{i_1^N} - b^N}{a^N}, \frac{Z^{i_2^N} - b^N}{a_t^N}, \ldots, \frac{Z^{i_k^N} - b^N}{a^N}\right) \nonumber \\
&\qquad\qquad\qquad\qquad \to \left(\frac{(E_1)^{-\gamma} - 1}{\gamma}, \frac{(E_1 + E_2)^{-\gamma} - 1}{\gamma}, \cdots, \frac{(E_1 + E_2 + \cdots + E_k)^{-\gamma} - 1}{\gamma}\right)
\end{align}
weakly as $N \to \infty$, where $E_1, E_2, \ldots, E_k$ are i.i.d. standard exponential random variables. Moreover, when the distribution of the variables $Z^i$ belongs to the domain of attraction of the extreme value distribution, the intermediate order statistics $Z^{i_{k_N}^N}$ with $k_N \to \infty$, $\frac{k_N}{N} \to 0$ as $N \to \infty$ are typically asymptotically normal, meaning that:
\begin{equation}\label{intermconv}
\frac{Z^{(k_N)} - \hat{b}^N}{\tilde{a}^N} \sim \mathcal{N}(0, 1) \quad \text{as} \quad N \to \infty,
\end{equation}
for some deterministic normalizing sequences $(\tilde{a}^N)_{N \in \mathbb{N}}$ and $(\hat{b}^N)_{N \in \mathbb{N}}$ which are different from $(a^N)_{N \in \mathbb{N}}$ and $(b^N)_{N \in \mathbb{N}}$. We refer to \cite[Chapters 1 and 2]{HF10} for a deeper study of probabilistic EVT. The main use of the above asymptotic results is: (a) the construction of consistent statistical estimators for $\gamma$, $a^N$, $\tilde{a}^N$, $b^N$ and $\hat{b}^N$ which are functions of multiple intermediate order statistics $Z^{i_{k_N}^N}$; (b) the estimation of tail probabilities $1 - F(z)$ for large $z$ after $\gamma$, $a^N$, $\tilde{a}^N$, $b^N$ and $\hat{b}^N$ have been estimated. Those techniques are particularly useful when $z > \max_{i \leq N}Z^i$, in which case the simple estimation $$1 - F(z) \approx \frac{1}{N}\sum_{i=1}^N\mathbf{1}_{\{Z^i > z\}}$$ is not effective due to the right-hand side being equal to $0$, which can lead to the incorrect conclusion that a rare but highly impactful event can never occur. We refer to \cite[Chapters 3 and 4]{HF10} for an introduction to statistical EVT, see also \cite{PI75, HI75, WE78} for the earliest works in this area, \cite{DED89, DR298, CP01, GDP02, GM02, GCF04, CG06, GDR08, CH09, DGR12, CGBD16, PVV17} for relevant developments over the last 40 years, and \cite{DF19, BBD19, AR20, HPN22, AEG23, GM23} for some more recent works with more contemporary methods. The weak convergence \eqref{evtresult0} has already been extended to a setup of $N$ uncontrolled diffusions $\{X^{1,N}, X^{2,N}, \ldots, X^{1,N}\}$ with a standard Mean-Field interaction in the drifts - i.e. interaction through the direct dependence of the drifts on the systemic empirical measure - where for some fixed $t \geq 0$ we have $Z^i = X_t^{i, N}$ for all $i \in \{1, 2, \ldots, N\}$, see \cite{KLZ22} for the $k=1$ case and \cite{KLZ23} for the extension to arbitrary fixed $k \in \mathbb{N}$. Here we further extend \eqref{evtresult0} to our diffusive game setup under Nash equilibrium, where for some fixed $t \geq 0$ we have $Z^i = X_t^{i, N, v^{N,*}}$ for $i \in \{1, 2, \ldots, N\}$. While this is expected to open the way for the prediction of extreme values in continuously evolving diffusive populations with a game structure, the extension of \eqref{intermconv} which concerns intermediate order statistics to our setup and the development of statistical methods are left for subsequent works.

To extend \eqref{evtresult0} to our diffusive game setup, we sort the Nash states at a time $t \in [0, T]$ as $$X_t^{i_1^N(t),N,v^{N,*}} \geq X_t^{i_2^N(t),N,v^{N,*}} \geq \ldots \geq X_t^{i_N^N(t),N,v^{N,*}},$$ with the ranks $i_j^N(t)$ depending on $t \geq 0$ since the Brownian motions $W^i$ are independent, and we show that for fixed $k \in \mathbb{N}$, the top $k$ order statistics $\{X_t^{i_j^N(t),N,v^{N,*}}: j \in \{1,2,\ldots,k\}\}$ satisfy:
    \begin{align}\label{jjj}
&\left(\frac{X_t^{i_1^N(t),N,v^{N,*}} - b_t^N}{a_t^N}, \frac{X_t^{i_2^N(t),N,v^{N,*}} - b_t^N}{a_t^N}, \ldots, \frac{X_t^{i_k^N(t),N,v^{N,*}} - b_t^N}{a_t^N}\right) \nonumber \\
&\qquad\qquad\qquad\qquad \to \left(\frac{(E_1)^{-\gamma_t} - 1}{\gamma_t}, \frac{(E_1 + E_2)^{-\gamma_t} - 1}{\gamma_t}, \cdots, \frac{(E_1 + E_2 + \cdots + E_k)^{-\gamma_t} - 1}{\gamma_t}\right)
\end{align}
weakly as $N \to \infty$ for an appropriate choice of the Extreme Value Index $\gamma_t \leq 0$ and the deterministic normalizing sequences $\{a_t^N\}_{N=1}^{\infty}$ and $\{b_t^N\}_{N=1}^{\infty}$, where $E_1, E_2, \ldots, E_k$ are i.i.d. standard exponential random variables. In view of classical Extreme Value Theory which concerns i.i.d. random variables, we pick $\{a_t^N\}_{N=1}^{\infty}$, $\{b_t^N\}_{N=1}^{\infty}$ and $\gamma_t$ such that  \eqref{jjj} holds when the Nash states $X_t^{i,N,v^{N,*}}$ are replaced by the i.i.d. random variables $X_t^{i,v^{*}}$, and then it suffices to derive \eqref{HpropChaos} when $H^N$ is an appropriate bounded function that captures the $N \to \infty$ behavior of the normalized upper order statistics of its arguments. In that case, we reduce \eqref{HpropChaos} to \eqref{HpropChaos2} by using \eqref{expestimate} as described earlier, and the latter is a propagation of chaos property for the upper order statistics of the processes $\bar{X}^{i,N}$ whose establishment is now sufficient for our result to hold. However, while propagation of chaos for the upper order statistics of diffusions with a certain type of Mean-Field interaction in their drifts has been obtained in the second author's previous work \cite{KLZ22, KLZ23}, the drift terms of the diffusive processes $\bar{X}^{i,N}$ in \eqref{approxsys1} depend on the empirical measure $\bar{\mu}_t^{N}$ in a more general way that is not covered by those papers. For this reason, we use a Taylor expansion in the space $\mathcal{P}(\mathbb{R})$ to linearize those drift functions with respect to their measure argument, which leads to the following approximation of the system \eqref{approxsys1} that is covered by the existing propagation of chaos results for upper order statistics:
\begin{equation}\label{approxsys2}
\begin{aligned}
&\tilde{X}_t^{i,N} = X_0^i + \int_0^t\hat{b}\left(\tilde{X}_s^{i,N}, \mathcal{L}(X_s^{v^{*}}), U_{x}\left(s, \tilde{X}_s^{i,N}, \mathcal{L}(X_s^{v^{*}})\right)\right)ds \\
&\qquad \qquad \quad + \int_0^t\int_{\mathbb{R}}\frac{\delta}{\delta m}\left\{\hat{b}\left(\tilde{X}_s^{i,N}, m, U_{x}\left(s, \tilde{X}_s^{i,N}, m\right)\right)\right\}(z_1) \Big|_{m = \mathcal{L}(X_s^{v^{*}})} \\
&\qquad \qquad \qquad \qquad \qquad \qquad \qquad \qquad \qquad \qquad \qquad \times (\tilde{\mu}_s^N - \mathcal{L}(X_s^{v^{*}}))(dz_1)ds + \sigma W_t^i, \\
&\,\,\,\,\, \tilde{\mu}_t^{N} = \frac{1}{N}\sum_{\ell = 1}^N\delta_{\tilde{X}_t^{\ell,N}},
\end{aligned}
\end{equation}
with $\frac{\delta }{\delta m}$ being our notation for a Frechet derivative with respect to $m \in \mathcal{P}(\mathbb{R})$. After that, we will see that the McKean-Vlasov SDE that governs the $N \to \infty$ behavior of the Mean-Field System \eqref{approxsys2} is again \eqref{repagent}, i.e the same as for the systems \eqref{approxsys1} and \eqref{systemNash} that \eqref{approxsys2} approximates. Therefore, controlling the relative entropy that arises by applying Girsanov's theorem to transform the law of $(\bar{X}_t^{1,N}, \bar{X}_t^{2,N}, \ldots, \bar{X}_t^{N,N})$ into that of $(\tilde{X}_t^{1,N}, \tilde{X}_t^{2,N}, \ldots, \tilde{X}_t^{N,N})$, we further reduce the establishment of our result to having a propagation of chaos property for upper order statistics that is provided by the existing literature. Finally, it should be noted that our drift linearization technique was also used in \cite[Step 3 of Section 6]{KLZ22} and \cite[Section 4]{KLZ23}, but it was performed in $\mathbb{R}$ instead of $\mathcal{P}(\mathbb{R})$ to eliminate a nonlinear dependence on a linear functional of the empirical measure.

\section{Assumptions and main results}
We begin with the most fundamental assumption of this paper:
\begin{assumption}\label{fundass}
For any $(x, y, m) \in \mathbb{R}^2 \times \mathcal{P}(\mathbb{R})$, the function:
\begin{equation*}
V \ni v \mapsto b(x, m, v)y + f(x, m, v)
\end{equation*}
admits a minimizer $v =\hat{v}\left(x, m, y\right)$.
\end{assumption}
\noindent We can now introduce the notations that will be used in everything that follows:
\begin{notation}
The following notations will be used throughout this paper:
\begin{enumerate}
\item We write $C$ for the arbitrary positive constant, which is deterministic, depends only on the parameters of the system \eqref{system} - \eqref{objectivequantity} (i.e the functions $b,f,g$, the volatility $\sigma$ and the distribution of the initial values $X_0^i$), and will generally change from line to line. 

\item Having assummed the existence of the minimizer $\hat{v}\left(x, m, y\right)$ in Assumption~\ref{fundass}, we denote:
\begin{align}\label{minimized2}
\hat{b}\left(x, m, y\right) =  b\left(x, m, \hat{v}\left(x, m, y\right)\right) \quad \text{and} \quad \hat{f}\left(x, m, y\right) =  f\left(x, m, \hat{v}\left(x, m, y\right)\right)
\end{align}
\item The differentiation of any function $q$ with respect to an $\mathbb{R}$-valued argument $w$ is denoted by $q_w$, where:
\begin{itemize}
    \item For $q$ being $\hat{b}$, $\hat{f}$, $\hat{v}$ or any of their derivatives we may have $w = x, y$ or $z_i$, which denote, respectively, the first, third and $(3+i)$-th argument (for $i \geq 1$ and only for derivatives of $q$ that are of order $\geq i$ with respect to the measure argument $m$).

    \item For $q$ being any derivative of $U$ we may have $w = t, x$ or $z_i$, which denote, respectively, the first, second and $(3+i)$-th argument (for $i \geq 1$ and only for derivatives of $q$ that are of order $\geq i$ with respect to the measure argument $m$).
\end{itemize}
The Frechet derivative of any function $q$ with respect to the measure argument $m$ is denoted by $\frac{\delta q}{\delta m}$, and we write $q_m$ for the corresponding intrinsic derivative. Hence, we will write for example $\frac{\delta q}{\delta m}(x, m, U_x(t, x, m), z_1)$ for the function that occurs when we plug $y = U_x(t, x, m)$ in $\frac{\delta q}{\delta m}(x, m, y, z_1)$, and $\frac{\delta}{\delta m}\{q(x, m, U_x(t, x, m))\}(z_1)$ for the total derivative of $q(x, m, U_x(t, x, m))$ with respect to $m$: $$\frac{\delta}{\delta m}\{q(x, m, U_x(t, x, m))\}(z_1) = \frac{\delta q}{\delta m}(x, m, U_x(t, x, m), z_1) + q_y(x, m, U_x(t, x, m))\frac{\delta U_x}{\delta m}(t, x, m, z_1),$$ which we can differentiate with respect to $z_1$ to obtain $$\{q(x, m, U_x(t, x, m))\}_m(z_1) = q_m(x, m, U_x(t, x, m), z_1) + q_y(x, m, U_x(t, x, m))U_{xm}(t, x, m, z_1).$$
\end{enumerate}
\end{notation}
We introduce now the conditions necessary for our results to hold, which highly overlap with those required for the establishment of a Central Limit Theorem and a Large Deviations Principle in \cite{Delarue2018FromTM, DLR20}. The first set of conditions is needed for both results we establish, and it is given below:
\begin{assumption}\label{ass0}
The following conditions are in force:
\begin{enumerate}
    \item The initial values $X_0^i$ are i.i.d. random variables, and their common law $\mu_0$ satisfies
    \begin{equation*}
    \int_{\mathbb{R}}x^{p'}\mu_0(dx) < \infty
    \end{equation*}
    for some $p' > 4$.

    \item The volatility $\sigma$ is assumed to be a positive constant.
    \item The classical partial derivatives $\hat{b}_x, \hat{b}_m, \hat{b}_y, \hat{f}_y$ and the second order classical partial derivative $\hat{b}_{yy}$ of the functions $\hat{b}\left(x, m, y\right)$ and $\hat{f}\left(x, m, y\right)$ exist and are bounded.

\item For each $N \in \mathbb{N}$, the system of PDEs \eqref{sysPDE} with the terminal conditions \eqref{terminalconditions1} admits classical solutions $U^{i,N}$, in the sense that each $U^{i,N}(t, \mathbf{x})$ is continuously differentiable in $t > 0$ and twice continuously differentiable in $\mathbf{x} = (x_1, x_2, \ldots, x_N) \in \mathbb{R}^N$. Finally, for each $N \in \mathbb{N}$, there exists a constant $K_N > 0$ such that
\begin{align*}
\left|U^{i,N}_{x_j}\left(t, \mathbf{x}\right)\right| &\leq K_N\left(1 + \sqrt{x_1^2+x_2^2+\ldots+x_N^2}\right) \quad \text{and} \nonumber \\
\left|U^{i,N}\left(t, \mathbf{x}\right)\right| &\leq K_N\left(1 + x_1^2+x_2^2+\ldots+x_N^2\right)
\end{align*}
for all $i,j \in \{1,2,\ldots,N\}$, all $t \in [0, T]$ and all $\mathbf{x} = (x_1, x_2, \ldots, x_N) \in \mathbb{R}^N$.

\item The Master equation \eqref{master} with the terminal condition $U(T,x,m) = g(x, m)$ admits a solution $U = U(t,x,m)$, whose partial derivatives $U_t$, $U_x$, $U_m$, $U_{xx}$, $U_{xm}$, $U_{mx}$, $U_{mm}$, $U_{mz_1}$ exist and are continuous under the product topologies on the spaces they are defined, e.g $[0,T] \times \mathbb{R} \times \mathcal{P}(\mathbb{R})$ for $U_{xx}$ and $[0,T] \times \mathbb{R} \times \mathcal{P}(\mathbb{R}) \times \mathbb{R}^2$ for $U_{mm}$. Finally, the derivatives $U_x$, $U_m$, $U_{xx}$, $U_{xm} = U_{mx}$ and $U_{mm}$ are bounded, and the derivative $U_{xmm}$ exists and is also bounded.
\end{enumerate}
\end{assumption}
\begin{remark}\label{werecallthis}
We make the following important notes:
\begin{enumerate}
    \item Assumption~\ref{ass0} is essentially \cite[Assumptions A and B]{Delarue2018FromTM} for $p^*=1$, which are required for recalling results and computations from that paper, accompanied by the existence and boundedness of $\hat{b}_{yy}$ and $U_{xmm}$. There is also a simplification in the sense that we ask for $\hat{b}$ and $\hat{f}$ to have bounded classical derivatives instead of simply being Lipschitz; for one of our results (Theorem~\ref{mainresult} below), our proof seems to work when $\hat{b}$ is Lipschitz in $x$ and $w$ and when $\hat{f}$ is Lipschitz in $y$, but e.g $\hat{b}_y$ and $\hat{b}_{yy}$ have to be classical derivatives that exist everywhere. 

    \item The conditions in Assumption~\ref{ass0} that are hard to verify are the existence of sufficiently regular solutions $U^{i,N}$ to \eqref{sysPDE} and a sufficiently regular solution $U$ to the Master equation \eqref{master}. Assumptions A and B in \cite{Delarue2018FromTM} are accompanied by references that verify all these conditions for a wide range of setups, except from the existence and boundedness of $U_{xmm}$: for \eqref{sysPDE}, the results of \cite{bookqlpde} are applicable to Stochastic Differential Games where the Hamiltonian $H(x,m,y) := \hat{b}(x,m,y)y + \hat{f}(x,m,y)$ is globally Lipschitz in $y$, while \cite{book2syspde} and \cite[Section 6.3.1]{CD2} contain results for setups where $H(x,m,y)$ is quadratic in $y$; for \eqref{master}, results for the existence of a sufficently regular solution can also be found in \cite{CD2}. For the existence and boundedness of $U_{xmm}$ and of other derivatives of $U$ in Assumptions \ref{ass0-} and \ref{ass0'} below that are not covered by the above references, we highly expect that having a regular enough terminal cost function $g(x, m)$ is sufficient. Possible ways for showing that could include the adaptation to higher order derivatives of the computations performed in \cite{CD2} for obtaining bounded second order derivatives (which we expect to be long and tedious but not requiring any novel ideas), and the extension of the following classical bootstraping method to PDEs on $[0, T] \times \mathbb{R} \times \mathcal{P}(\mathbb{R})$ and high-dimensional quasilinear parabolic systems: if $U$ is a classical solution to $$U_t + U_{xx} + q(U_x)U_x = 0$$
    with $U(T, \cdot)$ and $q(\cdot)$ being sufficiently regular, $U$ solves also the linear parabolic PDE $$U_t + U_{xx} + qU_x = 0$$ with $q = q(U_x)$ seen as a known $C^{1,1}$ function. From there, standard regularity theory of linear parabolic PDEs gives $U \in C^{1,3}$ when $U(T,\cdot) \in C^3$. 

    \item For our setup, the Central Limit Theorem in \cite{Delarue2018FromTM} and the Large Deviations Principe in \cite{DLR20} are also obtained when \cite[Assumption B]{Delarue2018FromTM} is replaced by \cite[Assumption B']{Delarue2018FromTM}, and a class of games for which this assumption is in force is also provided. In our paper, this amendment translates to imposing boundedness on $U$ and uniform boundedness (in $i$ and $N$) on $U^{i,N}$, but allowing for $\hat{f}_y$ to grow linearly in $y$. As \cite[Assumption B]{Delarue2018FromTM} is only used in the proof of Theorem~\ref{mainresult} below when \eqref{expestimate} is recalled, and since the latter also holds under \cite[Assumption B']{Delarue2018FromTM} for large $N$, we deduce that Theorem~\ref{mainresult} below (which gives the $N \to \infty$ behavior of the upper order statistics of the Nash states) holds under this modification as well.

    \item For a popular class of linear-quadratic games, $g(x,m)$ is a quadratic function and thus the first order derivatives $U_x$ and $U_m$ are not bounded. However, the authors in \cite{Delarue2018FromTM, DLR20} show that the Central Limit Theorem and the Large Deviations Principle they obtain can still hold in these setups, by studying the systemic risk model from \cite{CFS15} (see e.g \cite[Section 6]{DLR20}). This is also the case for our result on the $N \to \infty$ behavior of the upper order statistics of the Nash states which is given in Theorem~\ref{mainresult} below. The reason is that the proof uses the boundedness of $U_x$ and $U_m$ only for recalling \eqref{expestimate} from \cite{Delarue2018FromTM, DLR20}, which means that this boundedness of $U_x$ and $U_m$ can be ignored if \eqref{expestimate} can be verified directly. This is exactly what we do in Example~\ref{finexample} below, where we also consider the systemic risk model from \cite{CFS15} 
\end{enumerate}
\end{remark}

\noindent The next set of conditions is only required for our result on the weak convergence of the normalized upper order statistics of the Nash states; it contains the conditions required in addition to Assumption~\ref{ass0} for the proof we will give for that result.

\begin{assumption}\label{ass0-}
The following conditions are in force:
\begin{enumerate}
    \item For some $\kappa_0 > 0$, the law $\mu_0$ of the independent initial values $X_0^i$ satisfies
    \begin{equation*}
    \int_{\mathbb{R}}e^{\kappa_0 x^2}\mu_0(dx) < \infty,
    \end{equation*}
    which implies that it also has finite moments (satisfying, in particular, (i) of Assumption~\ref{ass0}).
    \item The derivatives $\hat{b}_{yx}$, $\hat{b}_{ym}$, $\hat{b}_{mm}$, $\frac{\delta U_x}{\delta m}$, $(\frac{\delta U_{x}}{\delta m})_x$, $(\frac{\delta \hat{b}}{\delta m})_x$, $(\frac{\delta \hat{b}}{\delta m})_y$ and $(\frac{\delta \hat{b}}{\delta m})_{yz_1}$ exist and are bounded. 
\end{enumerate}
\end{assumption}

\noindent Finally, we have the following conditions that are required in addition to Assumption~\ref{ass0} for the improvement of \eqref{expestimate} to an $L^{\infty}(\Omega)$ estimate:

\begin{assumption}\label{ass0'}
The following conditions are in force:
\begin{enumerate}
\item For the functions $\hat{b}\left(x, m, y\right)$ and $\hat{f}\left(x, m, y\right)$, the classical first order derivatives $f_x, f_m$ exist, and the classical second order derivatives 
$\hat{b}_{xy}$, $\hat{b}_{my}$, $\hat{f}_{xy}$, $\hat{f}_{my}$, $\hat{f}_{yy}$ exist and are bounded.

\item The third order partial derivatives of each $U^{i,N}$ that do not involve $t$, i.e $U_{x_{j_1} x_{j_2} x_{j_3}}^{i,N}$ for $j_1, j_2, j_3 \in \{1, 2, \ldots, N\}$, exist and are continuous in $(t, \mathbf{x}) \in \mathbb{R}^{N+1}$. Moreover, the second order partial derivatives of each $U^{i,N}$ that involve $t$, i.e $U_{t x_j}^{i,N}$ and $U_{x_j t}^{i,N}$ for $j \in \{1, 2, \ldots, N\}$, exist and are continuous in $(t, \mathbf{x}) \in \mathbb{R}^{N+1}$. Finally, we assume that the first order partial derivatives $U_{x_i}^{i,N}$ are bounded uniformly in $N$ and $i \in \{1, 2, \ldots, N\}$, and the second order partial derivatives $U_{x_i x_j}^{i,N}$ are bounded (not necessarily uniformly in $N$ and $i \in \{1, 2, \ldots, N\}$).

\item The third order partial derivatives of $U$ that do not involve differentiation in $t$, i.e $U_{xxx}$, $U_{xxm}$, $U_{xmx}$, $U_{mxx}$, $U_{mmx}$, $U_{mxm}$, $U_{xmm}$, $U_{mmm}$, $U_{xmz_1}$, $U_{mz_1x}$, $U_{mxz_1}$, $U_{mz_1z_1}$, $U_{mmz_1}$, $U_{mz_1m}$ and $U_{mmz_2}$, exist and are continuous under the product topologies on the spaces they are defined, e.g $[0,T] \times \mathbb{R} \times \mathcal{P}(\mathbb{R})$ for $U_{xxx}$ and $[0,T] \times \mathbb{R} \times \mathcal{P}(\mathbb{R}) \times \mathbb{R}^2$ for $U_{xmm}$. Moreover, the second order partial derivatives that involve a single differentiation in $t$, i.e $U_{tx}, U_{xt}, U_{tm}$ and $U_{mt}$, exist and are continuous under the product topologies on the spaces they are defined. Finally, the derivatives $U_{mz_1}, U_{xxm}, U_{xmz_1}$, $U_{mmm}, U_{mmz_1}$ and $U_{mmz_2}$ are bounded.
\end{enumerate}
\end{assumption}

\begin{remark}
We have the following important observations: 
\begin{enumerate}
    \item Under Assumption~\ref{ass0'}, since $U_x$ and $U_{x_i}^{i,N}$ are assumed to be bounded, the boundedness in $y$ of $\hat{b}_x, \hat{b}_m, \hat{b}_y, \hat{f}_y$ and $\hat{b}_{yy}$ in Assumption~\ref{ass0} can be reduced to local boundedness.
    \item Assumption~\ref{ass0'} imposes roughly that the functions $\hat{b}, \hat{f}, U$ and $U^{i,N}$ can be differentiated an additional time compared to the setup of \cite{Delarue2018FromTM, DLR20}. The reason is that the replacement of $L^{1}(\Omega)$ with $L^{\infty}(\Omega)$ in \eqref{expestimate} requires the reestablishment of an $L^{\infty}(\Omega)$ estimate for the differences $$U\left(t, X_t^{i,N,v^{N,*}}, \mu_t^{N, v^{N,*}}\right) - U^{i,N}\left(t, X_t^{1,N,v^{N,*}}, X_t^{2,N,v^{N,*}}, \ldots, X_t^{N,N,v^{N,*}} \right)$$
    which was obtained in those papers (\cite[Display (4.18)]{Delarue2018FromTM}), but this time with $U$ and $U^{i,N}$ replaced by $U_x$ and $U_{x_i}^{i,N}$ respectively. That $L^{\infty}(\Omega)$ estimate was obtained in \cite{Delarue2018FromTM, DLR20} via the manipulation of certain equations that involve $\hat{b}, \hat{f}, U, U^{i,N}$ and their derivatives, meaning that the above reestablishment requires computations of a similar nature after differentiating those equations, and the later increases by $1$ the orders of the highest order derivatives that appear in the computations which are the least sufficient highest orders of differentiability.
\end{enumerate}
\end{remark}
We can now present the two results we obtain in this paper. Our first result is the improvement of \eqref{expestimate} to an $L^{\infty}(\Omega)$ estimate:

\begin{theorem}\label{lem1}
Consider the Nash states $\{X^{1,N, v^{N, *}}, \, X^{2,N, v^{N, *}}, \, \ldots, \, X^{N,N, v^{N, *}}\}$ that satisfy the system \eqref{approxsys1}. If Assumptions~\ref{fundass}, \ref{ass0} and~\ref{ass0'} hold, we can find a $C > 0$ such that $\mathbb{P}$-almost surely, for all $t \in [0, T]$ and $N \in \mathbb{N}$ we have:
\begin{align}
&\sum_{i=1}^N\Bigg\{\hat{b}\left(X_t^{i, N, v^{N,*}}, \mu_t^{N, v^{N,*}}, U_{x_i}^{i, N}\left(t, X_t^{1, N, v^{N, *}}, X_t^{2, N, v^{N, *}}, \ldots, X_t^{N, N, v^{N, *}}\right)\right) \nonumber \\
& \qquad\qquad\qquad\qquad\qquad - \hat{b}\left(X_t^{i, N, v^{N,*}}, \mu_t^{N, v^{N,*}}, U_{x}\left(t, X_t^{i, N, v^{N,*}}, \mu_t^{N, v^{N,*}}\right) \right)\Bigg\}^2 \leq \frac{C}{N}.
\end{align}
\end{theorem}
\begin{remark}
It can be seen from the proof of the above estimate that we actually have an $L^{\infty}(\mathbb{R}^{N+1})$ bound of order $\mathcal{O}(N^{-1})$ for both
$$U_{x_i}^{i,N}(t, x_1, x_2, \ldots, x_N) -  U_{x}\left(t, x_i, N^{-1}\sum_{j=1}^N\delta_{x_j}\right)$$
and $U_{x_j}^{i,N}(t, x_1, x_2, \ldots, x_N)$ with $i \neq j$, which is uniform in $i$ and $N$.
\end{remark}
\noindent Then, we have our second result, which establishes the asymptotic behavior of the upper order statistics of the Nash states:
\begin{theorem}\label{mainresult}
Suppose that Assumptions~\ref{fundass}, \ref{ass0} and \ref{ass0-} hold. For each $i \in \mathbb{N}$, let $(X_t^{i, v^{*}})_{t \in [0, T]}$ be the unique solution $(X_t^{v^{*}})_{t \in [0, T]}$ to the McKean-Vlasov SDE \eqref{repagent} when $(W_t)_{t \in [0, T]}$ is replaced by $(W_t^i)_{t \in [0, T]}$. Suppose also that for a fixed $t \geq 0$, $\mathcal{L}(X_t^{v^{*}})$ belongs to the domain of attraction of an extreme value distribution with cumulative distribution function
\begin{equation*}
G_{\gamma_t}(x) = e^{-(1 + \gamma_t x)^{-\frac{1}{\gamma_t}}}, \quad 1 + \gamma_t x \geq 0,
\end{equation*}
for some value $\gamma_t \in \mathbb{R}$, in which case we can find deterministic sequences $(a_t^N)_{N \in \mathbb{N}}$ and $(b_t^N)_{N \in \mathbb{N}}$ such that
\begin{equation}\label{evtcond}
\mathbb{P}\left(\max_{i \leq N}\frac{X_t^{i, v^{*}} - b_t^N}{a_t^N} \leq x\right) \to G_{\gamma_t}(x)
\end{equation}
for any $x \in \mathbb{R}$ as $N \to \infty$. Then, sorting the Nash states $X_t^{i, N, v^{N, *}}$ as $$X_t^{i_1^N(t), N, v^{N, *}} \geq X_t^{i_2^N(t), N, v^{N, *}} \geq \ldots \geq X_t^{i_N^N(t), N, v^{N, *}},$$ for any $k \in \mathbb{N}$ we have the following convergence for the top $k$ order statistics of the Nash system:
\begin{align}\label{evtresult}
&\left(\frac{X_t^{i_1^N(t), N, v^{N, *}} - b_t^N}{a_t^N}, \frac{X_t^{i_2^N(t), N, v^{N, *}} - b_t^N}{a_t^N}, \ldots, \frac{X_t^{i_k^N(t), N, v^{N, *}} - b_t^N}{a_t^N}\right) \nonumber \\
&\qquad\qquad\qquad\qquad \to \left(\frac{(E_1)^{-\gamma_t} - 1}{\gamma_t}, \frac{(E_1 + E_2)^{-\gamma_t} - 1}{\gamma_t}, \cdots, \frac{(E_1 + E_2 + \cdots + E_k)^{-\gamma_t} - 1}{\gamma_t}\right)
\end{align}
weakly as $N \to \infty$, where $E_1, E_2, \ldots, E_k$ are i.i.d. standard exponential random variables.
\end{theorem}
\begin{remark}\label{anotherremarkwerecall}
It should be noted that:
\begin{enumerate}
    \item When $\gamma_t = 0$, we define $G_0(x) := \lim_{\gamma \to 0^+}G_{\gamma}(x) = e^{-e^{-x}}$ (the standard Gumbel distribution), in which case the limit in \eqref{evtresult} is interpreted as
\begin{equation*}
\left(-\log\left(E_1\right), -\log\left(E_1 + E_2\right), \ldots, -\log\left(E_1 + E_2 + \ldots + E_k\right)\right).
\end{equation*}
\item When the initial law $\mu_0$ is Gaussian and $\hat{b}(x,m,U_x(s,x,m))$ is linear in $x$, one can solve \eqref{repagent} explicitly to find that $X^{v^*}$ is a Gaussian process, i.e $X_t^{v^*} \sim \mathcal{N}(m_t, \sigma_t^2)$ for all $t \geq 0$. In that case, by \cite[Example 1.1.7]{HF10} we have that for each $t \geq 0$, the law $\mathcal{L}(X_t^{v^*})$ belongs to the domain of attraction of the standard Gumbel distribution and we have \eqref{evtcond} with $\gamma_t = 0$ and 
\begin{equation}\label{bseq}
{b^N_t} = \sigma_t \sqrt{2\log(N) - \log(\log(N)) -\log(4\pi)} + m_t
\end{equation}
and finally
\begin{equation}\label{aseq}
{a^N_t} = \frac{\sigma_t}{\sqrt{2\log(N) - \log(\log(N)) -\log(4\pi)}}
\end{equation}
(a relevant setup that arises in systemic risk is analyzed in Example~\ref{finexample} below). A standard way for verifying \eqref{evtcond} is the derivation of the Von Mises condition (see \cite[Theorem 1.1.8]{HF10}), and some recent computations that use Malliavin Calculus have shown that this is possible with $\gamma_t = 0$ when $\hat{b}(x,m,U_x(s,x,m))$ admits a limit as $x \to +\infty$ and when the initial law $\mu_0$ is a dirac measure (so $X_0^i = x_0$ for all $i \in \{1,2,\ldots,N\}$, for some $x_0 \in \mathbb{R}$), but the normalizing sequences $(a_t^N)_{N \in \mathbb{N}}$ and $(b_t^N)_{N \in \mathbb{N}}$ have not been computed explicitly. The derivation of an as wide as possible class of McKean-Vlasov SDEs \eqref{repagent} for which \eqref{evtcond} holds will be a major research direction in the near future, as well as the explicit computation of the sequences $(a_t^N)_{N \in \mathbb{N}}$ and $(b_t^N)_{N \in \mathbb{N}}$ whenever it is possible; having \eqref{evtcond} with $\gamma_t = 0$ whenever Assumptions~\ref{ass0} and~\ref{ass0-} are satisfied seems like a quite possible scenario.
\item As \cite[Theorem 1.2.1]{HF10} suggests, \eqref{evtcond} could hold with $\gamma_t > 0$ only when $X_t^{v^*}$ has regularly varying tails, and with $\gamma_t < 0$ only when $X_t^{v^*}$ is upper bounded by a deterministic constant. While the first is incompatible with Assumptions~\ref{ass0}, and~\ref{ass0-} by \eqref{mckeanmomentbound} in Lemma~\ref{lem2} below, the second is feasible in the extended setup where $\sigma$ depends on $X^{v^*}$ (see also Remark~\ref{ncvol}). For example, this would be the case if $-X^{v^*}$ could resemble a CIR process, i.e $(Y_t)_{t \in [0,T]}$ with
\begin{equation*}
Y_t = y + \int_0^tk(\theta - Y_s)ds + \xi\int_0^t\sqrt{|Y_s|}dW_s,
\end{equation*}
for $t \geq 0$, since one can apply \cite[Theorem 1.2.6]{HF10} on the known distribution function of $Y_t$ (see \cite{cox1985theory}) to deduce that the law of $-Y_t$ is in the domain of attraction of $G_{\gamma_t}$ with $\gamma_t = -\frac{\xi^2}{2k\theta}$
\end{enumerate}
\end{remark}
\begin{remark}\label{ncvol}
Theorem~\ref{mainresult} can be extended to setups where the constant $\sigma$ is replaced by a state-dependent volatility $\sigma(X_t^{i,N,v^N})$. A technique that allows this extension is discussed in the final section of the paper, and our argument is solid for the case where $\sigma: \mathbb{R} \mapsto \mathbb{R}_+$ is bounded and has upper bounded derivatives up to order $3$.
\end{remark}
\begin{example}\label{finexample}
Consider the linear-quadratic game introduced in \cite{CFS15}, where:
\begin{align*}
b(x, m, v) &= \bar{b}(\bar{m} - x) + v, \\
f(x, m, v) &= \frac{1}{2}v^2 - qv(\bar{m} - x) + \frac{\epsilon}{2}(\bar{m} - x)^2, \\
g(x, m) &= \frac{\tilde{g}}{2}(\bar{m} - x)^2
\end{align*}
in which $\bar{m} = \int y \cdot m(dy)$ and $\bar{b}, q, \epsilon, \bar{g}$ are constants. In this setup, the state processes $X^{i,N,v^{N}}$ describe the log-monetary reserves of $N$ banks that borrow from each other at a fixed rate $\bar{b}$ that is common for any two banks, and at any time $t \in [0, T]$, the $i$-th bank with log-wealth $X^{i,N,v^{N}}$ can also borrow from a central bank at a rate $v_t^{i,N}$ it chooses (where $v_t^{i,N} < 0$ corresponds to lending). The expected cost $J^{i,N}$ that the $i$-th bank wants to minimize (defined in \ref{objectivequantity}) is the expected value of an accumulated (until a future time $T$) penalty for deviating from the average performance of all the banks, combined with an accumulated (also until time $T$) cost of borrowing from the central bank. The above setup is also the main example provided in \cite{Delarue2018FromTM} and \cite{DLR20}, which establish a Central Limit Theorem and a Large Deviations Principle for our setup, and the authors obtain
\begin{align*}
v^{i,N,*}&= \left(q + \phi^N(t)\left(1 - \frac{1}{N}\right)\right)\left(\frac{1}{N}\sum_{j=1}^NX_t^{j,N,v^{N,*}} - X_t^{i,N,v^{N,*}}\right)
\end{align*}
where $\phi^N$ is the solution to an ODE of Riccati type, and also:
\begin{align*}
U(t,x,m) = \frac{\phi^{\infty}(t)}{2}(x - m)^2
\end{align*}
for $\phi^{\infty}(t) = \displaystyle{\lim_{N \to \infty}}\phi^{N}(t)$. In that case, the system \eqref{systemNash} becomes
\begin{align*}
X_t^{i,N,v^{N,*}} = X_0^i + \int_0^t\left(\bar{b}+q+\phi^N(s)\left(1 - \frac{1}{N}\right)\right)\left(\frac{1}{N}\sum_{j=1}^NX_s^{j,N,v^{N,*}} - X_s^{i,N,v^{N,*}}\right)ds+\sigma W_t^i,
\end{align*}
while both systems \eqref{approxsys1} and \eqref{approxsys2} coincide with
\begin{align*}
\bar{X}_t^{i,N} = X_0^i + \int_0^t\left(\bar{b}+q+\phi^{\infty}(s)\right)\left(\frac{1}{N}\sum_{j=1}^N\bar{X}_s^{j,N} - \bar{X}_s^{i,N}\right)ds+\sigma W_t^i,
\end{align*}
with the corresponding McKean-Vlasov SDE \eqref{repagent} being
\begin{align*}
X_t^{v^*} = X_0 + \int_0^t\left(\bar{b}+q+\phi^{\infty}(s)\right)\left(\mathbb{E}\left[X_t^{v^*}\right] - \bar{X}_s^{i,N}\right)ds+\sigma W_t.
\end{align*}
It is not difficult to verify that \eqref{expestimate} holds when the initial values $X_0^i$ satisfy (i) of Assumption~\ref{ass0-} (see \cite[Page 43, Display (6.6) and below]{Delarue2018FromTM} for concrete arguments), so as discussed in (iv) of Remark~\ref{werecallthis}, we can ignore the non-boundedness of $U_x$ and $U_m$ and recall Theorem~\ref{mainresult}. Then, \eqref{evtcond} holds with $\gamma_t = 0$ as discussed in (ii) of Remark~\ref{anotherremarkwerecall}, with the deterministic normalizing sequences $(a_t^N)_{N \in \mathbb{N}}$ and $(b_t^N)_{N \in \mathbb{N}}$ given by \eqref{bseq} and \eqref{aseq} when the $X_0^i$ are Gaussian, meaning that we also have the convergence \eqref{evtresult} with $\gamma_t = 0$ and the same sequences $(a_t^N)_{N \in \mathbb{N}}$ and $(b_t^N)_{N \in \mathbb{N}}$.
\end{example}
\begin{remark}
The authors in \cite{CFS15} also introduce a default barrier $D < 0$, meaning that
\begin{equation*}
X_t^{i,N,v^{N}} < D \,\,\, \text{at some} \,\, t \in [0, T] \quad \Leftrightarrow \quad \text{the i-th bank is in default at time} \,\, t.
\end{equation*}
In a relatively stable economy, there can be a large number $N$ of banks but no defaults observed until a time $t_0 \in [0, T]$, meaning that at time $t_0$, one has no picture for the probability of default. In that case, methods of Extreme Value Theory extended to our game setup can help with the estimation of the likelihood of the rare (but possibly highly impactful) event of $k$ banks being in default at the future time $t_1 > t_0$. Indeed, replacing each $X^{i,N,v^{N}}$ with $-X^{i,N,v^{N}}$ for each $i$, we obtain the same game with different initial conditions $X_0^i$, and the event of interest has probability
$$\mathbb{P}\left(X_{t_1}^{i_k^N,N,v^{N}} > -D \right) = \mathbb{P}\left(\frac{X_{t_1}^{i_k^N,N,v^{N}} -b_{t_1}^N}{a_{t_1}^N} > \frac{-D -b_{t_1}^N}{a_{t_1}^N} \right).$$ Taking $D = D_N = -a_{t_1}^NM - b_{t_1}^N$ with $M$ not extremely large compared to $N$, the large-$N$ behavior of the above probability of interest is governed by \eqref{evtresult} given by Theorem~\ref{mainresult}. Of course, this is a simplified regime where the the state processes are not stopped upon hitting the default barrier; applying the same idea to this more complex setup (which could also incorporate other realistic features like common noise) could be possible, but this is beyond the scope of this paper. Moreover, the terms $a_{t_1}^N$ and $b_{t_1}^N$ of the normalizing sequences need to be known, but an extension of statistical methods of Extreme Value Theory to the regime of Mean-Field Systems and Stochastic Differential Games (which is in our research plans for the near future) would allow for the estimation of $a_{t}^N$ and $b_{t}^N$ for $t \leq t_0$ (supposing that the states are fully observed until the present time $t_0$), and the prediction of $a_{t_1}^N$ and $b_{t_1}^N$ at the future time $t_1 > t_0$ should then be possible by using tools like regression. Finally, it should be noted that the above analysis does not require knowledge of functions that drive the Nash states (e.g $\hat{b}(x, m,y)$), meaning that it could be applicable to game models for systemic risk where the coefficients in equilibrium are not computable.
\end{remark}

To prove Theorem~\ref{lem1} and Theorem~\ref{mainresult}, we present a few lemmata. First, because the proof of Theorem~\ref{lem1} is very long and highly computational, we isolate some of its computations into the following lemma:
\begin{lemma}\label{lem0}
We define $$\mathcal{U}^{i,N}(t, \mathbf{x}) = U\left(t, x_i, \mu_{\mathbf{x}}^N\right).$$ 
Then, for $j, k \in \{1, 2, \ldots, N\}$ we can find functions $r_j^1(t, \mathbf{x})$ and $r_{j, k}^2(t, \mathbf{x})$ with $|r_j^1(t, \mathbf{x})| < \frac{C}{N}$ and $|r_{j, k}^2(t, \mathbf{x})| < \frac{C}{N}$ for all $t \geq 0$ and $\mathbf{x} \in \mathbb{R}^N$, such that
\begin{align}\label{masterderiv2}
0 = \mathcal{U}_{x_k t}^{i,N}(t, \mathbf{x}) &+ \hat{b}_x(x_k, \mu_{\mathbf{x}}^N, \mathcal{U}_{x_k}^{k, N}(t, \mathbf{x}) - r_k^1(t, \mathbf{x}))\mathcal{U}^{i,N}_{x_k}(t, \mathbf{x}) \nonumber \\ 
&+ \hat{f}_x(x_i, \mu_{\mathbf{x}}^N, \mathcal{U}_{x_i}^{i, N}(t, \mathbf{x}) - r_i^1(t, \mathbf{x}))\delta_{i k} + \frac{1}{N}\hat{f}_m(x_i, \mu_{\mathbf{x}}^N, \mathcal{U}_{x_i}^{i, N}(t, \mathbf{x}) - r_i^1(t, \mathbf{x}), x_k) \nonumber \\
&+ \hat{f}_y(x_i, \mu_{\mathbf{x}}^N, \mathcal{U}_{x_i}^{i, N}(t, \mathbf{x}) - r_i^1(t, \mathbf{x}))\mathcal{U}_{x_k x_i}^{i, N}(t, \mathbf{x})\nonumber \\
&+ \sum_{j=1}^N\frac{1}{N}\hat{b}_m(x_j, \mu_{\mathbf{x}}^N, \mathcal{U}_{x_j}^{j, N}(t, \mathbf{x}) - r_j^1(t, \mathbf{x}), x_k)\mathcal{U}^{i,N}_{x_j}(t, \mathbf{x}) \nonumber \\
&+ \sum_{j=1}^N\hat{b}_y(x_j, \mu_{\mathbf{x}}^N, \mathcal{U}_{x_j}^{j, N}(t, \mathbf{x}) - r_j^1(t, \mathbf{x}))\mathcal{U}^{i,N}_{x_j}(t, \mathbf{x})\mathcal{U}_{x_k x_j}^{j, N}(t, \mathbf{x}) \nonumber \\
&+ \sum_{j=1}^N\hat{b}(x_j, \mu_{\mathbf{x}}^N, U_{x}(t, x_j, \mu_{\mathbf{x}}^N))\mathcal{U}^{i,N}_{x_k x_j}(t, \mathbf{x}) + \frac{\sigma^2}{2}\sum_{j = 1}^N\mathcal{U}^{i,N}_{x_k x_j x_j}(t, \mathbf{x}) - r_{i, k}^2(t, \mathbf{x}),
\end{align}
for all $i, k \in \{1, 2, \ldots, N\}$ and all $t \geq 0$ and $\mathbf{x} \in \mathbb{R}^N$.
\end{lemma}  

\noindent Next, for the proof of Theorem~\ref{mainresult}, we will make use of the following lemma, which gives efficient bounds for the moments of a solution to a McKean-Vlasov SDE, and for the moments of the Wasserstein distances of the law of the solution to that SDE from the empirical measures that approximate this law. Some of the bounds obtained are natural and probably already known, but they are included for the sake of completeness.
\begin{lemma}\label{lem2}
Let $Z^{1}_0, Z^{2}_0, \ldots$ be a sequence of independent random variables with a common law $\mu_0 \in \mathcal{P}(\mathbb{R})$, and consider the empirical measures
\begin{align}
\mu^{N, Z^N}_t = \frac{1}{N}\sum_{\ell=1}^N\delta_{Z_t^{\ell,N}} \quad \text{and} \quad \mu^{N, Z}_t = \frac{1}{N}\sum_{\ell = 1}^N\delta_{Z_t^{\ell}}
\end{align}
where the dynamics of the particles $Z^{i, N}$ are given by the system
\begin{equation} \label{eq_particle_SDE}
Z^{i, N}_t = Z^{i}_0 + \int_0^tB\left(s, Z^{i, N}_{s}, \mu^{N, Z^N}_s \right) ds + W^i_t, \quad i = 1, 2, \ldots, N
\end{equation}
and $Z^i = (Z_t^{i})_{t \in [0, T]}$ is the solution to the corresponding McKean-Vlasov SDE:
\begin{equation}
\begin{aligned} \label{eq_McKean_SDE}
Z_t &= Z_0 + \int_0^tB\left(s, Z_s, \mathcal{L}\left(Z_s\right) \right) ds + W_t \\
\mathcal{L}\left(Z_0\right) &= \mu_0.
\end{aligned}
\end{equation}
when $(W_t)_{t \in [0, T]}$ and $Z_0$ are replaced by $(W_t^i)_{t \in [0, T]}$ and $Z_0^i$ respectively, for each $i \in \{1, 2, \ldots, N\}$. Suppose that there exist $\kappa_0, L > 0$ such that
\begin{equation}\label{initexpo}
\mathbb{E}\left[e^{\kappa_0Z_0^2}\right] < \infty
\end{equation}
and
\begin{align}\label{lip}
&\left|B(t, z_1, m_1) - B(t, z_2, m_2)\right| \leq L\left(|z_1 - z_2| + \mathcal{W}_1\left(m_1, m_2\right)\right)
\end{align}
for any two pairs $(z_1, m_1), (z_2, m_2) \in \mathbb{R} \times (\mathcal{P}(\mathbb{R}), \mathcal{W}_1)$ and any $t \in [0, T]$. Then, we have that:
\begin{equation}\label{mckeanmomentbound}
\mathbb{E}\left[Z_t^p\right] \leq C^pp\Gamma\left(\frac{p}{2}\right)
\end{equation}
and
\begin{equation}\label{t-wassbound2}
\mathbb{E}\left[\mathcal{W}_1^p\left(\mu^{N, Z}_t, \mathcal{L}\left(Z_t\right)\right)\right] \leq \frac{C^pp\Gamma\left(\frac{p}{2}\right)}{N^{\frac{p}{2}}}.
\end{equation}
and also
\begin{equation}\label{t-wassbound}
\mathbb{E}\left[\mathcal{W}_1^p\left(\mu^{N, Z^N}_t, \mathcal{L}\left(Z_t\right)\right)\right] \leq \frac{C^pp\Gamma\left(\frac{p}{2}\right)}{N^{\frac{p}{2}}}.
\end{equation} 
for any $t \in [0, T]$, any $p \geq 1$, and some $C > 0$ that does not depend on $t$ or $p$.
\end{lemma}

\noindent Finally, we present the lemma that is used for all the reductions in the proof of Theorem~\ref{mainresult}, the proof of which uses Pinsker's inequality.  
\begin{lemma}\label{lem3}
For each $N \in \mathbb{N}$, we consider the $N$-dimensional processes $$\mathbf{Z}^{N,k} = (Z^{1,N,k}, Z^{2,N,k}, \ldots, Z^{N,N,k})$$ for $k \in \{1, 2\}$, with $Z^{i,N,k} = (Z_t^{i,N,k})_{t \in [0, T]}$, which are taken to be the pathwise unique strong solutions to the $N$-dimensional SDEs:
\begin{align}\label{sys1}
\mathbf{Z}_t^{N, 1} = \mathbf{Z}_0^N + \int_0^t\mathbf{A}^{N}\left(s, \mathbf{Z}_s^{N,1}\right)\left(\mathbf{B}^{N,1}\left(s, \mathbf{Z}_s^{N,1}\right)ds + d\mathbf{W}^N_s\right) + \int_0^t\mathbf{C}^{N}\left(s, \mathbf{Z}_s^{N,1}\right)ds
\end{align}
and
\begin{align}\label{sys2}
\mathbf{Z}_t^{N, 2} = \mathbf{Z}_0^N + \int_0^t\mathbf{A}^{N}\left(s, \mathbf{Z}_s^{N,2}\right)\left(\mathbf{B}^{N,2}\left(s, \mathbf{Z}_s^{N,2}\right)ds + d\mathbf{W}^N_s\right) + \int_0^t\mathbf{C}^{N}\left(s, \mathbf{Z}_s^{N,2}\right)ds.
\end{align}
In the above, for $k \in \{1,2\}$, $\mathbf{B}^{N,k}(\cdot)$ and $\mathbf{C}^N(\cdot)$ are measurable $\mathbb{R}^N$-valued functions with $\mathbf{B}^{N,k}(\cdot)$ being locally bounded, $\mathbf{A}(\cdot)$ is a measurable function with values in the space of $N\times N$ matrices, and $\mathbf{W}^N := (W^1, W^2, \ldots, W^N)$. Suppose that for all $N \in \mathbb{N}$ and all $s \in [0, T]$ we have:
\begin{align}\label{momentcond1}
&\int_0^T\mathbb{E}\left[\left\Vert\mathbf{B}^{N,2}\left(s, \mathbf{Z}_s^{N,2}\right) - \mathbf{B}^{N,1}\left(s, \mathbf{Z}_s^{N,2}\right)\right\Vert_{2}^2\right]ds \leq \frac{C}{N}.
\end{align}
for some $C > 0$ that does not depend on $N$. Then, it holds that
\begin{align}\label{jjjjj}
\left|\mathbb{E}\left[H^N\left(\mathbf{Z}_t^{N,1}\right)\right] - \mathbb{E}\left[H^N\left(\mathbf{Z}_t^{N,2}\right)\right]\right| \to 0
\end{align}
as $N \to \infty$, for any measurable function $H^N: \mathbb{R}^N \mapsto [0, 1]$ and all $t \in [0, T]$.
\end{lemma}

\section{Proofs}
In this section we place all the proofs of this paper. Between the proofs of our two main results, we will first present the proof of Theorem~\ref{lem1}, but for that we must first prove Lemma~\ref{lem0}. 

\begin{proof}[Proof of Lemma \ref{lem0}]
For a fixed $\mathbf{x} = (x_1, x_2, \ldots, x_N) \in \mathbb{R}^N$ and for each $i \in \{1, 2, \ldots, N\}$, plugging $\mu \rightarrow \mu_{\mathbf{x}}^N$ and $x \rightarrow x_i$ in \eqref{master} and recalling the computations in \cite[Proof of Proposition 4.1]{Delarue2018FromTM} we get:
\begin{align}
0 = \mathcal{U}_t^{i,N}(t, \mathbf{x}) &+ \sum_{j=1}^N\hat{b}(x_j, \mu_{\mathbf{x}}^N, U_x(t, x_j, \mu_{\mathbf{x}}^N))\mathcal{U}^{i,N}_{x_j}(t, \mathbf{x}) + \hat{f}(x_i, \mu_{\mathbf{x}}^N, U_x(t, x_i, \mu_{\mathbf{x}}^N)) \nonumber \\
&+ \frac{\sigma^2}{2}\left(\mathcal{U}^{i,N}_{x_i x_i}(t, \mathbf{x}) - \frac{1}{N}U_{x m}(t, x_i, \mu_{\mathbf{x}}^N, x_i) - \frac{1}{N^2}U_{m m}(t, x_i, \mu_{\mathbf{x}}^N, x_i, x_i)\right)
\nonumber \\
&+ \frac{\sigma^2}{2}\sum_{\substack{j = 1 \\ j \neq i}}^N\left(\mathcal{U}^{i,N}_{x_j x_j}(t, \mathbf{x}) - \frac{1}{N^2}U_{m m}(t, x_i, \mu_{\mathbf{x}}^N, x_j, x_j)\right).
\end{align}
Differentiating the above with respect to $x_k$ we get:
\begin{align}\label{masterderiv}
0 = \mathcal{U}_{x_k t}^{i,N}(t, \mathbf{x}) &+ \hat{b}_x(x_k, \mu_{\mathbf{x}}^N, U_x(t, x_k, \mu_{\mathbf{x}}^N))\mathcal{U}^{i,N}_{x_k}(t, \mathbf{x}) + \hat{f}_x(x_i, \mu_{\mathbf{x}}^N, U_x(t, x_i, \mu_{\mathbf{x}}^N))\delta_{i k} \nonumber \\
&+ \frac{1}{N}\hat{f}_m(x_i, \mu_{\mathbf{x}}^N, U_x(t, x_i, \mu_{\mathbf{x}}^N), x_k) \nonumber \\
&+ \hat{f}_y(x_i, \mu_{\mathbf{x}}^N, U_x(t, x_i, \mu_{\mathbf{x}}^N))\frac{\partial}{\partial x_k} U_x(t, x_i, \mu_{\mathbf{x}}^N)\nonumber \\
&+ \sum_{j=1}^N\frac{1}{N}\hat{b}_m(x_j, \mu_{\mathbf{x}}^N, U_x(t, x_j, \mu_{\mathbf{x}}^N), x_k)\mathcal{U}^{i,N}_{x_j}(t, \mathbf{x}) \nonumber \\
&+ \sum_{j=1}^N\hat{b}_y(x_j, \mu_{\mathbf{x}}^N, U_x(t, x_j, \mu_{\mathbf{x}}^N))\mathcal{U}^{i,N}_{x_j}(t, \mathbf{x})\frac{\partial}{\partial x_k} U_x(t, x_j, \mu_{\mathbf{x}}^N) \nonumber \\
&+ \sum_{j=1}^N\hat{b}(x_j, \mu_{\mathbf{x}}^N, U_x(t, x_j, \mu_{\mathbf{x}}^N))\mathcal{U}^{i,N}_{x_k x_j}(t, \mathbf{x}) + \frac{\sigma^2}{2}\sum_{j = 1}^N\mathcal{U}^{i,N}_{x_k x_j x_j}(t, \mathbf{x}) \nonumber \\
&- \frac{\sigma^2}{2}\frac{\partial}{\partial x_k}\left(\frac{1}{N}U_{x m}(t, x_i, \mu_{\mathbf{x}}^N, x_i) + \frac{1}{N^2}U_{m m}(t, x_i, \mu_{\mathbf{x}}^N, x_i, x_i)\right)
\nonumber \\
&-\frac{\sigma^2}{2}\sum_{\substack{j = 1 \\ j \neq i}}^N\frac{1}{N^2}\frac{\partial}{\partial x_k}U_{m m}(t, x_i, \mu_{\mathbf{x}}^N, x_j, x_j)
\end{align}
where all derivatives exist in a classical sense and we can interchange the derivatives with each other and with integrals by the regularity provided by Assumptions~\ref{ass0} and~\ref{ass0'}. In the above, we have also used that for any differentiable function $F(\mu)$ of $\mu \in \mathcal{P}$ we have
\begin{align}\label{derivexpression}
\frac{\partial}{\partial x_j}F(\mu_{\mathbf{x}}^N) = \frac{1}{N}F_m(\mu_{\mathbf{x}}^N, x_j). 
\end{align}
Using now \eqref{derivexpression} again we can obtain the following relations:
\begin{align}\label{derivrel1}
\frac{\partial}{\partial x_k}U_{m m}(t, x_i, \mu_{\mathbf{x}}^N, x_j, x_j) &= U_{x m m}(t, x_i, \mu_{\mathbf{x}}^N, x_j, x_j)\delta_{i k} + \frac{1}{N}U_{m m m}(t, x_i, \mu_{\mathbf{x}}^N, x_j, x_j, x_k) \nonumber \\
&\qquad + U_{m m z_1}(t, x_i, \mu_{\mathbf{x}}^N, x_j, x_j)\delta_{j k} + U_{m m z_2}(t, x_i, \mu_{\mathbf{x}}^N, x_j, x_j)\delta_{j k}
\end{align}
for any $i, j, k \in \{1, 2, \ldots, N\}$, 
\begin{align}\label{derivrel2}
U_{x}(t, x_j, \mu_{\mathbf{x}}^N) &= \mathcal{U}_{x_j}^{j, N}(t, \mathbf{x}) - \frac{1}{N}U_{m}(t, x_j, \mu_{\mathbf{x}}^N, x_j) 
\end{align}
which implies also that
\begin{align}\label{derivrel3}
\frac{\partial}{\partial x_k}U_{x}(t, x_j, \mu_{\mathbf{x}}^N) &= \mathcal{U}_{x_k x_j}^{j, N}(t, \mathbf{x}) - \frac{1}{N}U_{m x}(t, x_j, \mu_{\mathbf{x}}^N, x_j)\delta_{j k} - \frac{1}{N}U_{m z_1}(t, x_j, \mu_{\mathbf{x}}^N, x_j)\delta_{j k} \nonumber \\
& \qquad\qquad\qquad\qquad\qquad\qquad\qquad\qquad\quad\, - \frac{1}{N^2}U_{m m}(t, x_j, \mu_{\mathbf{x}}^N, x_j, x_k) 
\end{align}
for any $j, k \in \{1, 2, \ldots, N\}$, and finally
\begin{align}\label{derivrel4}
\frac{\partial}{\partial x_k}U_{x m}(t, x_i, \mu_{\mathbf{x}}^N, x_i) = U_{x x m}(t, x_i, \mu_{\mathbf{x}}^N, x_i)\delta_{i k} &+ \frac{1}{N}U_{x m m}(t, x_i, \mu_{\mathbf{x}}^N, x_i, x_k) \nonumber \\
&+ U_{x m z_1}(t, x_i, \mu_{\mathbf{x}}^N, x_i)\delta_{i k}
\end{align}
for any $i, k \in \{1, 2, \ldots, N\}$. Substituting all the above in \eqref{masterderiv} we obtain the desired form with 
\begin{equation}
r_j^1(t, \mathbf{x}) = \frac{1}{N}U_{m}(t, x_j, \mu_{\mathbf{x}}^N, x_j)
\end{equation}
for all $j \in \{1, 2, \ldots, N\}$ and
\begin{align}
r_{i, k}^2(t, \mathbf{x}) = &- \frac{1}{N}\hat{f}_y(x_i, \mu_{\mathbf{x}}^N, \mathcal{U}_{x_i}^{i, N}(t, \mathbf{x}) - r_i^1(t, \mathbf{x})) \nonumber \\
&\qquad\qquad \times\left\{U_{m x}(t, x_i, \mu_{\mathbf{x}}^N, x_i)\delta_{i k} + U_{m z_1}(t, x_i, \mu_{\mathbf{x}}^N, x_i)\delta_{i k}\right\} \nonumber \\
&+ \frac{1}{N^2}\hat{f}_y(x_i, \mu_{\mathbf{x}}^N, \mathcal{U}_{x_i}^{i, N}(t, \mathbf{x}) - r_i^1(t, \mathbf{x}))U_{m m}(t, x_i, \mu_{\mathbf{x}}^N, x_i, x_k)\nonumber \\
&- \frac{1}{N}\hat{b}_y(x_k, \mu_{\mathbf{x}}^N, \mathcal{U}_{x_k}^{k, N}(t, \mathbf{x}) - r_k^1(t, \mathbf{x})) \nonumber \\
&\qquad\qquad \times\left\{U_{m x}(t, x_k, \mu_{\mathbf{x}}^N, x_k) + U_{m z_1}(t, x_k, \mu_{\mathbf{x}}^N, x_k)\right\} \nonumber \\
&+ \frac{1}{N^2}\sum_{j=1}^N\hat{b}_y(x_j, \mu_{\mathbf{x}}^N, \mathcal{U}_{x_j}^{j, N}(t, \mathbf{x}) - r_j^1(t, \mathbf{x}))U_{m m}(t, x_j, \mu_{\mathbf{x}}^N, x_j, x_k)\nonumber \\
&- \frac{\sigma^2}{2}\frac{1}{N}\left(U_{x x m}(t, x_i, \mu_{\mathbf{x}}^N, x_i)\delta_{i k} + \frac{1}{N}U_{x m m}(t, x_i, \mu_{\mathbf{x}}^N, x_i, x_k)\right) \nonumber \\
&- \frac{\sigma^2}{2}\frac{1}{N} U_{x m z_1}(t, x_i, \mu_{\mathbf{x}}^N, x_i)\delta_{i k}
\nonumber \\
&-\frac{\sigma^2}{2}\sum_{j = 1}^N\frac{1}{N^2}\left(U_{x m m}(t, x_i, \mu_{\mathbf{x}}^N, x_j, x_j)\delta_{i k} + \frac{1}{N}U_{m m m}(t, x_i, \mu_{\mathbf{x}}^N, x_j, x_j, x_k)\right) \nonumber \\
&-\frac{\sigma^2}{2}\frac{1}{N^2}\left(U_{m m z_1}(t, x_i, \mu_{\mathbf{x}}^N, x_k, x_k) + U_{m m z_2}(t, x_i, \mu_{\mathbf{x}}^N, x_k, x_k)\right)
\end{align}
for all $i, k \in \{1, 2, \ldots, N\}$, which are both bounded by $\frac{C}{N}$ by Assumptions~\ref{ass0} and~\ref{ass0'} so our proof is complete.
\end{proof}

\noindent We can now prove the first of our two main results.

\begin{proof}[Proof of Theorem~\ref{lem1}]
First, under the notation of Lemma~\ref{lem0}, we define 
\begin{equation}\label{zety}
\begin{aligned}
Y_t^{i,k,N} &= U_{x_k}^{i,N}\left(t, X_t^{1, N, v^{N, *}}, X_t^{2, N, v^{N, *}}, \ldots, X_t^{N, N, v^{N, *}}\right) \\
Z_t^{i,j,k,N} &= U_{x_k x_j}^{i,N}\left(t, X_t^{1, N, v^{N, *}}, X_t^{2, N, v^{N, *}}, \ldots, X_t^{N, N, v^{N, *}}\right) \\
\mathcal{Y}_t^{i,k,N} &= \mathcal{U}_{x_k}^{i,N}\left(t, X_t^{1, N, v^{N, *}}, X_t^{2, N, v^{N, *}}, \ldots, X_t^{N, N, v^{N, *}}\right) \\
\mathcal{Z}_t^{i,j,k,N} &= \mathcal{U}_{x_k x_j}^{i,N}\left(t, X_t^{1, N, v^{N, *}}, X_t^{2, N, v^{N, *}}, \ldots, X_t^{N, N, v^{N, *}}\right).
\end{aligned}
\end{equation}
By Assumptions~\ref{ass0} and~\ref{ass0'}, we can differentiate the $i$-th equation of \eqref{sysPDE} with respect to $x_k$, and use again \eqref{derivexpression} to obtain
\begin{align}\label{sysPDEderiv}
0 = U_{x_k t}^{i,N}(t, \mathbf{x}) &+ \hat{b}_x\left(x_{k}, \mu_{\mathbf{x}}^N, U_{x_{k}}^{k,N}(t, \mathbf{x}) \right)  U_{x_k}^{i,N}(t, \mathbf{x}) + \hat{f}_x\left(x_{i}, \mu_{\mathbf{x}}^N, U_{x_{i}}^{i,N}(t, \mathbf{x}) \right) \delta_{i k} \nonumber \\
&+ \frac{1 }{N} \hat{f}_m\left(x_i, \mu_{\mathbf{x}}^N, U_{x_i}^{i,N}(t, \mathbf{x}), x_k \right) + \hat{f}_y\left(x_{i}, \mu_{\mathbf{x}}^N, U_{x_{i}}^{i,N}(t, \mathbf{x}) \right)U_{x_k x_{i}}^{i,N}(t, \mathbf{x}) \nonumber \\
&+ \sum_{j = 1}^N \hat{b}\left(x_{j}, \mu_{\mathbf{x}}^N, U_{x_{j}}^{j, N}(t, \mathbf{x})\right) U_{x_k x_{j}}^{i,N}(t, \mathbf{x}) \nonumber \\
&+ \sum_{j = 1}^N \frac{1}{N} \hat{b}_m\left(x_{j}, \mu_{\mathbf{x}}^N, U_{x_{j}}^{j, N}(t, \mathbf{x}), x_k\right)U_{x_{j}}^{i,N}(t, \mathbf{x}) \nonumber \\
&+ \sum_{j = 1}^N \hat{b}_y\left(x_{j}, \mu_{\mathbf{x}}^N, U_{x_{j}}^{j, N}(t, \mathbf{x})\right)U_{x_{j}}^{i,N}(t, \mathbf{x})U_{x_k x_{j}}^{j, N}(t, \mathbf{x}) + \frac{\sigma^2}{2}\sum_{j = 1}^NU_{x_k x_{j} x_{j}}^{i, N}(t, \mathbf{x}),
\end{align}
along with terminal conditions 
\begin{equation}
U_{x_k}^{i,N}(T, \mathbf{x}) = \frac{\partial }{\partial x_k}g\left(x_i, \mu_{\mathbf{x}}^N\right).
\end{equation}
Therefore, recalling Assumptions~\ref{ass0} and~\ref{ass0'} again, we can apply It\^{o}'s formula on the process $U^{i,N}_{x_k}(t, X_t^{1, N, v^{N, *}}, X_t^{2, N, v^{N, *}}, \ldots, X_t^{N, N, v^{N, *}})$ and use \eqref{systemNash} and then \eqref{sysPDEderiv} to find that
\begin{align}\label{It\^{o}Y}
dY_t^{i,k,N} &= U_{x_k t}^{i,N}\left(t, X_t^{1, N, v^{N, *}}, X_t^{2, N, v^{N, *}}, \ldots, X_t^{N, N, v^{N, *}}\right)dt \nonumber \\
& \quad + \sum_{j=1}^NU_{x_k x_j}^{i,N}\left(t, X_t^{1, N, v^{N, *}}, X_t^{2, N, v^{N, *}}, \ldots, X_t^{N, N, v^{N, *}}\right)\nonumber \\
&\qquad \qquad \times \hat{b}\left(X_t^{j, N, v^{N,*}}, \mu_t^{N, v^{N,*}}, U_{x_j}^{j, N}\left(t, X_t^{1, N, v^{N, *}}, X_t^{2, N, v^{N, *}}, \ldots, X_t^{N, N, v^{N, *}}\right)\right)dt \nonumber \\
& \quad + \frac{\sigma^2}{2}\sum_{j=1}^NU_{x_k x_j x_j}^{i,N}\left(t, X_t^{1, N, v^{N, *}}, X_t^{2, N, v^{N, *}}, \ldots, X_t^{N, N, v^{N, *}}\right)dt + \sigma\sum_{j=1}^NZ_t^{i,j,k,N}dW_t^j \nonumber \\
&= -\hat{b}_x\left(X_t^{k, N, v^{N, *}}, \mu_{t}^{N, v^{N, *}}, Y_t^{k, k, N} \right)Y_t^{i, k, N}dt - \hat{f}_x\left(X_t^{i, N, v^{N, *}}, \mu_{t}^{N, v^{N, *}}, Y_t^{i, i, N} \right) \delta_{i k}dt \nonumber \\
&\quad- \frac{1 }{N} \hat{f}_m\left(X_t^{i, N, v^{N, *}}, \mu_{t}^{N, v^{N, *}}, Y_t^{i, i, N}, X_t^{k, N, v^{N, *}} \right)dt \nonumber \\
&\quad - \hat{f}_y\left(X_t^{i, N, v^{N, *}}, \mu_{t}^{N, v^{N, *}}, Y_t^{i, i, N} \right)Z_t^{i,i,k,N}dt \nonumber \\
&\quad - \frac{1}{N}\sum_{j = 1}^N \hat{b}_m\left(X_t^{j, N, v^{N, *}}, \mu_{t}^{N, v^{N, *}}, Y_t^{j, j, N}, X_t^{k, N, v^{N, *}}\right)Y_t^{i, j, N}dt \nonumber \\
&\quad - \sum_{j = 1}^N \hat{b}_y\left(X_t^{j, N, v^{N, *}}, \mu_{t}^{N, v^{N, *}}, Y_t^{j, j, N}\right)Y_t^{i,j,N}Z_t^{j, j, k, N}dt + \sigma\sum_{j=1}^NZ_t^{i,j,k,N}dW_t^j. 
\end{align}
On the other hand, applying It\^{o}'s formula on $\mathcal{U}^{i,N}_{x_k}(t, X_t^{1, N, v^{N, *}}, X_t^{2, N, v^{N, *}}, \ldots, X_t^{N, N, v^{N, *}})$ and using \eqref{systemNash} and then Lemma~\ref{lem0} we obtain
\begin{align}\label{It\^{o}YCal}
d\mathcal{Y}_t^{i, k, N} &= \mathcal{U}_{x_k t}^{i,N}\left(t, X_t^{1, N, v^{N, *}}, X_t^{2, N, v^{N, *}}, \ldots, X_t^{N, N, v^{N, *}}\right)dt \nonumber \\
& \quad + \sum_{j=1}^N\mathcal{U}_{x_k x_j}^{i,N}\left(t, X_t^{1, N, v^{N, *}}, X_t^{2, N, v^{N, *}}, \ldots, X_t^{N, N, v^{N, *}}\right)\nonumber \\
&\quad \qquad \quad \times \hat{b}\left(X_t^{j, N, v^{N,*}}, \mu_t^{N, v^{N,*}}, U_{x_j}^{j, N}\left(t, X_t^{1, N, v^{N, *}}, X_t^{2, N, v^{N, *}}, \ldots, X_t^{N, N, v^{N, *}}\right)\right)dt \nonumber \\
& \quad + \frac{\sigma^2}{2}\sum_{j=1}^N\mathcal{U}_{x_k x_j x_j}^{i,N}\left(t, X_t^{1, N, v^{N, *}}, X_t^{2, N, v^{N, *}}, \ldots, X_t^{N, N, v^{N, *}}\right)dt + \sigma\sum_{j=1}^N\mathcal{Z}_t^{i, j, k, N}dW_t^j \nonumber \\
&= -\hat{b}_x(X_t^{k, N, v^{N, *}}, \mu_t^{N, v^{N,*}}, \mathcal{Y}_t^{k,k,N} - R_t^{k,1})\mathcal{Y}_t^{i,k,N}dt \nonumber \\
&\quad -\hat{f}_x(X_t^{i, N, v^{N, *}}, \mu_t^{N, v^{N,*}}, \mathcal{Y}_t^{i,i,N} - R_t^{i,1})\delta_{i k}dt \nonumber \\
&\quad - \frac{1}{N}\hat{f}_m(X_t^{i, N, v^{N, *}}, \mu_t^{N, v^{N,*}}, \mathcal{Y}_t^{i,i,N} - R_t^{i,1}, X_t^{k, N, v^{N, *}})dt \nonumber \\
&\quad - \hat{f}_y(X_t^{i, N, v^{N, *}}, \mu_t^{N, v^{N,*}}, \mathcal{Y}_t^{i,i,N} - R_t^{i,1})\mathcal{Z}_{t}^{i, i, k, N}dt \nonumber \\
&\quad -\sum_{j=1}^N\frac{1}{N}\hat{b}_m(X_t^{j, N, v^{N, *}}, \mu_t^{N, v^{N,*}}, \mathcal{Y}_t^{j,j,N} - R_t^{j,1}, X_t^{k, N, v^{N, *}})\mathcal{Y}_t^{i, j, N}dt \nonumber \\
&\quad -\sum_{j=1}^N\hat{b}_y(X_t^{j, N, v^{N, *}}, \mu_t^{N, v^{N,*}}, \mathcal{Y}_t^{j,j,N} - R_t^{j,1})\mathcal{Y}_t^{i, j, N}\mathcal{Z}_{t}^{j, j, k, N}dt \nonumber \\
&\quad + R_t^{i,k,2}dt + \sigma\sum_{j=1}^N\mathcal{Z}_t^{i, j, k, N}dW_t^j
\end{align}
where for $j, k \in \{1, 2, \ldots, N\}$ we define $$R_t^{j,1} = r_j^1\left(t, X_t^{1, N, v^{N, *}}, X_t^{2, N, v^{N, *}}, \ldots, X_t^{N, N, v^{N, *}}\right)$$ and $$R_t^{j,k,2} = r_{j,k}^2\left(t, X_t^{1, N, v^{N, *}}, X_t^{2, N, v^{N, *}}, \ldots, X_t^{N, N, v^{N, *}}\right)$$ which are both bounded by $\frac{C}{N}$. As we have $$Y_T^{i, k, N} = \mathcal{Y}_T^{i, k, N} = \frac{\partial }{\partial x_k}g\left(x_i, \mu_{\mathbf{x}}^N\right),$$ subtracting \eqref{It\^{o}Y} from \eqref{It\^{o}YCal} and applying It\^{o}'s formula on $(\mathcal{Y}_t^{i, k, N} - Y_t^{i, k, N})^2$ we obtain:
\begin{align}\label{pdyn}
&\left(\mathcal{Y}_t^{i, k, N} - Y_t^{i, k, N}\right)^2 \nonumber \\
& \qquad = 2\int_t^T\left(\mathcal{Y}_s^{i, k, N} - Y_s^{i, k, N}\right)\Bigg\{\hat{b}_x(X_s^{k, N, v^{N, *}}, \mu_s^{N, v^{N,*}}, \mathcal{Y}_s^{k,k,N} - R_s^{k,1})\mathcal{Y}_s^{i,k,N} \nonumber \\
&\qquad\qquad\qquad\qquad\qquad\qquad\qquad\qquad\qquad\, -\hat{b}_x\left(X_s^{k, N, v^{N, *}}, \mu_{s}^{N, v^{N, *}}, Y_s^{k, k, N} \right)Y_s^{i, k, N}\Bigg\}ds \nonumber \\
&\qquad \qquad + 2\int_t^T\left(\mathcal{Y}_s^{i, k, N} - Y_s^{i, k, N}\right)\Bigg\{\hat{f}_x(X_s^{i, N, v^{N, *}}, \mu_s^{N, v^{N,*}}, \mathcal{Y}_s^{i,i,N} - R_s^{i,1})\delta_{i k} \nonumber \\
&\qquad\qquad\qquad\qquad\qquad\qquad\qquad\qquad\qquad\qquad\, -\hat{f}_x\left(X_s^{i, N, v^{N, *}}, \mu_{s}^{N, v^{N, *}}, Y_s^{i, i, N} \right) \delta_{i k}\Bigg\}ds \nonumber \\
&\qquad \qquad + \frac{2}{N}\int_t^T\left(\mathcal{Y}_s^{i, k, N} - Y_s^{i, k, N}\right) \nonumber \\
&\qquad\qquad\qquad\qquad\qquad\times\Bigg\{\hat{f}_m(X_s^{i, N, v^{N, *}}, \mu_s^{N, v^{N,*}}, \mathcal{Y}_s^{i,i,N} - R_s^{i,1}, X_s^{k, N, v^{N, *}}) \nonumber \\
&\qquad\qquad\qquad\qquad\qquad\qquad\qquad\quad -\hat{f}_m\left(X_s^{i, N, v^{N, *}}, \mu_{s}^{N, v^{N, *}}, Y_s^{i, i, N}, X_s^{k, N, v^{N, *}} \right)\Bigg\}ds \nonumber \\
&\qquad \qquad + 2\int_t^T\left(\mathcal{Y}_s^{i, k, N} - Y_s^{i, k, N}\right)\Bigg\{\hat{f}_y(X_s^{i, N, v^{N, *}}, \mu_s^{N, v^{N,*}}, \mathcal{Y}_s^{i,i,N} - R_s^{i,1})\mathcal{Z}_{s}^{i, i, k, N} \nonumber \\
&\qquad\qquad\qquad\qquad\qquad\qquad\qquad\qquad\qquad\qquad\, -\hat{f}_y\left(X_s^{i, N, v^{N, *}}, \mu_{s}^{N, v^{N, *}}, Y_s^{i, i, N} \right)Z_s^{i,i,k,N}\Bigg\}ds \nonumber \\
&\qquad \qquad + \frac{2}{N}\int_t^T\left(\mathcal{Y}_s^{i, k, N} - Y_s^{i, k, N}\right) \nonumber \\
&\qquad\qquad\qquad\qquad\qquad\times\sum_{j=1}^N\Bigg\{\hat{b}_m(X_s^{j, N, v^{N, *}}, \mu_s^{N, v^{N,*}}, \mathcal{Y}_s^{j,j,N} - R_s^{j,1}, X_s^{k, N, v^{N, *}})\mathcal{Y}_s^{i, j, N} \nonumber \\
&\qquad\qquad\qquad\qquad\qquad\qquad\qquad\quad -\hat{b}_m\left(X_s^{j, N, v^{N, *}}, \mu_{s}^{N, v^{N, *}}, Y_s^{j, j, N}, X_s^{k, N, v^{N, *}}\right)Y_s^{i, j, N}\Bigg\}ds \nonumber \\
&\qquad \qquad + 2\int_t^T\left(\mathcal{Y}_s^{i, k, N} - Y_s^{i, k, N}\right) \nonumber \\
&\qquad\qquad\qquad\qquad\qquad\times\sum_{j=1}^N\Bigg\{\hat{b}_y(X_s^{j, N, v^{N, *}}, \mu_s^{N, v^{N,*}}, \mathcal{Y}_s^{j,j,N} - R_s^{j,1})\mathcal{Y}_s^{i, j, N}\mathcal{Z}_{s}^{j, j, k, N} \nonumber \\
&\qquad\qquad\qquad\qquad\qquad\qquad\qquad\quad - \hat{b}_y\left(X_s^{j, N, v^{N, *}}, \mu_{s}^{N, v^{N, *}}, Y_s^{j, j, N}\right)Y_s^{i,j,N}Z_s^{j, j, k, N}\Bigg\}ds \nonumber \\
&\qquad \qquad - 2\sigma\int_t^T\left(\mathcal{Y}_s^{i, k, N} - Y_s^{i, k, N}\right)\sum_{j=1}^N\left(\mathcal{Z}_s^{i, j, k, N} - Z_s^{i, j, k, N}\right)dW_s^j \nonumber \\
&\qquad \qquad - \sigma^2\int_t^T\sum_{j=1}^N\left(\mathcal{Z}_s^{i, j, k, N} - Z_s^{i, j, k, N}\right)^2ds.
\end{align}
Next, by \eqref{derivexpression} we have 
\begin{align}
\mathcal{U}_{x_j}^{i, N}(t, \mathbf{x}) = U_{x}(t, x_i, \mu_{\mathbf{x}}^N)\delta_{i j} + \frac{1}{N}U_{m}(t, x_i, \mu_{\mathbf{x}}^N, x_j) 
\end{align}
for any $i, j \in \{1, 2, \ldots, N\}$, which also implies that
\begin{align}
\mathcal{U}_{x_k x_j}^{i, N}(t, \mathbf{x}) = U_{x x}(t, x_i, \mu_{\mathbf{x}}^N)\delta_{i j}\delta_{i k} &+ \frac{1}{N}U_{m x}(t, x_i, \mu_{\mathbf{x}}^N, x_k)\delta_{i j} + \frac{1}{N}U_{m x}(t, x_i, \mu_{\mathbf{x}}^N, x_j)\delta_{i k} \nonumber \\
&- \frac{1}{N^2}U_{m m}(t, x_i, \mu_{\mathbf{x}}^N, x_j, x_k) + \frac{1}{N}U_{m z_1}(t, x_i, \mu_{\mathbf{x}}^N, x_j)\delta_{j k}
\end{align}
for any $i, j, k \in \{1, 2, \ldots, N\}$, so the boundedness of all the second order derivatives except $U_{mz_1}$ in Assumption~\ref{ass0} and the boundedness of $U_{mz_1}$ in Assumption~\ref{ass0'} imply that the processes $\mathcal{Y}_s^{i,j,N}$ and $\mathcal{Z}_s^{i,j,k,N}$ defined in \eqref{zety} satisfy 
\begin{equation}\label{bestdecays}
\begin{aligned}
\left|\mathcal{Y}_s^{i,j,N}\right| &\leq C\left(\frac{1}{N} + \delta_{i j}\right), \\
\left|\mathcal{Z}_s^{i,j,k,N}\right| &\leq C\left(\frac{1}{N^2} + \frac{1}{N}\delta_{i j} + \frac{1}{N}\delta_{i k} + \frac{1}{N}\delta_{j k} + \delta_{i j}\delta_{i k}\right).
\end{aligned}
\end{equation}
Hence, we can write 
\begin{align}\label{bracketbound1}
&\left|\hat{b}_x(X_s^{k, N, v^{N, *}}, \mu_s^{N, v^{N,*}}, \mathcal{Y}_s^{k,k,N} - R_s^{k,1})\mathcal{Y}_s^{i,k,N} -\hat{b}_x\left(X_s^{k, N, v^{N, *}}, \mu_{s}^{N, v^{N, *}}, Y_s^{k, k, N} \right)Y_s^{i, k, N}\right| \nonumber \\
&\qquad \qquad \leq \left\Vert \hat{b}_x \right\Vert_{\infty} \left|\mathcal{Y}_s^{i,k,N} - Y_s^{i, k, N}\right| + \left\Vert \hat{b}_{xy} \right\Vert_{\infty} \left|\mathcal{Y}_s^{k,k,N} - R_s^{k,1} - Y_s^{k, k, N}\right|\left|\mathcal{Y}_s^{i,k,N}\right| \nonumber \\
&\qquad \qquad \leq C\left\{\left|\mathcal{Y}_s^{i,k,N} - Y_s^{i, k, N}\right| + \left(\left|\mathcal{Y}_s^{k,k,N} - Y_s^{k, k, N}\right| + \frac{1}{N}\right)\left(\frac{1}{N} + \delta_{i k}\right)\right\}
\end{align}
where the first inequality is obtained by bounding the distance of the two terms in the first line by the sum of the distances of those terms from $\hat{b}_x(X_s^{k, N, v^{N, *}}, \mu_{s}^{N, v^{N, *}}, Y_s^{k, k, N})\mathcal{Y}_s^{i, k, N}$, and by using then the mean value theorem, Assumption~\ref{ass0}, Assumption~\ref{ass0'} and the $\mathcal{O}(\frac{1}{N})$ bounds of the processes $|R_t^{i,1}|$ and $|R_t^{i,k,2}|$. Similarly, we can obtain:
\begin{align}\label{bracketbound2}
&\left|\hat{f}_x(X_s^{i, N, v^{N, *}}, \mu_s^{N, v^{N,*}}, \mathcal{Y}_s^{i,i,N} - R_s^{i,1})\delta_{i k} - \hat{f}_x\left(X_s^{i, N, v^{N, *}}, \mu_{s}^{N, v^{N, *}}, Y_s^{i, i, N} \right) \delta_{i k}\right| \nonumber \\
&\qquad \qquad \leq \left\Vert \hat{f}_{xy} \right\Vert_{\infty} \left|\mathcal{Y}_s^{i,i,N} - R_s^{i,1} - Y_s^{i, i, N}\right|\delta_{i k} \nonumber \\
&\qquad \qquad \leq C\left(\left|\mathcal{Y}_s^{i,i,N} - Y_s^{i, i, N}\right| + \frac{1}{N}\right)\delta_{i k},
\end{align}
\begin{align}\label{bracketbound3}
&\left|\hat{f}_m(X_s^{i, N, v^{N, *}}, \mu_s^{N, v^{N,*}}, \mathcal{Y}_s^{i,i,N} - R_s^{i,1}, X_s^{k, N, v^{N, *}}) - \hat{f}_m\left(X_s^{i, N, v^{N, *}}, \mu_{s}^{N, v^{N, *}}, Y_s^{i, i, N}, X_s^{k, N, v^{N, *}} \right)\right| \nonumber \\
&\qquad \qquad \leq \left\Vert \hat{f}_{my} \right\Vert_{\infty} \left|\mathcal{Y}_s^{i,i,N} - R_s^{i,1} - Y_s^{i, i, N}\right| \nonumber \\
&\qquad \qquad \leq C\left(\left|\mathcal{Y}_s^{i,i,N} - Y_s^{i, i, N}\right| + \frac{1}{N}\right),
\end{align}
\begin{align}\label{bracketbound4}
&\left|\hat{f}_y(X_s^{i, N, v^{N, *}}, \mu_s^{N, v^{N,*}}, \mathcal{Y}_s^{i,i,N} - R_s^{i,1})\mathcal{Z}_{s}^{i, i, k, N} - \hat{f}_y\left(X_s^{i, N, v^{N, *}}, \mu_{s}^{N, v^{N, *}}, Y_s^{i, i, N} \right)Z_s^{i,i,k,N}\right| \nonumber \\
&\qquad \qquad \leq \left\Vert \hat{f}_y \right\Vert_{\infty} \left|\mathcal{Z}_s^{i,i,k,N} - Z_s^{i, i, k, N}\right| + \left\Vert \hat{f}_{yy} \right\Vert_{\infty} \left|\mathcal{Y}_s^{i,i,N} - R_s^{i,1} - Y_s^{i, i, N}\right|\left|\mathcal{Z}_s^{i,i,k,N}\right| \nonumber \\
&\qquad \qquad \leq C\left\{\left|\mathcal{Z}_s^{i,i,k,N} - Z_s^{i, i, k, N}\right| + \left(\left|\mathcal{Y}_s^{i,i,N} - Y_s^{i, i, N}\right| + \frac{1}{N}\right)\left(\frac{1}{N} + \delta_{i k}\right)\right\},
\end{align}
\begin{align}\label{bracketbound5}
&\Big|\hat{b}_m(X_s^{j, N, v^{N, *}}, \mu_s^{N, v^{N,*}}, \mathcal{Y}_s^{j,j,N} - R_s^{j,1}, X_s^{k, N, v^{N, *}})\mathcal{Y}_s^{i, j, N} \nonumber \\
& \qquad \qquad \qquad \qquad \qquad \qquad \qquad \qquad -\hat{b}_m\left(X_s^{j, N, v^{N, *}}, \mu_{s}^{N, v^{N, *}}, Y_s^{j, j, N}, X_s^{k, N, v^{N, *}}\right)Y_s^{i, j, N}\Big| \nonumber \\
&\qquad \qquad \leq \left\Vert \hat{b}_m \right\Vert_{\infty} \left|\mathcal{Y}_s^{i,j,N} - Y_s^{i, j, N}\right| + \left\Vert \hat{b}_{my} \right\Vert_{\infty} \left|\mathcal{Y}_s^{j,j,N} - R_s^{j,1} - Y_s^{j, j, N}\right|\left|\mathcal{Y}_s^{i,j,N}\right| \nonumber \\
&\qquad \qquad \leq C\left\{\left|\mathcal{Y}_s^{i,j,N} - Y_s^{i, j, N}\right| + \left(\left|\mathcal{Y}_s^{j,j,N} - Y_s^{j, j, N}\right| + \frac{1}{N}\right)\left(\frac{1}{N} + \delta_{i j}\right)\right\},
\end{align}
and finally
\begin{align}\label{bracketbound6}
&\Big|\hat{b}_y(X_s^{j, N, v^{N, *}}, \mu_s^{N, v^{N,*}}, \mathcal{Y}_s^{j,j,N} - R_s^{j,1})\mathcal{Y}_s^{i, j, N}\mathcal{Z}_{s}^{j, j, k, N} \nonumber \\
& \qquad \qquad \qquad \qquad \qquad \qquad \qquad - \hat{b}_y\left(X_s^{j, N, v^{N, *}}, \mu_{s}^{N, v^{N, *}}, Y_s^{j, j, N}\right)Y_s^{i,j,N}Z_s^{j, j, k, N}\Big| \nonumber \\
&\qquad \leq \left\Vert \hat{b}_y \right\Vert_{\infty} \left|\mathcal{Y}_s^{i,j,N}\mathcal{Z}_s^{j,j,k,N} - Y_s^{i,j,N}Z_s^{j,j,k,N}\right| \nonumber \\
& \qquad \qquad \qquad \qquad \qquad \qquad \qquad + \left\Vert \hat{b}_{yy} \right\Vert_{\infty} \left|\mathcal{Y}_s^{j,j,N} - R_s^{j,1} - Y_s^{j, j, N}\right|\left|\mathcal{Y}_s^{i,j,N}\mathcal{Z}_s^{j,j,k,N}\right| \nonumber \\
&\qquad = \left\Vert \hat{b}_y \right\Vert_{\infty} \bigg|\left(\mathcal{Y}_s^{i,j,N} - Y_s^{i,j,N}\right)\mathcal{Z}_s^{j,j,k,N} + \left(\mathcal{Z}_s^{j,j,k,N} - Z_s^{j,j,k,N}\right)\mathcal{Y}_s^{i,j,N} \nonumber \\
& \qquad\qquad\quad\,\, \qquad \qquad \qquad \qquad \qquad \qquad + \left(\mathcal{Z}_s^{j,j,k,N} - Z_s^{j,j,k,N}\right)\left(Y_s^{i,j,N} - \mathcal{Y}_s^{i,j,N}\right)\bigg|\nonumber \\
& \qquad \qquad \qquad \qquad \qquad \qquad \qquad + \left\Vert \hat{b}_{yy} \right\Vert_{\infty} \left|\mathcal{Y}_s^{j,j,N} - R_s^{j,1} - Y_s^{j, j, N}\right|\left|\mathcal{Y}_s^{i,j,N}\mathcal{Z}_s^{j,j,k,N}\right| \nonumber \\
&\qquad \leq \left\Vert \hat{b}_y \right\Vert_{\infty} \bigg(\left|\mathcal{Y}_s^{i,j,N} - Y_s^{i,j,N}\right|\left|\mathcal{Z}_s^{j,j,k,N}\right| + \left|\mathcal{Z}_s^{j,j,k,N} - Z_s^{j,j,k,N}\right|\left|\mathcal{Y}_s^{i,j,N}\right| \nonumber \\
& \qquad\qquad\qquad\, \qquad \qquad \qquad \qquad \qquad \qquad + \left|\mathcal{Z}_s^{j,j,k,N} - Z_s^{j,j,k,N}\right|\left|Y_s^{i,j,N} - \mathcal{Y}_s^{i,j,N}\right|\bigg) \nonumber \\
& \qquad \qquad \qquad \qquad \qquad \qquad \qquad + \left\Vert \hat{b}_{yy} \right\Vert_{\infty} \left|\mathcal{Y}_s^{j,j,N} - R_s^{j,1} - Y_s^{j, j, N}\right|\left|\mathcal{Y}_s^{i,j,N}\mathcal{Z}_s^{j,j,k,N}\right| \nonumber \\
&\qquad  \leq C\bigg\{\left|\mathcal{Y}_s^{i,j,N} - Y_s^{i, j, N}\right|\left(\frac{1}{N} + \delta_{j k}\right) + \left|\mathcal{Z}_s^{j,j,k,N} - Z_s^{j, j, k, N}\right|\left(\frac{1}{N} + \delta_{i j}\right) \nonumber \\
& \qquad\qquad\quad\, \qquad \qquad \qquad \qquad \qquad \qquad + \left|\mathcal{Z}_s^{j,j,k,N} - Z_s^{j,j,k,N}\right|\left|Y_s^{i,j,N} - \mathcal{Y}_s^{i,j,N}\right| \nonumber \\
&\qquad \qquad \qquad \qquad \qquad \qquad \qquad + \left(\left|\mathcal{Y}_s^{j,j,N} - Y_s^{j, j, N}\right| + \frac{1}{N}\right)\left(\frac{1}{N^2} + \frac{\delta_{i j} + \delta_{j k}}{N} + \delta_{i j}\delta_{j k}\right)\bigg\}.
\end{align}
Plugging \eqref{bracketbound1} - \eqref{bracketbound6} in \eqref{pdyn} we obtain
\begin{align}\label{pbound}
&\left(\mathcal{Y}_t^{i, k, N} - Y_t^{i, k, N}\right)^2 + \sigma^2\int_t^T\sum_{j=1}^N\left(\mathcal{Z}_s^{i, j, k, N} - Z_s^{i, j, k, N}\right)^2ds \nonumber \\
& \qquad \leq C\Bigg\{\int_t^T\left(\mathcal{Y}_s^{i, k, N} - Y_s^{i, k, N}\right)^{2}ds +  \int_t^T\left|\mathcal{Y}_s^{i, k, N} - Y_s^{i, k, N}\right|\left|\mathcal{Y}_s^{k, k, N} - Y_s^{k, k, N}\right|ds \nonumber \\
&\qquad \qquad \qquad + \int_t^T\left|\mathcal{Y}_s^{i, k, N} - Y_s^{i, k, N}\right|\frac{1}{N}\left(\frac{1}{N} + \delta_{i k}\right)ds \nonumber \\
&\qquad \qquad \qquad + \int_t^T\left|\mathcal{Y}_s^{i, k, N} - Y_s^{i, k, N}\right|\left|\mathcal{Y}_s^{i, i, N} - Y_s^{i, i, N}\right|\left(\frac{1}{N} + \delta_{i k}\right)ds \nonumber \\
&\qquad \qquad \qquad + \int_t^T\left|\mathcal{Y}_s^{i, k, N} - Y_s^{i, k, N}\right|\left|\mathcal{Z}_s^{i, i, k, N} - Z_s^{i, i, k, N}\right|ds \nonumber \\
&\qquad \qquad \qquad + \int_t^T\left|\mathcal{Y}_s^{i, k, N} - Y_s^{i, k, N}\right|\sum_{\substack{j = 1 \\ j \neq i}}^N\left|\mathcal{Y}_s^{i, j, N} - Y_s^{i, j, N}\right|\left(\frac{1}{N} + \delta_{j k}\right)ds \nonumber \\
&\qquad \qquad \qquad + \int_t^T\left|\mathcal{Y}_s^{i, k, N} - Y_s^{i, k, N}\right|\sum_{\substack{j = 1 \\ j \neq i}}^N\left|\mathcal{Y}_s^{j, j, N} - Y_s^{j, j, N}\right|\left(\frac{1}{N^2} + \frac{\delta_{j k}}{N}\right)ds \nonumber \\
&\qquad \qquad \qquad + \int_t^T\left|\mathcal{Y}_s^{i, k, N} - Y_s^{i, k, N}\right|\frac{1}{N}\sum_{\substack{j = 1 \\ j \neq i}}^N\left|\mathcal{Z}_s^{j, j, k, N} - Z_s^{j, j, k, N}\right|ds \nonumber \\
&\qquad \qquad \qquad + \int_t^T\left|\mathcal{Y}_s^{i, k, N} - Y_s^{i, k, N}\right|\sum_{j = 1}^N\left|\mathcal{Z}_s^{j, j, k, N} - Z_s^{j, j, k, N}\right|\left|\mathcal{Y}_s^{i, j, N} - Y_s^{i, j, N}\right|ds\Bigg\} \nonumber \\
&\qquad \qquad \qquad - 2\sigma\int_t^T\left(\mathcal{Y}_s^{i, k, N} - Y_s^{i, k, N}\right)\sum_{j=1}^N\left(\mathcal{Z}_s^{i, j, k, N} - Z_s^{i, j, k, N}\right)dW_s^j
\end{align}
Next, we use the inequality $2ab \leq a^2 + b^2$, the Cauchy-Schwartz inequality and some trivial bounds to control the terms that appear in the above if we set $k = i$. Specifically, for every $i \in \{1, 2, \ldots, N\}$ and any $\epsilon > 0$ we obtain 
\begin{equation}\label{AMGM1}
\int_t^T\left|\mathcal{Y}_s^{i, i, N} - Y_s^{i, i, N}\right|\frac{1}{N}ds \leq \frac{1}{2}\int_t^T\left(\mathcal{Y}_s^{i, i, N} - Y_s^{i, i, N}\right)^2ds + \frac{T - t}{2N^2},
\end{equation}
\begin{align}\label{AMGM2}
&\int_t^T\left|\mathcal{Y}_s^{i, i, N} - Y_s^{i, i, N}\right|\left|\mathcal{Z}_s^{i, i, i, N} - Z_s^{i, i, i, N}\right|ds \nonumber \\
& \qquad\qquad\qquad\qquad \leq \frac{1}{2\epsilon}\int_t^T\left(\mathcal{Y}_s^{i, i, N} - Y_s^{i, i, N}\right)^2ds + \frac{\epsilon}{2}\int_t^T\left(\mathcal{Z}_s^{i, i, i, N} - Z_s^{i, i, i, N}\right)^2ds,
\end{align} 
\begin{align}\label{AMGM3}
&\int_t^T\left|\mathcal{Y}_s^{i, i, N} - Y_s^{i, i, N}\right|\sum_{\substack{j = 1 \\ j \neq i}}^N\left|\mathcal{Y}_s^{i, j, N} - Y_s^{i, j, N}\right|\left(\frac{1}{N} + \delta_{j i}\right)ds \nonumber \\
&\qquad\qquad\qquad\qquad = \int_t^T\left|\mathcal{Y}_s^{i, i, N} - Y_s^{i, i, N}\right|\frac{1}{N}\sum_{\substack{j = 1 \\ j \neq i}}^N\left|\mathcal{Y}_s^{i, j, N} - Y_s^{i, j, N}\right|ds \nonumber \\
&\qquad\qquad\qquad\qquad \leq \frac{1}{2}\int_t^T\left(\mathcal{Y}_s^{i, i, N} - Y_s^{i, i, N}\right)^2ds + \frac{1}{2}\int_t^T\frac{1}{N}\sum_{\substack{j = 1 \\ j \neq i}}^N\left(\mathcal{Y}_s^{i, j, N} - Y_s^{i, j, N}\right)^2ds,
\end{align}
\begin{align}\label{AMGM4}
&\int_t^T\left|\mathcal{Y}_s^{i, i, N} - Y_s^{i, i, N}\right|\sum_{\substack{j = 1 \\ j \neq i}}^N\left|\mathcal{Y}_s^{j, j, N} - Y_s^{j, j, N}\right|\left(\frac{1}{N^2} + \frac{\delta_{j i}}{N}\right)ds \nonumber \\
&\qquad\qquad\qquad\qquad = \int_t^T\left|\mathcal{Y}_s^{i, i, N} - Y_s^{i, i, N}\right|\frac{1}{N^2}\sum_{\substack{j = 1 \\ j \neq i}}^N\left|\mathcal{Y}_s^{j, j, N} - Y_s^{j, j, N}\right|ds \nonumber \\
&\qquad\qquad\qquad\qquad \leq \frac{1}{2}\int_t^T\left(\mathcal{Y}_s^{i, i, N} - Y_s^{i, i, N}\right)^2ds + \frac{1}{2}\int_t^T\frac{1}{N^3}\sum_{\substack{j = 1 \\ j \neq i}}^N\left(\mathcal{Y}_s^{j, j, N} - Y_s^{j, j, N}\right)^2ds
\end{align}
and
\begin{align}\label{AMGM5}
&\int_t^T\left|\mathcal{Y}_s^{i, i, N} - Y_s^{i, i, N}\right|\frac{1}{N}\sum_{\substack{j = 1 \\ j \neq i}}^N\left|\mathcal{Z}_s^{j, j, i, N} - Z_s^{j, j, i, N}\right|ds \nonumber \\
&\qquad\qquad\qquad\qquad \leq \frac{1}{2\epsilon}\int_t^T\left(\mathcal{Y}_s^{i, i, N} - Y_s^{i, i, N}\right)^2ds + \frac{\epsilon}{2}\int_t^T\frac{1}{N}\sum_{\substack{j = 1 \\ j \neq i}}^N\left(\mathcal{Z}_s^{j, j, i, N} - Z_s^{j, j, i, N}\right)^2ds,
\end{align}
and since $Y_s^{i,i,N}$ and $\mathcal{Y}_s^{i,i,N}$ are bounded uniformly in $N$ and $i \in \{1, 2, \ldots, N\}$ due to \eqref{bestdecays} and the uniform boundedness of $U_{x_i}^{i,N}$ provided by Assumption~\ref{ass0'}, we have also
\begin{align}\label{AMGM6}
&\int_t^T\left|\mathcal{Y}_s^{i, i, N} - Y_s^{i, i, N}\right|\sum_{j = 1}^N\left|\mathcal{Z}_s^{j, j, i, N} - Z_s^{j, j, i, N}\right|\left|\mathcal{Y}_s^{i, j, N} - Y_s^{i, j, N}\right|ds \nonumber \\
&\qquad\qquad\qquad\qquad = \int_t^T\sum_{j = 1}^N\left|\mathcal{Y}_s^{i, i, N} - Y_s^{i, i, N}\right|\left|\mathcal{Y}_s^{i, j, N} - Y_s^{i, j, N}\right|\left|\mathcal{Z}_s^{j, j, i, N} - Z_s^{j, j, i, N}\right|ds \nonumber \\
&\qquad\qquad\qquad\qquad \leq \int_t^T\sum_{j = 1}^N\frac{1}{2\epsilon}\left(\mathcal{Y}_s^{i, i, N} - Y_s^{i, i, N}\right)^2\left(\mathcal{Y}_s^{i, j, N} - Y_s^{i, j, N}\right)^2ds \nonumber \\
&\qquad\qquad\qquad\qquad\qquad\qquad + \int_t^T\sum_{j = 1}^N\frac{\epsilon}{2}\left(\mathcal{Z}_s^{j, j, i, N} - Z_s^{j, j, i, N}\right)^2ds \nonumber \\
&\qquad\qquad\qquad\qquad \leq \frac{C}{2\epsilon}\int_t^T\sum_{j = 1}^N\left(\mathcal{Y}_s^{i, j, N} - Y_s^{i, j, N}\right)^2ds + \frac{\epsilon}{2}\int_t^T\sum_{j = 1}^N\left(\mathcal{Z}_s^{j, j, i, N} - Z_s^{j, j, i, N}\right)^2ds.
\end{align}
Taking $k = i$ in \eqref{pbound}, using \eqref{AMGM1} - \eqref{AMGM6}, summing for all $i \in \{1, 2, \ldots, N\}$, and finally observing that $Z_s^{i, j, i, N} = Z_s^{i, i, j, N}$ and $\mathcal{Z}_s^{i, j, i, N} = \mathcal{Z}_s^{i, i, j, N}$ for all $i, j \in \{1, 2, \ldots, N\}$ since the derivatives of $U^{i,N}$ and $\mathcal{U}^{i,N}$ in $x_i$ and $x_j$ can be interchanged due to the regularity provided by Assumption~\ref{ass0}, we find that
\begin{align}\label{pbound2}
&\sum_{i=1}^N\left(\mathcal{Y}_t^{i, i, N} - Y_t^{i, i, N}\right)^2 + \left(\sigma^2 - C\epsilon\right)\int_t^T\sum_{i,j=1}^N\left(\mathcal{Z}_s^{i, i, j, N} - Z_s^{i, i, j, N}\right)^2ds \nonumber \\
& \qquad\qquad \leq C\Bigg\{\left(1 + \frac{1}{\epsilon}\right)\int_t^T\sum_{i=1}^N\left(\mathcal{Y}_s^{i, i, N} - Y_s^{i, i, N}\right)^2ds \nonumber \\
& \qquad\qquad\qquad\qquad\qquad + \int_t^T\left(\frac{1}{N} + \frac{1}{\epsilon}\right)\sum_{i,j=1}^N\left(\mathcal{Y}_s^{i, j, N} - Y_s^{i, j, N}\right)^2ds + \frac{1}{N}\Bigg\} \nonumber \\
& \qquad\qquad\qquad\qquad +2\sigma\int_t^T\left(\mathcal{Y}_s^{i, i, N} - Y_s^{i, i, N}\right)\sum_{i,j=1}^N\left(\mathcal{Z}_s^{i, i, j, N} - Z_s^{i, i, j, N}\right)dW_s^j.
\end{align}
The stochastic integrals in the above expression are martingales due to the boundedness of $Z_s^{i,i,j,N}$ and $\mathcal{Z}_s^{i,i,j,N}$ defined in \eqref{zety}, which follows from \eqref{bestdecays} and the boundedness of the functions $U_{x_i x_j}^{i,N}$ provided by Assumption~\ref{ass0'}. Observing also that $Z_s^{i,j}$ and $\mathcal{Z}_s^{i,j}$ defined in \cite{Delarue2018FromTM} are precisely $Y_s^{i,j,N}$ and $\mathcal{Y}_s^{i,j,N}$ defined here respectively, by \cite[Equation (4.19)]{Delarue2018FromTM} we have that 
\begin{equation*}
\int_t^T\mathbb{E}\left[\sum_{i,j=1}^N\left(\mathcal{Y}_s^{i, j, N} - Y_s^{i, j, N}\right)^2ds \, \Bigg| \, \mathcal{F}_t\right] \leq \frac{C}{N}
\end{equation*}
$\mathbb{P}$ - almost surely. Hence, taking expectation given $\mathcal{F}_t$ on \eqref{pbound2}, picking $\epsilon$ to be sufficiently small and performing the transformation $s \rightarrow T - s$, we find that the function
\begin{align}
g^{N}(t) := \left\Vert\sum_{i=1}^N\left(\mathcal{Y}_{T - t}^{i, i, N} - Y_{T - t}^{i, i, N}\right)^2\right\Vert_{L^{\infty}\left(\Omega\right)}
\end{align}
satisfies
\begin{align}\label{finalpbound}
g^{N}(t) \leq C\int_0^tg^{N}(s)ds + \frac{C}{N}
\end{align}
for all $t \in [0, T]$. Then, a simple application of Gr\"onwall's inequality gives $g^{N}(t) \leq \frac{C}{N}e^{Ct}$ for all $t \in [0, T]$, so that $\mathbb{P}$ - almost surely we have
\begin{align}\label{infnormbound}
\sum_{i=1}^N\left(\mathcal{Y}_{t}^{i, i, N} - Y_{t}^{i, i, N}\right)^2 \leq \frac{C}{N}
\end{align}
for all $t \in [0, T]$ and for some deterministic constant $C > 0$. 

By using now the mean value theorem, \eqref{derivrel2}, \eqref{zety}, the triangle inequality, the elementary inequality $(a + b)^2 \leq 2(a^2 + b^2)$ and \eqref{infnormbound}, we can bound
\begin{align}
&\sum_{i=1}^N\Bigg\{\hat{b}\left(X_t^{i, N, v^{N,*}}, \mu_t^{N, v^{N,*}}, U_{x_i}^{i, N}\left(t, X_t^{1, N, v^{N, *}}, X_t^{2, N, v^{N, *}}, \ldots, X_t^{N, N, v^{N, *}}\right)\right) \nonumber \\
& \qquad\qquad\qquad\qquad\qquad - \hat{b}\left(X_t^{i, N, v^{N,*}}, \mu_t^{N, v^{N,*}}, U_{x}\left(t, X_t^{i, N, v^{N,*}}, \mu_t^{N, v^{N,*}}\right) \right)\Bigg\}^2 \nonumber \\
&\qquad\qquad \leq \left\Vert \hat{b}_y\right\Vert_{\infty}^2\sum_{i=1}^N\Bigg\{U_{x_i}^{i, N}\left(t, X_t^{1, N, v^{N, *}}, X_t^{2, N, v^{N, *}}, \ldots, X_t^{N, N, v^{N, *}}\right) \nonumber \\
& \qquad\qquad\qquad\qquad\qquad\qquad\qquad\qquad\qquad\quad\,\,\, -  U_{x}\left(t, X_t^{i, N, v^{N,*}}, \mu_t^{N, v^{N,*}}\right)\Bigg\}^2 \nonumber \\
& \qquad\qquad \leq \left\Vert \hat{b}_y\right\Vert_{\infty}^2\left(\frac{2\left\Vert U_m\right\Vert_{\infty}^2}{N} + 2\sum_{i=1}^N\left(Y_t^{i,i,N} - \mathcal{Y}_t^{i,i,N}\right)^2\right) \nonumber \\
& \qquad\qquad \leq \frac{C}{N}
\end{align}
$\mathbb{P}$-almost surely, so our proof is complete.
\end{proof}

\noindent We proceed now to proving the lemmata necessary for our second main result, i.e for Theorem~\ref{mainresult}.

\begin{proof}[Proof of Lemma \ref{lem2}]
Borrowing the notation from \cite{FOGU15}, for $\gamma, t > 0$ we consider 
\begin{equation}\label{expobound}
\mathcal{E}_{2,\gamma}\left(\mathcal{L}\left(Z_t\right)\right) = \mathbb{E}\left[e^{\gamma Z_t^2}\right]
\end{equation}
and we will first show that this is a finite and bounded in $t \in [0, T]$ quantity for sufficiently small $\gamma > 0$. Using \eqref{lip} and the triangle inequality we get that
\begin{equation}\label{comp1}
\left|Z_t\right| \leq \left|Z_0\right| + \int_0^t\left\{\left|B\left(s, 0, \mathcal{L}\left(Z_s\right) \right)\right| + L\left|Z_s\right|\right\}ds + \sup_{r \in [0, T]}\left|W_r\right| 
\end{equation}
for $t \in [0, T]$, so an application of Gr\"onwall's inequality gives
\begin{equation}
\left|Z_t\right| \leq \left\{\left|Z_0\right| + \int_0^t\left|B\left(s, 0, \mathcal{L}\left(Z_s\right) \right)\right|ds + \sup_{r \in [0, T]}\left|W_r\right| \right\}e^{Lt}
\end{equation}
which can be used along with $$\sup_{r \in [0, T]}|W_r| \leq \sup_{r \in [0, T]}W_r + \sup_{r \in [0, T]}(-W_r)$$ and the Cauchy-Schwartz inequality to give
\begin{align}\label{expobound2}
&\mathcal{E}_{2,\gamma}\left(\mathcal{L}\left(Z_t\right)\right) \leq e^{4\gamma e^{2LT}T\int_0^T B^2\left(s, 0, \mathcal{L}\left(Z_s\right) \right)ds} \nonumber \\
&\qquad\qquad\qquad\qquad \times \sqrt{\mathbb{E}\left[e^{12\gamma e^{2LT}Z_0^2}\right] \mathbb{E}\left[e^{12\gamma e^{2LT}\left(\sup_{r \in [0, T]}W_r \right)^2}\right]\mathbb{E}\left[e^{12\gamma e^{2LT}\left(\sup_{r \in [0, T]}(-W_r) \right)^2}\right]}.
\end{align}
We know now that $\sup_{r \in [0, T]}W_r$ and $\sup_{r \in [0, T]}(-W_r)$ have the same distribution as the absolute value of a Gaussian random variable, so using also \eqref{initexpo} we can deduce that the right-hand side of \eqref{expobound2} is finite for sufficiently small $\gamma > 0$. This gives the desired finiteness and boundedness of $\mathcal{E}_{2,\gamma}\left(\mathcal{L}\left(Z_t\right)\right)$, so by Markov's inequality we have that
\begin{align}\label{concentration2}
\mathbb{P}\left(Z_t > x\right) &= \mathbb{P}\left(e^{\gamma Z_t^2} > e^{\gamma x^2}\right)\nonumber \\
&\leq \mathcal{E}_{2,\gamma}\left(\mathcal{L}\left(Z_t\right)\right)e^{-\gamma x^2} \nonumber \\
&\leq Ce^{-\gamma x^2}
\end{align}
for any $t \in [0, T]$, with $C$ being e.g the right-hand side of \eqref{expobound2} (which does not depend on $t$). Moreover, the finiteness and boundedness of $\mathcal{E}_{2,\gamma}\left(\mathcal{L}\left(Z_t\right)\right)$ allows us to recall \cite[Theorem 1.2]{FOGU15} and deduce that there exist $c, C > 0$ such that
\begin{equation}\label{concentration}
\mathbb{P}\left(\mathcal{W}_1\left(\mu^{N, Z}_t, \mathcal{L}\left(Z_t\right)\right) > x\right) \leq Ce^{-cNx^2}
\end{equation}
for any $t \in [0, T]$. Then, \eqref{mckeanmomentbound} and \eqref{t-wassbound2} are obtained by applying the argument used in the proof of \cite[Lemma 1.4]{HDS15} on \eqref{concentration2} and \eqref{concentration} respectively. We will now bound the moments of $\mathcal{W}_1(\mu^{N, Z^N}_t, \mathcal{L}(Z_t))$ by using a coupling argument. For $p \geq 1$, we have for each $i \in \{1, 2, \ldots, N\}$ and all $t \in [0, T]$:
\begin{align}\label{pest}
\left|\left(Z_t^{i,N} - Z_t^i\right)^p\right| &= \left|p\int_0^t\left(Z_s^{i,N} - Z_s^i\right)^{p-1}\left(B\left(s, Z^{i, N}_{s}, \mu^{N, Z^N}_s \right) - B\left(s, Z^{i}_{s}, \mathcal{L}\left(Z_s\right) \right)\right)ds\right| \nonumber \\
&\leq pL\int_0^t\left|Z_s^{i,N} - Z_s^i\right|^{p-1}\left(\left|Z^{i, N}_{s} - Z^{i}_{s}\right| + \mathcal{W}_1\left(\mu^{N, Z^N}_s, \mathcal{L}\left(Z_s\right)\right)\right)ds \nonumber \\
&\leq pL\int_0^t\left|Z_s^{i,N} - Z_s^i\right|^{p}ds + pL\int_0^t\left|Z_s^{i,N} - Z_s^i\right|^{p-1} \mathcal{W}_1\left(\mu^{N, Z^N}_s, \mu^{N, Z}_s\right)ds \nonumber \\
&\qquad\qquad\qquad\qquad\,\,\,\,\,\,\,\,\,\,\,\quad + pL\int_0^t\left|Z_s^{i,N} - Z_s^i\right|^{p-1} \mathcal{W}_1\left(\mu^{N, Z}_s, \mathcal{L}\left(Z_s\right)\right)ds \nonumber \\
&\leq (3p - 2)L\int_0^t\left|Z_s^{i,N} - Z_s^i\right|^{p}ds + L\int_0^t \mathcal{W}_1^p\left(\mu^{N, Z^N}_s, \mu^{N, Z}_s\right)ds \nonumber \\
&\qquad\qquad\qquad\qquad\qquad\qquad\,\,\,\,\,\quad + L\int_0^t\mathcal{W}_1^p\left(\mu^{N, Z}_s, \mathcal{L}\left(Z_s\right)\right)ds
\end{align}
where we used \eqref{lip}, the triangle inequality for the Wasserstein distance and the elementary inequality $pa^{p-1}b \leq (p-1)a^p + b^p$. Denoting now by $\Pi$ the set of all permutations $\sigma~:~ \{1, 2, \ldots, N\} \mapsto \{1, 2, \ldots, N\}$, it is known that
\begin{align}\label{wassbound}
\mathcal{W}_1^p\left(\mu^{N, Z^N}_s, \mu^{N, Z}_s\right) &\leq \mathcal{W}_p^p\left(\mu^{N, Z^N}_s, \mu^{N, Z}_s\right) \nonumber \\
&= \inf_{\sigma \in \Pi}\frac{1}{N}\sum_{i=1}^N\left|Z_t^{i,N} - Z_t^{\sigma(i)}\right|^p \leq \frac{1}{N}\sum_{i=1}^N\left|Z_t^{i,N} - Z_t^i\right|^p
\end{align}
so taking the average of \eqref{pest} over all $i \in \{1, 2, \ldots, N\}$ we can obtain
\begin{align}\label{pest2}
\frac{1}{N}\sum_{i=1}^N\left|Z_t^{i,N} - Z_t^i\right|^p &\leq (3p - 1)L\int_0^t\frac{1}{N}\sum_{i=1}^N\left|Z_s^{i,N} - Z_s^i\right|^{p}ds + L\int_0^t\mathcal{W}_1^p\left(\mu^{N, Z}_s, \mathcal{L}\left(Z_s\right)\right)ds.
\end{align}
Applying Gr\"onwall's inequality on the above and using \eqref{wassbound} again, we find that
\begin{align}\label{1stGron}
&\mathcal{W}_1^p\left(\mu^{N, Z^N}_s, \mu^{N, Z}_s\right) \leq \frac{1}{N}\sum_{i=1}^N\left|Z_t^{i,N} - Z_t^i\right|^p \leq e^{(3p - 1)LT}\int_0^t\mathcal{W}_1^p\left(\mu^{N, Z}_s, \mathcal{L}(Z_s)\right)ds.
\end{align}
so using again the triangle inequality along with the inequality $(a + b)^p \leq 2^p(a^p + b^p)$ we obtain
\begin{align}\label{almostfinal}
\mathcal{W}_1^p\left(\mu^{N, Z^N}_s, \mathcal{L}(Z_s)\right) &\leq 2^p\mathcal{W}_1^p\left(\mu^{N, Z}_t, \mathcal{L}(Z_t)\right) + 2^p\mathcal{W}_1^p\left(\mu^{N, Z}_t, \mu^{N, Z^N}_t\right) \nonumber \\
&\leq 2^p\mathcal{W}_1^p\left(\mu^{N, Z}_t, \mathcal{L}(Z_t)\right) + 2^pe^{(3p - 1)LT}\int_0^t\mathcal{W}_1^p\left(\mu^{N, Z}_s, \mathcal{L}(Z_s)\right)ds.
\end{align}
The desired result follows by taking expectations on the above and by recalling \eqref{t-wassbound2}
\end{proof}

\begin{proof}[Proof of Lemma \ref{lem3}] 
The result will be immediate if we show that for any fixed $N$ and any choice of the function $H^N: \mathbb{R}^N \mapsto [0, 1]$ we have:
\begin{align}\label{totalvarbound}
\left|\mathbb{E}\left[H^N\left(\mathbf{Z}_t^{N,1}\right)\right] - \mathbb{E}\left[H^N\left(\mathbf{Z}_t^{N,2}\right)\right]\right| \leq \frac{1}{2}\sqrt{\int_0^T\mathbb{E}\left[ \left\Vert\mathbf{B}^{N,2}\left(s, \mathbf{Z}^{N, 2}_{s}\right) - \mathbf{B}^{N,1}\left(s, \mathbf{Z}^{N, 2}_{s}\right)\right\Vert_{2}^2\right]ds}. 
\end{align}
We will first obtain the above when the coefficient functions $\mathbf{B}^{N,k}(\cdot)$ are bounded, in which case the right-hand side is simply the relative entropy of a measure change via Girsanov's theorem that transforms the law of $\mathbf{Z}^{N,1}$ into that of $\mathbf{Z}^{N,2}$. To do that, we define a Radon-Nikodym density process $E^{N} = (E_t^{N})_{t \in [0, T]}$ as
\begin{equation}\label{RND}
E_t^{N} = \exp\left(M_t^{N} - \frac12\langle M^{N} \rangle_t\right)
\end{equation}
for $t \in [0, T]$, where
\begin{equation} \label{eq_MN}
M^{N}_t = \int_0^t \left(\mathbf{B}^{N,2}\left(s, \mathbf{Z}^{N, 1}_{s}\right) - \mathbf{B}^{N,1}\left(s, \mathbf{Z}^{N, 1}_{s}\right)\right) \cdot d\mathbf{W}^N_s.
\end{equation}
By the boundedness of $\mathbf{B}^{N,k}(\cdot)$ and Novikov's condition, the stochastic exponential $(E_t^{N})_{t \in [0,T]}$ is a martingale. By Girsanov's theorem, the process $\tilde{\mathbf{W}}^N = (\tilde{\mathbf{W}}_t^N)_{t \in [0, T]}$ with
\begin{equation}
\tilde{\mathbf{W}}_t^N = \mathbf{W}_t^N - \int_0^t\left(\mathbf{B}^{N,2}\left(s, \mathbf{Z}^{N, 1}_{s}\right) - \mathbf{B}^{N,1}\left(s, \mathbf{Z}^{N, 1}_{s}\right)\right) ds
\end{equation}
is an $N$-dimensional standard Brownian motion under the probability measure $\mathbb{Q}^N$ that satisfies $\frac{d\mathbb{Q}^N}{d\mathbb{P}}\big|_{\mathcal{F}_t} = E_t^{N}$, and we can write \eqref{sys1} as
\begin{align}\label{sys1changed}
\mathbf{Z}_t^{N, 1} = \mathbf{Z}_0^N + \int_0^t\mathbf{A}^{N}\left(s, \mathbf{Z}_s^{N,1}\right)\left(\mathbf{B}^{N,2}\left(s, \mathbf{Z}_s^{N,1}\right)ds + d\tilde{\mathbf{W}}^N_s\right) + \int_0^t\mathbf{C}^{N}\left(s, \mathbf{Z}_s^{N,1}\right)ds
\end{align}
so that $\mathbf{Z}_s^{N,1}$ is the solution to \eqref{sys2} driven by $\tilde{\mathbf{W}}^N$ instead of $\mathbf{W}$. By pathwise uniqueness and thus uniqueness in law of the solution to the SDE \eqref{sys2}, the law of $\mathbf{Z}_t^{N,1}$ under $\mathbb{Q}^N$ coincides with that of $\mathbf{Z}_t^{N,2}$ under $\mathbb{P}$. Then, the relative entropy of the measures $\mathbb{Q}^N$ and $\mathbb{P}$ is given by
\begin{align}\label{KLdivergence}
D_{KL}\left(\mathbb{Q}^N \,\,\, \Vert \,\,\, \mathbb{P}\right) &= \mathbb{E}_{\mathbb{Q}^N}\left[\log\left(\frac{d\mathbb{Q}^N}{d\mathbb{P}}\right)\right] \nonumber \\
&= \mathbb{E}_{\mathbb{Q}^N}\left[M_T^{N} - \frac12\langle M^{N}\rangle_T\right] \nonumber \\ 
&= \mathbb{E}_{\mathbb{Q}^N}\Bigg[\int_0^T \left(\mathbf{B}^{N,2}\left(s, \mathbf{Z}^{N, 1}_{s}\right) - \mathbf{B}^{N,1}\left(s, \mathbf{Z}^{N, 1}_{s}\right)\right) \cdot d\mathbf{W}^N_s  \nonumber \\
& \qquad \qquad \qquad \qquad  - \frac{1}{2}\int_0^T \left\Vert(\mathbf{B}^{N,2}\left(s, \mathbf{Z}^{N, 1}_{s}\right) - \mathbf{B}^{N,1}\left(s, \mathbf{Z}^{N, 1}_{s}\right)\right\Vert_{2}^2 ds\Bigg] \nonumber \\ 
&= \mathbb{E}_{\mathbb{Q}^N}\Bigg[\int_0^T \left(\mathbf{B}^{N,2}\left(s, \mathbf{Z}^{N, 1}_{s}\right) - \mathbf{B}^{N,1}\left(s, \mathbf{Z}^{N, 1}_{s}\right)\right) \cdot d\mathbf{\tilde{W}}^N_s  \nonumber \\
& \qquad \qquad \qquad \qquad  + \frac{1}{2}\int_0^T \left\Vert\mathbf{B}^{N,2}\left(s, \mathbf{Z}^{N, 1}_{s}\right) - \mathbf{B}^{N,1}\left(s, \mathbf{Z}^{N, 1}_{s}\right)\right\Vert_{2}^2 ds\Bigg]
\nonumber \\ 
&= \frac{1}{2}\mathbb{E}_{\mathbb{Q}^N}\left[\int_0^T \left\Vert\mathbf{B}^{N,2}\left(s, \mathbf{Z}^{N, 1}_{s}\right) - \mathbf{B}^{N,1}\left(s, \mathbf{Z}^{N, 1}_{s}\right)\right\Vert_{2}^2 ds\right]  \nonumber \\
&= \frac{1}{2}\mathbb{E}\left[\int_0^T \left\Vert\mathbf{B}^{N,2}\left(s, \mathbf{Z}^{N, 2}_{s}\right) - \mathbf{B}^{N,1}\left(s, \mathbf{Z}^{N, 2}_{s}\right)\right\Vert_{2}^2 ds\right]
\end{align}
Hence, writing $\Vert \mathbb{P} - \mathbb{Q}^N\Vert_{TV}$ for the total variation distance of the probability measures $\mathbb{P}$ and $\mathbb{Q}^N$, since $H^N$ takes values in $[0, 1]$, by Pinsker's inequality we have that
\begin{align}\label{pinsker}
\left|\mathbb{E}\left[H^N\left(\mathbf{Z}_t^{N,1}\right)\right] - \mathbb{E}\left[H^N\left(\mathbf{Z}_t^{N,2}\right)\right]\right| &= \left|\mathbb{E}\left[H^N\left(\mathbf{Z}_t^{N,1}\right)\right] - \mathbb{E}^{\mathbb{Q}^N}\left[H^N\left(\mathbf{Z}_t^{N,1}\right)\right]\right| \nonumber \\
&= \left|\int_{\Omega}H^N\left(\mathbf{Z}_t^{N,1}\right) \cdot (d\mathbb{P} - d\mathbb{Q}^N)\right| \nonumber \\
&\leq \Vert \mathbb{P} - \mathbb{Q}^N\Vert_{TV} \nonumber \\
&\leq \sqrt{\frac{1}{2}D_{KL}\left(\mathbb{Q}^N \,\,\, \Vert \,\,\, \mathbb{P}\right)} 
\end{align}
so when the functions $\mathbf{B}^{N,k}(\cdot)$ are bounded, \eqref{totalvarbound} follows by plugging \eqref{KLdivergence} into \eqref{pinsker}. When the functions $\mathbf{B}^{N,k}$ are not bounded, one can consider the sequences of stopping times $(\tau^{n,k})_{n\in \mathbb{N}}$ with $\tau^{n,k} := \inf\{t \leq T: \Vert\mathbf{Z}_t^{N,k}\Vert_2 \geq n\}$ for $k \in \{1,2\}$, and work with the stopped processes $\mathbf{Z}_t^{N,k,n} := \mathbf{Z}_{t\wedge \tau^{n,k}}^{N,k}$. Then, for every $n \in \mathbb{N}$ and $k \in \{1,2\}$, $(\mathbf{Z}_t^{N,k,n})_{t\in [0,T]}$ satisfies:
\begin{align}\label{sysstopped}
\mathbf{Z}_t^{N, k,n} = \mathbf{Z}_0^N + \int_0^t\mathbf{A}^{N,n}\left(s, \mathbf{Z}_s^{N,k,n}\right)\left(\mathbf{B}^{N,k,n}\left(s, \mathbf{Z}_s^{N,k,n}\right)ds + d\mathbf{W}^N_s\right) + \int_0^t\mathbf{C}^{N,n}\left(s, \mathbf{Z}_s^{N,k,n}\right)ds,
\end{align}
with the coefficients $\mathbf{A}^{N,n}(s,\mathbf{z}) := \mathbf{A}^{N}(s,\mathbf{z})\mathbf{1}_{\{\Vert\mathbf{z}\Vert_2 < n\}}$ and $\mathbf{C}^{N,n}(s,\mathbf{z}) := \mathbf{C}^{N}(s,\mathbf{z})\mathbf{1}_{\{\Vert\mathbf{z}\Vert_2 < n\}}$ which are not necessarily bounded, and the coefficients $\mathbf{B}^{N,k,n}(s,\mathbf{z}) := \mathbf{B}^{N,k}(s,\mathbf{z})\mathbf{1}_{\{\Vert\mathbf{z}\Vert_2 < n\}}$ which are bounded due to the local boundedness of the functions $\mathbf{B}^{N,k}(\cdot)$. Thus, we can obtain \eqref{totalvarbound} for the stopped processes $\mathbf{Z}_t^{N,k,n}$:
\begin{align}\label{totalvarbound2}
&\left|\mathbb{E}\left[H^N\left(\mathbf{Z}_t^{N,1,n}\right)\right] - \mathbb{E}\left[H^N\left(\mathbf{Z}_t^{N,2,n}\right)\right]\right| \nonumber \\
&\qquad \leq \frac{1}{2}\sqrt{\int_0^T\mathbb{E}\left[ \left\Vert\mathbf{B}^{N,2,n}\left(s, \mathbf{Z}^{N,2,n}_{s}\right) - \mathbf{B}^{N,1,n}\left(s, \mathbf{Z}^{N, 2, n}_{s}\right)\right\Vert_{2}^2\right]ds} 
\end{align}
and \eqref{totalvarbound} for the non-stopped processes $\mathbf{Z}_t^{N,k}$ occurs by taking $n \to \infty$ in the above. To take the limit on \eqref{totalvarbound2} we use the dominated convergence theorem on the left-hand side, and for the right-hand side we observe that
\begin{align*}
&\int_0^T\mathbb{E}\left[ \left\Vert\mathbf{B}^{N,2}\left(s, \mathbf{Z}^{N, 2}_{s}\right) - \mathbf{B}^{N,1}\left(s, \mathbf{Z}^{N, 2}_{s}\right)\right\Vert_{2}^2\right]ds \nonumber \\
& \qquad \qquad \qquad- \int_0^T\mathbb{E}\left[ \left\Vert\mathbf{B}^{N,2,n}\left(s, \mathbf{Z}^{N,2,n}_{s}\right) - \mathbf{B}^{N,1,n}\left(s, \mathbf{Z}^{N, 2, n}_{s}\right)\right\Vert_{2}^2\right]ds\nonumber \\
& \qquad = \mathbb{E}\left[\int_{\tau^{n,2}}^T \left\Vert\mathbf{B}^{N,2}\left(s, \mathbf{Z}^{N,2}_{s}\right) - \mathbf{B}^{N,1}\left(s, \mathbf{Z}^{N, 2}_{s}\right)\right\Vert_{2}^2ds\right]
\end{align*}
which vanishes as $n \to \infty$ by the monotone convergence theorem, since it is clearly finite by \eqref{momentcond1} and since $\tau^{2,n} \uparrow T$ almost surely as $n \to \infty$ due to $\mathbf{Z}^{N, 2}$ having continuous paths on $[0, T]$.
\end{proof}
\noindent We are now ready to prove our second main result.
\begin{proof}[Proof of Theorem~\ref{mainresult}]
We show first that the function in the drift terms of \eqref{approxsys2}, i.e
\begin{align}\label{driftfunctionapprox}
b^*(t, x, \tilde{m}) &:= \hat{b}\left(x, \mathcal{L}(X_s^{v^{*}}), U_{x}\left(s, x, \mathcal{L}(X_s^{v^{*}})\right)\right) \nonumber \\
&\qquad \qquad+ \int_{\mathbb{R}}\frac{\delta}{\delta m}\left\{\hat{b}\left(x, m, U_{x}\left(s, x, m\right)\right)\right\}(z_1) \Big|_{m = \mathcal{L}(X_s^{v^{*}})} (\tilde{m} - \mathcal{L}(X_s^{v^{*}}))(dz_1),
\end{align}
has bounded derivatives in both $x$ and $\tilde{m}$. For the boundedness of $b^*_x(t,x,\tilde{m})$, it suffices to show that both summands on the right-hand side of \eqref{driftfunctionapprox} have a bounded derivative in $x$, where for the first summand we have  
\begin{align}
\left|\frac{\partial}{\partial x}\hat{b}\left(x, \mathcal{L}(X_s^{v^{*}}), U_{x}\left(s, x, \mathcal{L}(X_s^{v^{*}})\right)\right)\right| \leq \left\Vert \hat{b}_x\right\Vert_{\infty} + \left\Vert \hat{b}_y\right\Vert_{\infty} \left\Vert U_{xx}\right\Vert_{\infty}
\end{align}
whose right-hand side is finite by Assumption~\ref{ass0}, and because the total mass of both $\tilde{m}$ and $\mathcal{L}(X_s^{v^{*}})$ on $\mathbb{R}$ is equal to $1$, the boundedness of the derivative in $x$ of the second summand is derived from:
\begin{align}\label{neededtwice}
&\left|\frac{\partial }{\partial x}\left\{\frac{\delta}{\delta m}\left\{\hat{b}\left(x, m, U_{x}\left(t, x, m\right)\right)\right\}(z_1)\right\}\right| \nonumber \\
&\qquad \qquad = \Bigg|\frac{\delta \hat{b}_x}{\delta m}\left(x, m, U_{x}\left(t, x, m\right), z_1\right) + \frac{\delta \hat{b}_y}{\delta m}\left(x, m, U_{x}\left(t, x, m\right), z_1\right)U_{xx}\left(t, x, m\right) \nonumber \\
&\qquad \qquad \qquad \qquad +\left(\hat{b}_{y x}\left(x, m, U_{x}\left(t, x, m\right)\right) +\hat{b}_{y y}\left(x, m, U_{x}\left(t, x, m\right)\right)U_{xx}\left(t, x, m\right)\right)\frac{\delta U_x}{\delta m}\left(t, x, m, z_1\right) \nonumber \\
&\qquad \qquad \qquad \qquad + \hat{b}_{y}\left(x, m, U_{x}\left(t, x, m\right)\right)\frac{\delta U_{x x}}{\delta m}\left(t, x, m, z_1\right)\Bigg| \nonumber \\
&\qquad \qquad \leq \left\Vert \left(\frac{\delta \hat{b}}{\delta m}\right)_x \right\Vert_{\infty} + \left\Vert \left(\frac{\delta \hat{b}}{\delta m}\right)_y \right\Vert_{\infty} \cdot \left\Vert U_{xx} \right\Vert_{\infty}\nonumber \\
&\qquad \qquad \quad\qquad \qquad + \left(\left\Vert \hat{b}_{yx} \right\Vert_{\infty} + \left\Vert\hat{b}_{yy} \right\Vert_{\infty} \cdot \left\Vert U_{xx} \right\Vert_{\infty}\right)\cdot \left\Vert \frac{\delta U_{x}}{\delta m} \right\Vert_{\infty} + \left\Vert\hat{b}_{y} \right\Vert_{\infty} \cdot\left\Vert \left(\frac{\delta U_{x}}{\delta m} \right)_x\right\Vert_{\infty}
\end{align}
where the right-hand side is also finite by Assumptions~\ref{ass0} and \ref{ass0-}. Note that for the previous argument we do not need to put the derivative inside the integral, as we can simply approximate it by a difference quotient that can be pushed inside the integral and use then the mean value theorem to bound it by the right hand side of \eqref{neededtwice}. Finally, to bound $|b^*_m(t,x,\tilde{m}, z_1)|$, we observe that it is equal to
\begin{align}\label{neededtwice2}
&\left|\frac{\partial }{\partial z_1}\left\{\frac{\delta}{\delta m}\left\{\hat{b}\left(x, m, U_{x}\left(t, x, m\right)\right)\right\}(z_1)\right\}\right| \nonumber \\
&\qquad \qquad = \left|\hat{b}_m\left(x, m, U_{x}\left(t, x, m\right), z_1\right) + \hat{b}_{y}\left(x, m, U_{x}\left(t, x, m\right)\right)U_{x m}\left(t, x, m, z_1\right)\right| \nonumber \\
&\qquad \qquad \leq \left\Vert \hat{b}_m \right\Vert_{\infty} + \left\Vert \hat{b}_y \right\Vert_{\infty} \cdot \left\Vert U_{xm} \right\Vert_{\infty}
\end{align}
with the right-hand side being finite by Assumption~\ref{ass0}. 

From the boundedness of $b^*_x(t,x,\tilde{m})$ and $b^*_m(t,x,\tilde{m}, z_1)$ which we established above, we deduce that \eqref{approxsys2} admits a unique strong solution and that it satisfies the conditions of Lemma~\ref{lem2}. Indeed, the first is due to the boundedness of $$(b^*(t,x_i,\mu_{\mathbf{x}}^N))_{x_j} = b^*_x(t,x_i,\mu_{\mathbf{x}}^N)\delta_{ij} + \frac{1}{N}b^*_m(t,x_i,\mu_{\mathbf{x}}^N, x_j)$$
for any $i, j \in \{1, 2, \ldots, N\}$, which implies that \eqref{approxsys2} has Lipschitz coefficients in $x_1, x_2, \ldots, x_N$, and the second follows by recalling also (i) of Assumption~\ref{ass0-}. Moreover, the particle system \eqref{approxsys2} is of the form \cite[Equation (1.2)]{KLZ23}, with $A$ being a constant function, $B$ being affine with constant slope in its third argument, $C = 0$ and the function $g$ given by
\begin{equation*}
g\left(t, \tilde{X}_{[0, t]}^{i,N}, \mathbf{y}_{[0, t]}\right) = \frac{\delta}{\delta m}\left\{\hat{b}\left(\tilde{X}_{t}^{i,N}, m, U_{x}\left(t, \tilde{X}_{t}^{i,N}, m\right)\right)\right\}(\mathbf{y}_t)\Big|_{m = \mathcal{L}(X_t^{v^{*}})}
\end{equation*}
so that \cite[Assumption 2.1]{KLZ23} is satisfied automatically.

Before verifying \cite[Assumption 2.2]{KLZ23}, we write the corresponding McKean-Vlasov SDE, i.e. the counterpart of \cite[Equation 1.4]{KLZ23}: 
\begin{align}\label{approxsysmckean}
&\tilde{X}_t^{v^*} = X_0 + \int_0^t\hat{b}\left(\tilde{X}_s^{v^*}, \mathcal{L}(X_s^{v^{*}}), U_{x}\left(s, \tilde{X}_s^{v^*}, \mathcal{L}(X_s^{v^{*}})\right)\right)ds \nonumber \\
&\qquad \qquad \quad + \int_0^t\int_{\mathbb{R}}\frac{\delta}{\delta m}\left\{\hat{b}\left(\tilde{X}_s^{v^*}, m, U_{x}\left(s, \tilde{X}_s^{v^*}, m\right)\right)\right\}(z_1) \Big|_{m = \mathcal{L}(X_s^{v^{*}})} \nonumber \\
&\qquad \qquad \qquad \qquad \qquad \qquad \qquad \qquad \qquad \qquad \qquad \times (\mathcal{L}(\tilde{X}_s^{v^{*}}) - \mathcal{L}(X_s^{v^{*}}))(dz_1)ds + \sigma W_t,
\end{align}
and we claim that this reduces to \eqref{repagent} since $\mathcal{L}(\tilde{X}_t^{v^{*}}) \equiv \mathcal{L}(X_t^{v^{*}})$. Indeed, subtracting \eqref{repagent} from the above we obtain
\begin{align}\label{mckeandiff}
&\tilde{X}_t^{v^*} - X_t^{v^*} = \int_0^t\Big\{\hat{b}\left(\tilde{X}_s^{v^*}, \mathcal{L}(X_s^{v^{*}}), U_{x}\left(s, \tilde{X}_s^{v^*}, \mathcal{L}(X_s^{v^{*}})\right)\right) \nonumber \\
&\qquad \qquad \qquad \qquad \qquad \qquad \qquad - \hat{b}\left(X_s^{v^*}, \mathcal{L}(X_s^{v^{*}}), U_{x}\left(s, X_s^{v^*}, \mathcal{L}(X_s^{v^{*}})\right)\right)\Big\}ds \nonumber \\
&\qquad \qquad \qquad \quad + \int_0^t\int_{\mathbb{R}}\frac{\delta}{\delta m}\left\{\hat{b}\left(\tilde{X}_s^{v^*}, m, U_{x}\left(s, \tilde{X}_s^{v^*}, m\right)\right)\right\}(z_1) \Big|_{m = \mathcal{L}(X_s^{v^{*}})} \nonumber \\
&\qquad \qquad \qquad \qquad \qquad \qquad \qquad \qquad \qquad \qquad \qquad \times (\mathcal{L}(\tilde{X}_s^{v^{*}}) - \mathcal{L}(X_s^{v^{*}}))(dz_1)ds
\end{align}
which implies that
\begin{align}\label{mckeandiffabs}
\left|\tilde{X}_t^{v^*} - X_t^{v^*}\right| &\leq \int_0^t\left(\left\Vert \hat{b}_x\right\Vert_{\infty} + \left\Vert \hat{b}_y\right\Vert_{\infty} \left\Vert U_{xx}\right\Vert_{\infty}\right)\left|\tilde{X}_s^{v^*} - X_s^{v^*}\right|ds \nonumber \\
&\qquad + \int_0^t\left(\left\Vert\hat{b}_{m}\right\Vert_{\infty} + \left\Vert\hat{b}_{y}\right\Vert_{\infty}\left\Vert U_{xm}\right\Vert_{\infty}\right)\mathcal{W}_1\left(\mathcal{L}(\tilde{X}_s^{v^{*}}), \mathcal{L}(X_s^{v^{*}})\right)ds \nonumber \\
&\leq C\int_0^t\left\{\left|\tilde{X}_s^{v^*} - X_s^{v^*}\right| + \mathbb{E}\left[\left|\tilde{X}_s^{v^*} - X_s^{v^*}\right|\right]\right\}ds
\end{align}
so our claim follows by taking expectations and using Gr\"onwall's inequality.

To verify now \cite[Assumption 2.2]{KLZ23}, we write $\mu_t^{N, v^*}$ for the empirical measure of the i.i.d. variables $X_t^{i, v^{*}}$ and $\mu_t^{v^*}$ for their common law, that is
\begin{equation*}
\mu_t^{N, v^*} = \frac{1}{N}\sum_{\ell = 1}^N\delta_{X_t^{\ell, v^{*}}} \quad \text{and} \quad \mu_t^{v^*} = \mathcal{L}(X_t^{v^{*}}) \quad \text{for} \quad t \in [0, T].
\end{equation*}
Then, since the right-hand sides of \eqref{neededtwice} and \eqref{neededtwice2} are finite by Assumptions~\ref{ass0} and \ref{ass0-}, we obtain
\begin{equation*}
\left|\frac{\delta}{\delta m}\left\{\hat{b}\left(X_t^{i, v^{*}}, m, U_{x}\left(t, X_t^{i, v^{*}}, m\right)\right)\right\}(X_t^{j, v^{*}})\Big|_{m = \mathcal{L}(X_t^{v^{*}})}\right| \leq C\left(1 + \left|X_t^{i, v^{*}}\right| + \left|X_t^{j, v^{*}}\right|\right)
\end{equation*}
and also
\begin{equation*}
\left|\int_{\mathbb{R}}\frac{\delta}{\delta m}\left\{\hat{b}\left(X_t^{i, v^{*}}, m, U_{x}\left(t, X_t^{i, v^{*}}, m\right)\right)\right\}(\mathbf{y}_t)\Big|_{m = \mathcal{L}(X_t^{v^{*}})}(\mu_t^{N, v^*} - \mu_t^{v^*})(d\mathbf{y}_t)\right| \leq C\mathcal{W}_1\left(\mu_t^{N, v^*}, \mu_t^{v^*}\right)
\end{equation*}
whose $2p$-moments are bounded by $C^p\Gamma(p)$ and $\frac{C^p}{N^p}\Gamma(p)$ respectively. The latter is obtained by applying Lemma~\ref{lem2} on the system \eqref{approxsys2} which satisfies the required assumptions as we have shown earlier, and it completes the verification of \cite[Assumption 2.2]{KLZ23}. 

By the above, we can apply \cite[Theorem 2.3]{KLZ23} on the system \eqref{approxsys2} with the corresponding McKean-Vlasov SDE being \eqref{repagent}, so that \eqref{evtcond} implies the weak convergence of the point process
\begin{align}
U^{N}\{\tilde{X}_t^{1,N}, \tilde{X}_t^{2,N}, \ldots, \tilde{X}_t^{N,N}\}\left(I\right) := \sum_{i=1}^N \mathbf{1}_{\left\{\left(\frac{i}{N}, \frac{\tilde{X}_t^{i,N} - b_t^N}{a_t^N} \right)\in I\right\}} \quad \text{for} \quad I = [a, b] \times [c, d]
\end{align}
to a Poisson random measure on $(0,1]\times(x_*,x^*]$ with intensity measure $\nu$ given by
\[
\nu((a,b]\times(c,d]) = (b - a) \left((1 + \gamma_t c)^{-\frac{1}{\gamma_t}} - (1 + \gamma_t d)^{-\frac{1}{\gamma_t}}\right) \quad \text{for} \quad a<b \quad \text{and} \quad c<d,
\]
where $x_*,x^*$ are the left and right endpoints of the support of $G_{\gamma_t}$. By repeating the arguments from \cite[Pages 13-14]{KLZ23} which establish equiconvergence for two sequences of point processes via tightness and characterization of their limits on rectangles, we can deduce the same convergence for the point process of the particles $\bar{X}_t^{i,N}$ that satisfy the system \eqref{approxsys1}, that is,
\begin{align}
U^{N}\{\bar{X}_t^{1,N}, \bar{X}_t^{2,N}, \ldots, \bar{X}_t^{N,N}\}\left(I\right) := \sum_{i=1}^N \mathbf{1}_{\left\{\left(\frac{i}{N}, \frac{\bar{X}_t^{i,N} - b_t^N}{a_t^N} \right)\in I\right\}} \quad \text{for} \quad I = [a, b] \times [c, d],
\end{align}
provided that we can establish
\begin{equation}\label{1stlimiteq}
\lim_{N \to \infty}\mathbb{P}\left(U^{N}\{\bar{X}_t^{1,N}, \bar{X}_t^{2,N}, \ldots, \bar{X}_t^{N,N}\}\left(I\right) = n\right) = \lim_{N \to \infty}\mathbb{P}\left(U^{N}\{\tilde{X}_t^{1,N}, \tilde{X}_t^{2,N}, \ldots, \tilde{X}_t^{N,N}\}\left(I\right) = n\right)
\end{equation}
for any rectangle $I$ and $n \in \mathbb{N}$. Then, we can deduce in the same way that the same weak limit is admitted by the point process of the Nash system \eqref{systemNash}, i.e.
\begin{align}\label{nashpp}
&U^{N}\{X_t^{1,N, v^{N,*}}, X_t^{2,N, v^{N,*}}, \ldots, X_t^{N,N, v^{N,*}}\}\left(I\right) \nonumber \\
&\qquad \qquad \qquad \qquad\qquad \qquad \qquad \qquad := \sum_{i=1}^N \mathbf{1}_{\left\{\left(\frac{i}{N}, \frac{X_t^{i,N, v^{N,*}} - b_t^N}{a_t^N} \right)\in I\right\}} \quad \text{for} \quad I = [a, b] \times [c, d],
\end{align}
provided that we can show that
\begin{align}\label{2ndlimiteq}
&\lim_{N \to \infty}\mathbb{P}\left(U^{N}\{X_t^{1,N, v^{N,*}}, X_t^{2,N, v^{N,*}}, \ldots, X_t^{N,N, v^{N,*}}\}\left(I\right) = n\right) \nonumber \\
&\qquad\qquad\qquad\qquad\qquad \qquad \qquad \qquad  = \lim_{N \to \infty}\mathbb{P}\left(U^{N}\{\bar{X}_t^{1,N}, \bar{X}_t^{2,N}, \ldots, \bar{X}_t^{N,N}\}\left(I\right) = n\right)
\end{align}
for any rectangle $I$ and $n \in \mathbb{N}$. Having all the above, the desired result follows by expressing the joint law of the left hand side of \eqref{evtresult} in terms of probabilities that involve the point process \eqref{nashpp}, which is done by repeating the arguments from \cite[Sections 3.1 and 3.3]{KLZ23} with the uncontrolled interacting processes $X^{i,N}$ therein replaced by the Nash states $X^{i,N,v^{N,*}}$ from our setup. Hence, we only need to establish \eqref{1stlimiteq} and \eqref{2ndlimiteq}; the second is obtained by applying Lemma~\ref{lem3} on the systems \eqref{systemNash} and \eqref{approxsys1} with
\begin{equation}\label{HN}
H^N\left(z_1, z_2, \ldots, z_N\right) = \mathbf{1}_{\{U^N\{z_1, z_2, \ldots, z_N\}(I) = n\}}
\end{equation}
for any rectangle $I$ and $n \in \mathbb{N}$, where the condition \eqref{momentcond1} is automatically satisfied by the estimate \ref{expestimate} (given by \cite[Theorem 4.1]{DLR20}, whose assumptions are included in Assumption~\ref{ass0}); on the other hand, \eqref{1stlimiteq} is obtained by applying Lemma~\ref{lem3} on the systems \eqref{approxsys1} and \eqref{approxsys2}, with $H^N$ also given by \eqref{HN}, but this time condition \eqref{momentcond1} needs some extra work to be verified. For the later, we use that the quantity we are bounding consists of squared error terms of a Taylor expansion:
\begin{align}
&\Bigg\{\hat{b}\left(\bar{X}_t^{i,N}, \bar{\mu}_t^{N}, U_{x}\left(t, \bar{X}_t^{i,N}, \bar{\mu}_t^{N}\right) \right) - \hat{b}\left(\bar{X}_t^{i,N}, \mathcal{L}(X_t^{v^{*}}), U_{x}\left(t, \bar{X}_t^{i,N}, \mathcal{L}(X_t^{v^{*}})\right)\right) \nonumber \\
&\qquad\qquad - \int_{\mathbb{R}}\frac{\delta}{\delta m}\left\{\hat{b}\left(\bar{X}_t^{i,N}, m, U_{x}\left(t, \bar{X}_t^{i,N}, m\right)\right)\right\}(z_1) \Big|_{m = \mathcal{L}(X_t^{v^{*}})}(\bar{\mu}_t^N - \mathcal{L}(X_t^{v^{*}}))(dz_1)\Bigg\}^2 \nonumber \\
&\qquad = \Bigg\{\int_0^1\int_0^1\int_{\mathbb{R}}\int_{\mathbb{R}}\frac{\delta^2}{\delta m^2}\left\{\hat{b}\left(\bar{X}_t^{i,N}, m, U_{x}\left(t, \bar{X}_t^{i,N}, m\right)\right)\right\}(z_1, z_2) \Big|_{\substack{m = \lambda'\left(\lambda\bar{\mu}_t^N + (1 - \lambda)\mathcal{L}(X_t^{v^{*}})\right) \\ + (1 - \lambda')\mathcal{L}(X_t^{v^{*}})}} \nonumber \\
&\qquad\qquad\qquad\qquad\qquad\qquad\qquad\qquad\qquad \times \lambda(\bar{\mu}_t^N - \mathcal{L}(X_t^{v^{*}}))(dz_1)(\bar{\mu}_t^N - \mathcal{L}(X_t^{v^{*}}))(dz_2)d\lambda d\lambda' \Bigg\}^2 \nonumber \\
&\qquad \leq \frac{1}{4}\Bigg\{\left\Vert\hat{b}_{m m}\right\Vert_{\infty} + \left(\left\Vert\hat{b}_{y m}\right\Vert_{\infty} + \left\Vert\left(\frac{\delta\hat{b}}{\delta m}\right)_{yz_1}\right\Vert_{\infty}\right)\left\Vert U_{x m}\right\Vert_{\infty} \nonumber \\
&\qquad\qquad\qquad\qquad\qquad\qquad\qquad\,\,\,\, + \left\Vert\hat{b}_{y y}\right\Vert_{\infty}\left\Vert U_{x m}\right\Vert_{\infty}^2 + \left\Vert\hat{b}_{y}\right\Vert_{\infty}\left\Vert U_{x m m}\right\Vert_{\infty}\Bigg\}^2\times\mathcal{W}_1^4\left(\bar{\mu}_t^N, \mathcal{L}(X_t^{v^{*}})\right)
\end{align}
so \eqref{momentcond1} for the systems \eqref{approxsys1} and \eqref{approxsys2} is obtained by taking the sum of the above bound for $i \in \{1, 2, \ldots, N\}$, taking expectations, and finally recalling again Assumptions~\ref{ass0} and \ref{ass0-} along with the third estimate of Lemma~\ref{lem2} for the system \eqref{approxsys1}.
\end{proof}

\section{Discussion of games with state-dependent volatility}
We conclude this paper with a more analytic discussion of Remark~\ref{ncvol}. Consider the following variation of \eqref{system} where $\sigma$ is a function of the state:
\begin{align}\label{systemncvol}
X_t^{i, N, v^N} = X_0^i + \int_0^t b\left(X_s^{i, N, v^N}, \mu_s^{N, v^N}, v_s^{i, N} \right)ds + \sigma\left(X_s^{i, N, v^N}\right) W_t^i, \qquad t \geq 0.
\end{align}
Assuming that the function $\sigma$ takes values in some compact subinterval of $\mathbb{R}_{+}$ and has bounded derivatives up to order $3$, it is easy to verify that the scaling $X_t^{i, N, v^N} \mapsto S(X_t^{i, N, v^N})$ for
\begin{equation}
S(x) = \int_0^x\frac{dz}{\sigma(z)}
\end{equation}
reduces the Stochastic Differential Game to the form that has already been treated, with the control processes remaining unchanged and the coefficient functions replaced as follows:
\begin{equation*}
\begin{aligned}
b(x, \mu, v) &\mapsto S'\left(S^{-1}(x)\right)b\left(S^{-1}(x), S^{-1}(\mu), v\right) + \frac{1}{2}S''\left(S^{-1}(x)\right)\sigma^2\left(S^{-1}(x)\right) \\
&= \frac{1}{\sigma\left(S^{-1}(x)\right)}b\left(S^{-1}(x), S^{-1}(\mu), v\right) - \frac{1}{2}\sigma'\left(S^{-1}(x)\right) \\
\sigma(x) &\mapsto S'\left(S^{-1}(x)\right)\sigma\left(S^{-1}(x)\right) \equiv 1 \\
f(x, \mu, v) &\mapsto f\left(S^{-1}(x), S^{-1}(\mu), v\right) \\
g(x, \mu) &\mapsto g\left(S^{-1}(x), S^{-1}(\mu)\right)
\end{aligned}
\end{equation*}
where for any measure $\mu \in \mathcal{P}(\mathbb{R})$ we write $S^{-1}(\mu)$ for the measure that maps any interval $I = (a, b)$ to $\mu(S(I))) = \mu(S(a), S(b))$.
Then, it can be seen from the definition of the master field in \cite[Pages 327-329]{CD2} that the function $U$ is replaced as follows:
\begin{equation*}
U(t, x, \mu) \mapsto U(t, S^{-1}(x), S^{-1}(\mu)).
\end{equation*}
It is now easy to verify that our transformation preserves Assumptions~\ref{ass0} and \ref{ass0-}, where for example we have
\begin{equation*}
\frac{\delta }{\delta m}\left\{q\left(S^{-1}(\mu)\right)\right\}(z_1) = \frac{\delta q}{\delta m}\left(S^{-1}(\mu), S^{-1}(z_1)\right)\sigma(z_1)
\end{equation*}
whenever $q$ is a differentiable function of probability measures on $\mathbb{R}$. Now, to extend our result to the system \eqref{systemncvol}, it suffices to show that our transformation preserves the condition \eqref{evtcond}, while the inverse transformation preserves the resulting convergence \eqref{evtresult}. For the first, we pick normalizing sequences $(\tilde{a}_t^N)_{N \in \mathbb{N}}$ and $(\hat{b}_t^N)_{N \in \mathbb{N}}$ with $\hat{b}_t^N = S(b_t^N)$ and $\tilde{a}_t^N = a_t^NS'(b_t^N)$, so by the mean value theorem and the fact that $S(x)$ is strictly increasing, for some sequence of random variables $(\xi_N)_{N \in \mathbb{N}}$ we have
\begin{equation*}
\max_{i \leq N}\frac{S\left(X_t^{i, v^{*}}\right) - \hat{b}_t^N}{\tilde{a}_t^N} = \max_{i \leq N}\frac{X_t^{i, v^{*}} - b_t^N}{a_t^N} + \frac{1}{2}\frac{S''\left(\xi_N\right)}{S'\left(b_t^N\right)}a_t^N\left(\max_{i \leq N}\frac{X_t^{i, v^{*}} - b_t^N}{a_t^N}\right)^2
\end{equation*}
where $a_t^N \to 0$ as $N \to \infty$ when $\gamma_t \leq 0$ (see \cite[Corollary 1.2.4]{HF10}) and $S'(x) = \frac{1}{\sigma(x)}$ is lower bounded and has a bounded derivative. Hence, the weak convergence of the first term on the right-hand side to $G_{\gamma_t}$ implies that the second term on the right-hand side vanishes as $N \to \infty$, in which case the left-hand side converges also to $G_{\gamma_t}$ weakly. To verify that the inverse transformation preserves \eqref{evtresult}, for the $j$-th upper order statistic of the Nash system we use the mean value theorem and the fact that $S^{-1}$ is increasing to write 
\begin{align*}
\frac{X_t^{i_j^N(t), N, v^{N,*}} - b_t^N}{a_t^N} &= \frac{S^{-1}\left(S\left(X_t^{i_j^N(t), N, v^{N, *}}\right)\right) - S^{-1}\left(\hat{b}_t^N\right)}{a_t^N} \nonumber \\
&= \frac{S\left(X_t^{i_j^N(t), N, v^{N, *}}\right) - \hat{b}_t^N}{\tilde{a}_t^N} \nonumber \\
&\qquad + \frac{1}{2}\left(S^{-1}\right)''\left(\tilde{\xi}_N\right)\left(S'\left(b_t^N\right)\right)^2a_t^N\left(\frac{S\left(X_t^{i_j^N(t), N, v^{N, *}}\right) - \hat{b}_t^N}{\tilde{a}_t^N}\right)^2
\end{align*}
for some sequence of random variables $(\tilde{\xi}_N)_{N \in \mathbb{N}}$, where we used that $(S^{-1})'(\hat{b}_t^N) = \frac{1}{S'(b_t^N)} = \frac{a_t^N}{\tilde{a}_t^N}$. Again, since $S'(x) = \frac{1}{\sigma(x)}$ and $(S^{-1})''(x) = \sigma'(S^{-1}(x))\sigma(S^{-1}(x))$ are both bounded, the weak convergence of the first term on the right-hand side to $G_{\gamma_t}$ implies that the second term on the right-hand side vanishes as $N \to \infty$, so the same convergence holds for the left-hand side as well. 

The above computations show that the upper boundedness of $\sigma(x)$ could be relaxed to some slow growth as $x \to \infty$, which would be a more natural assumption that covers a much wider class of games, but this would require both the master field $U$ and the components $\hat{b}$ and $\hat{f}$ of the optimized Hamiltonian to have derivatives that vanish sufficiently fast as their arguments tend to infinity (to ensure a sufficiently fast decay of $a_t^N$ and slow growth of $b_t^N$ and $\tilde{\xi}_N$ as $N \to \infty$).

\section{Acknowledgments} We would like to thank Dan Lacker and an anonymous referee for providing us with valuable comments for improving this paper. 

\bibliography{references}

@article{RBM0,
title = {Random Batch Methods (RBM) for interacting particle systems},
journal = {Journal of Computational Physics},
volume = {400},
pages = {108877},
year = {2020},
issn = {0021-9991},
doi = {https://doi.org/10.1016/j.jcp.2019.108877},
url = {https://www.sciencedirect.com/science/article/pii/S0021999119305741},
author = {Shi Jin and Lei Li and Jian-Guo Liu}
}

@article{cox1985theory,
  author = {Cox, John C. and Ingersoll, Jonathan E. and Ross, Stephen A.},
  title = {A Theory of the Term Structure of Interest Rates},
  journal = {Econometrica},
  year = {1985},
  volume = {53},
  number = {2},
  pages = {385--407},
  doi = {10.2307/1911242}
}

@book{bookqlpde,
author = {Ladyzhenskai︠a︡, O. A. and Solonnikov, V. A. and Uralʹt︠s︡eva, N. N. and Smith, S.},
address = {Providence},
booktitle = {Linear and quasi-linear equations of parabolic type},
isbn = {9780821815731},
language = {eng},
lccn = {68019440},
publisher = {American Mathematical Society},
series = {Translations of mathematical monographs, v. 23},
title = {Linear and quasi-linear equations of parabolic type },
year = {1968},
}

@article{book2syspde,
title = "Nonlinear Elliptic Systems in Stochastic Game Theory",
author = "A. Bensoussan",
year = "1984",
doi = "10.1515/crll.1984.350.23",
language = "English",
volume = "1984",
pages = "23--67",
journal = "Journal fur die Reine und Angewandte Mathematik",
issn = "0075-4102",
publisher = "Walter de Gruyter GmbH",
number = "350",
}

@article{RBM1,
author = {Ha, Seung-Yeal and Jin, Shi and Kim, Doheon and Ko, Dongnam},
title = {Uniform-in-time error estimate of the random batch method for the Cucker–Smale model},
journal = {Mathematical Models and Methods in Applied Sciences},
volume = {31},
number = {06},
pages = {1099-1135},
year = {2021},
doi = {10.1142/S0218202521400029},
URL = {https://doi.org/10.1142/S0218202521400029},
eprint = {https://doi.org/10.1142/S0218202521400029}
}

@article{RBM2,
author = {Jin, Shi and Li, Lei and Sun, Yiqun},
title = {On the Random Batch Method for Second Order Interacting Particle Systems},
journal = {Multiscale Modeling \& Simulation},
volume = {20},
number = {2},
pages = {741-768},
year = {2022},
doi = {10.1137/20M1383069},
URL = {https://doi.org/10.1137/20M1383069},
eprint = {https://doi.org/10.1137/20M1383069}}

@article{RBM3,
	author = {{Huang, Zhenyu} and {Jin, Shi} and {Li, Lei}},
	title = {Mean field error estimate of the random batch method for large interacting particle system},
	DOI= "10.1051/m2an/2024071",
	url= "https://doi.org/10.1051/m2an/2024071",
	journal = {ESAIM: M2AN},
	year = 2025,
	volume = 59,
	number = 1,
	pages = "265-289",
}

@article{RBM4,
author = {Jin, Shi and Li, Lei and Xu, Zhenli and Zhao, Yue},
title = {A Random Batch Ewald Method for Particle Systems with Coulomb Interactions},
journal = {SIAM Journal on Scientific Computing},
volume = {43},
number = {4},
pages = {B937-B960},
year = {2021},
doi = {10.1137/20M1371385},
URL = {https://doi.org/10.1137/20M1371385},
eprint = {https://doi.org/10.1137/20M1371385}
}

@article{RBM5,
author = {Li, Lei and Xu, Zhenli and Zhao, Yue},
title = {A Random-Batch Monte Carlo Method for Many-Body Systems with Singular Kernels},
journal = {SIAM Journal on Scientific Computing},
volume = {42},
number = {3},
pages = {A1486-A1509},
year = {2020},
doi = {10.1137/19M1302077},
URL = {https://doi.org/10.1137/19M1302077},
eprint = {https://doi.org/10.1137/19M1302077}
}

@article {KLZ22,
     title={Propagation of chaos for maxima of particle systems with mean-field drift interaction},
  author={Kolliopoulos, N. and Larsson, M. and Zhang, Z.},
  journal={Probability Theory and Related Fields 1--35},
  year={2023},
  publisher={Springer Berlin Heidelberg Berlin/Heidelberg}
}

@article{KLZ23,
  title={Propagation of chaos for point processes induced by particle systems with mean-field drift interaction},
  author={Kolliopoulos, Nikolaos and Larsson, Martin and Zhang, Zeyu},
  journal={Journal of Theoretical Probability},
  volume={38},
  number={1},
  pages={1--22},
  year={2025},
  publisher={Springer US}
}

@article{Delarue2018FromTM,
  title={From the master equation to mean field game limit theory: a central limit theorem},
  author={François Delarue and Daniel Lacker and Kavita Ramanan},
  journal={Electronic Journal of Probability},
  year={2018},
  url={https://api.semanticscholar.org/CorpusID:119680950}
}

@manual{HDS15,
title = {Graduate lecture notes in High - Dimensional Statistics, Chapter 1}, 
author = {Philippe Rigollet},
organization = {MIT OpenCourseWare, Massachusetts Institute of Technology},
address = "\url{ https://ocw.mit.edu/courses/18-s997-high-dimensional-statistics-spring-2015/resources/mit18\_s997s15\_chapter1/}",
year = {2015}
}

@article{LA23,
  title={Hierarchies, entropy, and quantitative propagation of chaos for mean field diffusions},
  author={Daniel Lacker},
  journal={Probability and Mathematical Physics},
  year={2021},
  url={https://api.semanticscholar.org/CorpusID:234095995}
}

@book{HF10,
  title={Extreme Value Theory: An Introduction},
  author={de Haan, L. and Ferreira, A.},
  isbn={9780387344713},
  lccn={2006925909},
  series={Springer Series in Operations Research and Financial Engineering},
  url={https://books.google.com/books?id=t6tfXnykazEC},
  year={2007},
  publisher={Springer New York}
}

@article {PI75,
    AUTHOR = {Pickands, III, James},
     TITLE = {Statistical inference using extreme order statistics},
   JOURNAL = {Ann. Statist.},
  FJOURNAL = {The Annals of Statistics},
    VOLUME = {3},
      YEAR = {1975},
     PAGES = {119--131},
      ISSN = {0090-5364,2168-8966},
   MRCLASS = {62G05},
  MRNUMBER = {423667},
MRREVIEWER = {H.\ L.\ Seal},
       URL =
              {http://links.jstor.org/sici?sici=0090-5364(197501)3:1<119:SIUEOS>2.0.CO;2-O&origin=MSN},
}

@article {HI75,
    AUTHOR = {Hill, Bruce M.},
     TITLE = {A simple general approach to inference about the tail of a
              distribution},
   JOURNAL = {Ann. Statist.},
  FJOURNAL = {The Annals of Statistics},
    VOLUME = {3},
      YEAR = {1975},
    NUMBER = {5},
     PAGES = {1163--1174},
      ISSN = {0090-5364,2168-8966},
   MRCLASS = {62F10},
  MRNUMBER = {378204},
MRREVIEWER = {Anthony\ O'Hagan},
       URL =
              {http://links.jstor.org/sici?sici=0090-5364(197509)3:5<1163:ASGATI>2.0.CO;2-Y&origin=MSN},
}

@article {WE78,
    AUTHOR = {Weissman, Ishay},
     TITLE = {Estimation of parameters and large quantiles based on the
              {$k$} largest observations},
   JOURNAL = {J. Amer. Statist. Assoc.},
  FJOURNAL = {Journal of the American Statistical Association},
    VOLUME = {73},
      YEAR = {1978},
    NUMBER = {364},
     PAGES = {812--815},
      ISSN = {0162-1459,1537-274X},
   MRCLASS = {62F12 (62G30)},
  MRNUMBER = {521329},
MRREVIEWER = {J.\ Tiago de Oliveira},
       URL =
              {http://links.jstor.org/sici?sici=0162-1459(197812)73:364<812:EOPALQ>2.0.CO;2-M&origin=MSN},
}

@article {DED89,
    AUTHOR = {Dekkers, A. L. M. and Einmahl, J. H. J. and de Haan, L.},
     TITLE = {A moment estimator for the index of an extreme-value
              distribution},
   JOURNAL = {Ann. Statist.},
  FJOURNAL = {The Annals of Statistics},
    VOLUME = {17},
      YEAR = {1989},
    NUMBER = {4},
     PAGES = {1833--1855},
      ISSN = {0090-5364,2168-8966},
   MRCLASS = {62E20 (62G30)},
  MRNUMBER = {1026315},
MRREVIEWER = {S\'{a}ndor\ Cs\"{o}rg\H{o}},
       DOI = {10.1214/aos/1176347397},
       URL = {https://doi.org/10.1214/aos/1176347397},
}

@article {DR298,
    AUTHOR = {Drees, Holger},
     TITLE = {Optimal rates of convergence for estimates of the extreme
              value index},
   JOURNAL = {Ann. Statist.},
  FJOURNAL = {The Annals of Statistics},
    VOLUME = {26},
      YEAR = {1998},
    NUMBER = {1},
     PAGES = {434--448},
      ISSN = {0090-5364,2168-8966},
   MRCLASS = {62G05 (62G20)},
  MRNUMBER = {1608148},
MRREVIEWER = {L\'{a}szl\'{o}\ Viharos},
       DOI = {10.1214/aos/1030563992},
       URL = {https://doi.org/10.1214/aos/1030563992},
}

@article {CP01,
    AUTHOR = {Cheng, Shihong and Peng, Liang},
     TITLE = {Confidence intervals for the tail index},
   JOURNAL = {Bernoulli},
  FJOURNAL = {Bernoulli. Official Journal of the Bernoulli Society for
              Mathematical Statistics and Probability},
    VOLUME = {7},
      YEAR = {2001},
    NUMBER = {5},
     PAGES = {751--760},
      ISSN = {1350-7265,1573-9759},
   MRCLASS = {62G05 (62G15)},
  MRNUMBER = {1867084},
MRREVIEWER = {Miguel\ A.\ Arcones},
       DOI = {10.2307/3318540},
       URL = {https://doi.org/10.2307/3318540},
}

@article {GM02,
    AUTHOR = {Gomes, M. Ivette and Martins, M. Jo\~{a}o},
     TITLE = {``{A}symptotically unbiased'' estimators of the tail index
              based on external estimation of the second order parameter},
   JOURNAL = {Extremes},
  FJOURNAL = {Extremes. Statistical Theory and Applications in Science,
              Engineering and Economics},
    VOLUME = {5},
      YEAR = {2002},
    NUMBER = {1},
     PAGES = {5--31},
      ISSN = {1386-1999,1572-915X},
   MRCLASS = {62G32 (62F10)},
  MRNUMBER = {1947785},
       DOI = {10.1023/A:1020925908039},
       URL = {https://doi.org/10.1023/A:1020925908039},
}

@article {GDP02,
    AUTHOR = {Gomes, M. Ivette and de Haan, Laurens and Peng, Liang},
     TITLE = {Semi-parametric estimation of the second order parameter in
              statistics of extremes},
   JOURNAL = {Extremes},
  FJOURNAL = {Extremes. Statistical Theory and Applications in Science,
              Engineering and Economics},
    VOLUME = {5},
      YEAR = {2002},
    NUMBER = {4},
     PAGES = {387--414},
      ISSN = {1386-1999,1572-915X},
   MRCLASS = {62G32 (62G05 62G20)},
  MRNUMBER = {2002125},
       DOI = {10.1023/A:1025128326588},
       URL = {https://doi.org/10.1023/A:1025128326588},
}

@article {GDR08,
    AUTHOR = {Gomes, M. Ivette and de Haan, Laurens and Rodrigues,
              L\'{\i}gia Henriques},
     TITLE = {Tail index estimation for heavy-tailed models: accommodation
              of bias in weighted log-excesses},
   JOURNAL = {J. R. Stat. Soc. Ser. B Stat. Methodol.},
  FJOURNAL = {Journal of the Royal Statistical Society. Series B.
              Statistical Methodology},
    VOLUME = {70},
      YEAR = {2008},
    NUMBER = {1},
     PAGES = {31--52},
      ISSN = {1369-7412,1467-9868},
   MRCLASS = {99-01},
  MRNUMBER = {2412630},
}

@article {GCF04,
    AUTHOR = {Gomes, M. Ivette and Caeiro, Frederico and Figueiredo,
              Fernanda},
     TITLE = {Bias reduction of a tail index estimator through an external
              estimation of the second-order parameter},
   JOURNAL = {Statistics},
  FJOURNAL = {Statistics. A Journal of Theoretical and Applied Statistics},
    VOLUME = {38},
      YEAR = {2004},
    NUMBER = {6},
     PAGES = {497--510},
      ISSN = {0233-1888,1029-4910},
   MRCLASS = {62G05 (62E20 62F10 62F35)},
  MRNUMBER = {2109631},
       DOI = {10.1080/02331880412331284304},
       URL = {https://doi.org/10.1080/02331880412331284304},
}

@article {CG06,
    AUTHOR = {Caeiro, Frederico and Gomes, M. Ivette},
     TITLE = {A new class of estimators of a ``scale'' second order
              parameter},
   JOURNAL = {Extremes},
  FJOURNAL = {Extremes. Statistical Theory and Applications in Science,
              Engineering and Economics},
    VOLUME = {9},
      YEAR = {2006},
    NUMBER = {3-4},
     PAGES = {193--211},
      ISSN = {1386-1999,1572-915X},
   MRCLASS = {99-01},
  MRNUMBER = {2367836},
       DOI = {10.1007/s10687-006-0026-7},
       URL = {https://doi.org/10.1007/s10687-006-0026-7},
}

@article {CH09,
    AUTHOR = {Zhou, Chen},
     TITLE = {Existence and consistency of the maximum likelihood estimator
              for the extreme value index},
   JOURNAL = {J. Multivariate Anal.},
  FJOURNAL = {Journal of Multivariate Analysis},
    VOLUME = {100},
      YEAR = {2009},
    NUMBER = {4},
     PAGES = {794--815},
      ISSN = {0047-259X,1095-7243},
   MRCLASS = {62F10 (60G70 62F12 62G20)},
  MRNUMBER = {2478199},
MRREVIEWER = {Sreenivasan\ Ravi},
       DOI = {10.1016/j.jmva.2008.08.009},
       URL = {https://doi.org/10.1016/j.jmva.2008.08.009},
}

@article {DGR12,
    AUTHOR = {de Wet, Tertius and Goegebeur, Yuri and Reimert Munch, Maria},
     TITLE = {Asymptotically unbiased estimation of the second order tail
              parameter},
   JOURNAL = {Statist. Probab. Lett.},
  FJOURNAL = {Statistics \& Probability Letters},
    VOLUME = {82},
      YEAR = {2012},
    NUMBER = {3},
     PAGES = {565--573},
      ISSN = {0167-7152,1879-2103},
   MRCLASS = {62G32 (62G05 62G20)},
  MRNUMBER = {2887473},
       DOI = {10.1016/j.spl.2011.11.016},
       URL = {https://doi.org/10.1016/j.spl.2011.11.016},
}

@article {CGBD16,
    AUTHOR = {Caeiro, Frederico and Gomes, M. Ivette and Beirlant, Jan and
              de Wet, Tertius},
     TITLE = {Mean-of-order {$p$} reduced-bias extreme value index
              estimation under a third-order framework},
   JOURNAL = {Extremes},
  FJOURNAL = {Extremes. Statistical Theory and Applications in Science,
              Engineering and Economics},
    VOLUME = {19},
      YEAR = {2016},
    NUMBER = {4},
     PAGES = {561--589},
      ISSN = {1386-1999,1572-915X},
   MRCLASS = {62G32 (62G20 65C05)},
  MRNUMBER = {3558346},
MRREVIEWER = {Michael\ Falk},
       DOI = {10.1007/s10687-016-0261-5},
       URL = {https://doi.org/10.1007/s10687-016-0261-5},
}

@article {PVV17,
    AUTHOR = {Paulauskas, Vygantas and Vai\v{c}iulis, Marijus},
     TITLE = {A class of new tail index estimators},
   JOURNAL = {Ann. Inst. Statist. Math.},
  FJOURNAL = {Annals of the Institute of Statistical Mathematics},
    VOLUME = {69},
      YEAR = {2017},
    NUMBER = {2},
     PAGES = {461--487},
      ISSN = {0020-3157,1572-9052},
   MRCLASS = {62G32 (60F05 62F12)},
  MRNUMBER = {3611528},
       DOI = {10.1007/s10463-015-0548-3},
       URL = {https://doi.org/10.1007/s10463-015-0548-3},
}

@article {DF19,
    AUTHOR = {Dombry, Cl\'{e}ment and Ferreira, Ana},
     TITLE = {Maximum likelihood estimators based on the block maxima
              method},
   JOURNAL = {Bernoulli},
  FJOURNAL = {Bernoulli. Official Journal of the Bernoulli Society for
              Mathematical Statistics and Probability},
    VOLUME = {25},
      YEAR = {2019},
    NUMBER = {3},
     PAGES = {1690--1723},
      ISSN = {1350-7265,1573-9759},
   MRCLASS = {62G32 (62E20)},
  MRNUMBER = {3961227},
MRREVIEWER = {Emanuele\ Taufer},
       DOI = {10.3150/18-BEJ1032},
       URL = {https://doi.org/10.3150/18-BEJ1032},
}

@article {BBD19,
    AUTHOR = {Buitendag, Sven and Beirlant, Jan and de Wet, Tertius},
     TITLE = {Ridge regression estimators for the extreme value index},
   JOURNAL = {Extremes},
  FJOURNAL = {Extremes. Statistical Theory and Applications in Science,
              Engineering and Economics},
    VOLUME = {22},
      YEAR = {2019},
    NUMBER = {2},
     PAGES = {271--292},
      ISSN = {1386-1999,1572-915X},
   MRCLASS = {60G70 (62J07)},
  MRNUMBER = {3949046},
       DOI = {10.1007/s10687-018-0338-4},
       URL = {https://doi.org/10.1007/s10687-018-0338-4},
}

@article {AR20,
    AUTHOR = {Akhtyamov, P. I. and Rodionov, I. V.},
     TITLE = {On an estimator for the location and scale parameters of
              distribution tails},
   JOURNAL = {Fundam. Prikl. Mat.},
  FJOURNAL = {Fundamental\cprime naya i Prikladnaya Matematika},
    VOLUME = {23},
      YEAR = {2020},
    NUMBER = {1},
     PAGES = {25--49},
      ISSN = {1560-5159,2076-6203},
   MRCLASS = {62F10 (62F12 62G32)},
  MRNUMBER = {4152983},
}

@article {HPN22,
    AUTHOR = {Hu, Shuang and Peng, Zuoxiang and Nadarajah, Saralees},
     TITLE = {Location invariant heavy tail index estimation with block
              method},
   JOURNAL = {Statistics},
  FJOURNAL = {Statistics. A Journal of Theoretical and Applied Statistics},
    VOLUME = {56},
      YEAR = {2022},
    NUMBER = {3},
     PAGES = {479--497},
      ISSN = {0233-1888,1029-4910},
   MRCLASS = {60G70 (62G32)},
  MRNUMBER = {4446617},
       DOI = {10.1080/02331888.2022.2071898},
       URL = {https://doi.org/10.1080/02331888.2022.2071898},
}

@article {GM23,
    AUTHOR = {Gardes, Laurent and Maistre, Samuel},
     TITLE = {Nonparametric asymptotic confidence intervals for extreme
              quantiles},
   JOURNAL = {Scand. J. Stat.},
  FJOURNAL = {Scandinavian Journal of Statistics. Theory and Applications},
    VOLUME = {50},
      YEAR = {2023},
    NUMBER = {2},
     PAGES = {825--841},
      ISSN = {0303-6898,1467-9469},
   MRCLASS = {62G15 (62G20 62G30)},
  MRNUMBER = {4599934},
MRREVIEWER = {L.\ D\"{u}mbgen},
}

@article {AEG23,
    AUTHOR = {Allouche, Micha\"{e}l and El Methni, Jonathan and Girard,
              St\'{e}phane},
     TITLE = {A refined {W}eissman estimator for extreme quantiles},
   JOURNAL = {Extremes},
  FJOURNAL = {Extremes. Statistical Theory and Applications in Science,
              Engineering and Economics},
    VOLUME = {26},
      YEAR = {2023},
    NUMBER = {3},
     PAGES = {545--572},
      ISSN = {1386-1999,1572-915X},
   MRCLASS = {62G32 (60G70 62G20)},
  MRNUMBER = {4627354},
       DOI = {10.1007/s10687-022-00452-8},
       URL = {https://doi.org/10.1007/s10687-022-00452-8},
}

@book {CD1,
    AUTHOR = {Carmona, Ren\'e{} and Delarue, Fran\c cois},
     TITLE = {Probabilistic theory of mean field games with applications.
              {I}},
    SERIES = {Probability Theory and Stochastic Modelling},
    VOLUME = {83},
      NOTE = {Mean field FBSDEs, control, and games},
 PUBLISHER = {Springer, Cham},
      YEAR = {2018},
     PAGES = {xxv+713},
      ISBN = {978-3-319-56437-1; 978-3-319-58920-6},
   MRCLASS = {60-02 (35R60 49N70 49N90 60H15 60H30 91A15 93E20)},
  MRNUMBER = {3752669},
MRREVIEWER = {Vassili\ N.\ Kolokol\cprime tsov},
}

@book {CD2,
    AUTHOR = {Carmona, Ren\'e{} and Delarue, Fran\c cois},
     TITLE = {Probabilistic theory of mean field games with applications.
              {II}},
    SERIES = {Probability Theory and Stochastic Modelling},
    VOLUME = {84},
      NOTE = {Mean field games with common noise and master equations},
 PUBLISHER = {Springer, Cham},
      YEAR = {2018},
     PAGES = {xxiv+697},
      ISBN = {978-3-319-56435-7; 978-3-319-56436-4},
   MRCLASS = {60-02 (35R60 49L20 60G55 60H10 60H30 91A13 91A15)},
  MRNUMBER = {3753660},
MRREVIEWER = {Vassili\ N.\ Kolokol\cprime tsov},
}

@book {CDLL19,
    AUTHOR = {Cardaliaguet, Pierre and Delarue, Fran\c cois and Lasry,
              Jean-Michel and Lions, Pierre-Louis},
     TITLE = {The master equation and the convergence problem in mean field
              games},
    SERIES = {Annals of Mathematics Studies},
    VOLUME = {201},
 PUBLISHER = {Princeton University Press, Princeton, NJ},
      YEAR = {2019},
     PAGES = {x+212},
      ISBN = {978-0-691-19071-6; 978-0-691-19070-9},
   MRCLASS = {49-02 (49N80 60H30 60K35 91A07 91A15 91A16)},
  MRNUMBER = {3967062},
MRREVIEWER = {Ren\'e\ Carmona},
       DOI = {10.2307/j.ctvckq7qf},
       URL = {https://doi.org/10.2307/j.ctvckq7qf},
}

@article {DLR20,
    AUTHOR = {Delarue, Fran\c cois and Lacker, Daniel and Ramanan, Kavita},
     TITLE = {From the master equation to mean field game limit theory:
              large deviations and concentration of measure},
   JOURNAL = {Ann. Probab.},
  FJOURNAL = {The Annals of Probability},
    VOLUME = {48},
      YEAR = {2020},
    NUMBER = {1},
     PAGES = {211--263},
      ISSN = {0091-1798,2168-894X},
   MRCLASS = {60F10 (60E15 60H10 60K35 91A15 91A16 91G80)},
  MRNUMBER = {4079435},
       DOI = {10.1214/19-AOP1359},
       URL = {https://doi.org/10.1214/19-AOP1359},
}

@article {FOGU15,
    AUTHOR = {Fournier, Nicolas and Guillin, Arnaud},
     TITLE = {On the rate of convergence in {W}asserstein distance of the
              empirical measure},
   JOURNAL = {Probab. Theory Related Fields},
  FJOURNAL = {Probability Theory and Related Fields},
    VOLUME = {162},
      YEAR = {2015},
    NUMBER = {3-4},
     PAGES = {707--738},
      ISSN = {0178-8051,1432-2064},
   MRCLASS = {60F25 (60E15 60F10)},
  MRNUMBER = {3383341},
MRREVIEWER = {Jos\'e\ Trashorras},
       DOI = {10.1007/s00440-014-0583-7},
       URL = {https://doi.org/10.1007/s00440-014-0583-7},
}

@article{MR968996,
	author = {G\"{a}rtner, J\"{u}rgen},
	doi = {10.1002/mana.19881370116},
	fjournal = {Mathematische Nachrichten},
	issn = {0025-584X},
	journal = {Math. Nachr.},
	mrclass = {60K35 (60J60)},
	mrnumber = {968996},
	mrreviewer = {Anton Wakolbinger},
	pages = {197--248},
	title = {On the {M}c{K}ean-{V}lasov limit for interacting diffusions},
	url = {https://doi.org/10.1002/mana.19881370116},
	volume = {137},
	year = {1988}}

@incollection{MR1108185,
	author = {Sznitman, Alain-Sol},
	booktitle = {\'{E}cole d'\'{E}t\'{e} de {P}robabilit\'{e}s de {S}aint-{F}lour {XIX}---1989},
	doi = {10.1007/BFb0085169},
	mrclass = {60J60 (60K35 82C40)},
	mrnumber = {1108185},
	mrreviewer = {Maria E. Vares},
	pages = {165--251},
	publisher = {Springer, Berlin},
	series = {Lecture Notes in Math.},
	title = {Topics in propagation of chaos},
	url = {https://doi.org/10.1007/BFb0085169},
	volume = {1464},
	year = {1991}}

@article {CFS15,
    AUTHOR = {Carmona, Ren\'{e} and Fouque, Jean-Pierre and Sun, Li-Hsien},
     TITLE = {Mean field games and systemic risk},
   JOURNAL = {Commun. Math. Sci.},
  FJOURNAL = {Communications in Mathematical Sciences},
    VOLUME = {13},
      YEAR = {2015},
    NUMBER = {4},
     PAGES = {911--933},
      ISSN = {1539-6746},
   MRCLASS = {91A80 (60H30 91B64 93E20)},
  MRNUMBER = {3325083},
MRREVIEWER = {George Stoica},
       DOI = {10.4310/CMS.2015.v13.n4.a4},
       URL = {https://doi.org/10.4310/CMS.2015.v13.n4.a4},
}

@article{Carmona2020ApplicationsOM,
  title={Applications of Mean Field Games in financial engineering and economic theory},
  author={Ren{\'e} A. Carmona},
  journal={Proceedings of Symposia in Applied Mathematics},
  year={2020},
  url={https://api.semanticscholar.org/CorpusID:228083777}
}

@article{FU23,
    AUTHOR = {Fu, Guanxing and Zhou, Chao},
     TITLE = {Mean field portfolio games},
   JOURNAL = {Finance Stoch.},
  FJOURNAL = {Finance and Stochastics},
    VOLUME = {27},
      YEAR = {2023},
    NUMBER = {1},
     PAGES = {189--231},
      ISSN = {0949-2984,1432-1122},
   MRCLASS = {91G10 (91A16 91A80)},
  MRNUMBER = {4524380},
       DOI = {10.1007/s00780-022-00492-9},
       URL = {https://doi.org/10.1007/s00780-022-00492-9},
}

@incollection {HUZA,
    AUTHOR = {Hu, Ruimeng and Zariphopoulou, Thaleia},
     TITLE = {{$N$}-player and mean-field games in {I}t\^o-diffusion markets
              with competitive or homophilous interaction},
 BOOKTITLE = {Stochastic analysis, filtering, and stochastic optimization},
     PAGES = {209--237},
 PUBLISHER = {Springer, Cham},
      YEAR = {[2022] \copyright 2022},
      ISBN = {978-3-030-98518-9; 978-3-030-98519-6},
   MRCLASS = {91A16 (60H30)},
  MRNUMBER = {4433815},
}

@article {LAZA,
    AUTHOR = {Lacker, D. and Zariphopoulou, T.},
     TITLE = {Mean field and {$n$}-agent games for optimal investment under
              relative performance criteria},
   JOURNAL = {Math. Finance},
  FJOURNAL = {Mathematical Finance. An International Journal of Mathematics,
              Statistics and Financial Economics},
    VOLUME = {29},
      YEAR = {2019},
    NUMBER = {4},
     PAGES = {1003--1038},
      ISSN = {0960-1627},
   MRCLASS = {91A15 (91A07 91G10)},
  MRNUMBER = {4014625},
       DOI = {10.1111/mafi.12206},
       URL = {https://doi.org/10.1111/mafi.12206},
}

@article {LASO20,
    AUTHOR = {Lacker, Daniel and Soret, Agathe},
     TITLE = {Many-player games of optimal consumption and investment under
              relative performance criteria},
   JOURNAL = {Math. Financ. Econ.},
  FJOURNAL = {Mathematics and Financial Economics},
    VOLUME = {14},
      YEAR = {2020},
    NUMBER = {2},
     PAGES = {263--281},
      ISSN = {1862-9679,1862-9660},
   MRCLASS = {91G10 (91A16)},
  MRNUMBER = {4076185},
MRREVIEWER = {Gui\ Mei\ Luo},
       DOI = {10.1007/s11579-019-00255-9},
       URL = {https://doi.org/10.1007/s11579-019-00255-9},
}

@article {ZA24,
    AUTHOR = {Zariphopoulou, Thaleia},
     TITLE = {Mean field and {$n$}-player games in {I}to-diffusion markets
              under forward performance criteria},
   JOURNAL = {Probab. Uncertain. Quant. Risk},
  FJOURNAL = {Probability, Uncertainty and Quantitative Risk},
    VOLUME = {9},
      YEAR = {2024},
    NUMBER = {2},
     PAGES = {123--148},
      ISSN = {2095-9672,2367-0126},
   MRCLASS = {91A16 (60H30 91G10 93E20)},
  MRNUMBER = {4767977},
       DOI = {10.3934/puqr.2024008},
       URL = {https://doi.org/10.3934/puqr.2024008},
}

@article{MR221595,
	author = {McKean, Jr., H. P.},
	doi = {10.1073/pnas.56.6.1907},
	fjournal = {Proceedings of the National Academy of Sciences of the United States of America},
	issn = {0027-8424},
	journal = {Proc. Nat. Acad. Sci. U.S.A.},
	mrclass = {60.62},
	mrnumber = {221595},
	mrreviewer = {F. B. Knight},
	pages = {1907--1911},
	title = {A class of {M}arkov processes associated with nonlinear parabolic equations},
	url = {https://doi.org/10.1073/pnas.56.6.1907},
	volume = {56},
	year = {1966}}
\bibliographystyle{abbrv}

\end{document}